\documentclass{article}[11pt]
\usepackage{amssymb}
\usepackage{amscd}
\usepackage{amsthm,amsmath}
\usepackage{amsfonts}
\usepackage{latexsym}
\usepackage{graphicx}
\usepackage[all,tips, 2cell]{xy}
\UseTwocells
\usepackage{verbatim}
\usepackage[margin=1.1in]{geometry}

\DeclareFontFamily{U}{mathc}{}
\DeclareFontShape{U}{mathc}{m}{it}%
{<->s*[1.03] mathc10}{}
\DeclareMathAlphabet{\mathcal}{U}{mathc}{m}{it}

\usepackage{hyperref}
\newtheorem{theorem}{Theorem}[section]
\newtheorem{lemma}[theorem]{Lemma}
\newtheorem{proposition}[theorem]{Proposition}
\newtheorem{corollary}[theorem]{Corollary} 
\theoremstyle{definition}  
\newtheorem{definition}[theorem]{Definition}
\newtheorem{example}[theorem]{Example}
\newtheorem{conjecture}[theorem]{Conjecture}  

\newtheorem{remark}[theorem]{Remark}

\numberwithin{equation}{section}

\newcommand{\id}{\text{Id}}

\newcommand{\nl}{\newline}

\newcommand{\Pic}{{{Pic}}}
\newcommand{\BrPic}{{{BrPic}}}
\newcommand{\Picbr}{Pic_{br}}
\newcommand{\uHom}{\underline{\text{Hom}}} 

\newcommand{\Aut}{Aut}

\DeclareMathOperator{\Coker}{\operatorname{\mathsf{Coker}}}

\newcommand{\Inv}{Inv}

\DeclareMathOperator{\Hom}{\operatorname{\mathsf{Hom}}}
\DeclareMathOperator{\Quad}{\operatorname{\mathsf{Quad}}}
\DeclareMathOperator{\Ext}{\operatorname{\mathsf{Ext}}}
\DeclareMathOperator{\Extpower}{\mathsf{\Lambda}}
\DeclareMathOperator{\Ker}{\operatorname{\mathsf{Ker}}}

\newcommand{\uuEx}{\mathbf{Ex}}

\newcommand{\uuExctr}{\mathbf{Ex_{ctr}}}

\newcommand{\uuExcrbr}{\mathbf{Ex_{cr-br}}}

\newcommand{\uuExbr}{\mathbf{Ex_{br}}}

\newcommand{\uuExbrqt}{\mathbf{Ex_{br-qt}}}

\newcommand{\uuExsym}{\mathbf{Ex_{sym}}}

\newcommand{\uPic}{\mathcal{P\mkern-3mu ic}} 
\newcommand{\uuPic}{\mathbf{Pic}}
\newcommand{\uBrPic}{\mathcal{\B\mkern-3mu rPic}} 
\newcommand{\uuBrPic}{\mathbf{BrPic}}
\newcommand{\uPicbr}{\mathcal{{P\mkern-3mu ic}_{br}}} 
\newcommand{\uuPicbr}{\mathbf{{Pic}_{br}}}
\newcommand{\PicbrI}{{Pic}^1_{br}}
 
\newcommand{\uuPicbrI}{\mathbf{{Pic}^1_{br}}}
\newcommand{\Picsym}{Pic_{sym}}
\newcommand{\uPicsym}{\mathcal{{P\mkern-3mu ic}_{sym}}} 
\newcommand{\uuPicsym}{\mathbf{{Pic}_{sym}}}
\newcommand{\uAut}{\mathcal{A\mkern-3mu ut}} 

\newcommand{\uInv}{\mathcal{I\mkern-3mu nv}} 
\newcommand{\uuInv}{\mathbf{Inv}}

\newcommand{\rev}{\text{rev}}

\newcommand{\eps}{\varepsilon}

\newcommand{\B}{\mathcal{B}}
\newcommand{\C}{\mathcal{C}}
\newcommand{\I}{\mathcal{I}}
\newcommand{\R}{\mathcal{R}}
\newcommand{\D}{\mathcal{D}}
\newcommand{\E}{\mathcal{E}}
\newcommand{\F}{\mathcal{F}}
\newcommand{\Z}{\mathcal{Z}}
\renewcommand{\H}{\mathcal{H}}
\renewcommand{\L}{\mathcal{L}}
\newcommand{\M}{\mathcal{M}}
\newcommand{\A}{\mathcal{A}}

\newcommand{\N}{\mathcal{N}}
\newcommand{\K}{\mathcal{K}}

\newcommand{\be}{\mathbf{1}}
\renewcommand{\P}{\mathcal{P}}
\renewcommand{\S}{\mathcal{S}}
\newcommand{\Q}{\mathcal{Q}}
\newcommand{\G}{\mathcal{G}}

\renewcommand{\be}{\mathbf{1}}

\newcommand{\bt}{\boxtimes}
\newcommand{\ot}{\otimes}
\newcommand{\op}{\oplus}

\newcommand{\bZ}{{\mathbb{Z}}}

\newcommand{\beq}{\begin{equation}}
\newcommand{\eeq}{\end{equation}}
\newcommand{\ve}{{\varepsilon}}
\newcommand{\lb}{\label}
\newcommand{\bpf}{\begin{proof}}
\newcommand{\epf}{\end{proof}}

\newcommand{\bth}{\begin{theorem}}
\renewcommand{\eth}{\end{theorem}}
\newcommand{\bpr}{\begin{proposition}}
\newcommand{\epr}{\end{proposition}}
\newcommand{\ble}{\begin{lemma}}
\newcommand{\ele}{\end{lemma}}
\newcommand{\bco}{\begin{corollary}}
\newcommand{\eco}{\end{corollary}}
\newcommand{\bde}{\begin{definition}}
\newcommand{\ede}{\end{definition}}
\newcommand{\bex}{\begin{example}}
\newcommand{\eex}{\end{example}}
\newcommand{\bre}{\begin{remark}}
\newcommand{\ere}{\end{remark}}
\newcommand{\bcj}{\begin{conjecture}}
\newcommand{\ecj}{\end{conjecture}}

\newcommand{\bM}{\mathbf{M}}

\newcommand{\Ex}{{\bf Ex}}
\newcommand{\Mod}{\mathbf{Mod}}
\newcommand{\Modsym}{\mathbf{Mod_{sym}}}
\newcommand{\Modbr}{\mathbf{Mod_{br}}}
\newcommand{\ModbrO}{\mathbf{Mod_{br}^0}}
\newcommand{\ModbrI}{\mathbf{Mod_{br}^1}}
\newcommand{\Bimod}{\mathbf{Bimod}}
\newcommand{\Fun}{\mathcal{F\mkern-3mu un}}

\newcommand{\End}{\mathcal{E \mkern-3mu nd}}

\newcommand{\opp}{\text{op}}
\newcommand{\Vect}{\mathcal{V \mkern-3mu ect}} 
\newcommand{\sVect}{\mathcal{sV \mkern-3mu ect}} 
\newcommand{\Rep}{\mathcal{R\mkern-3mu ep}}

\newcommand\void[1]{}

\newcommand{\INV}{\mathbf{Inv}}
\mathchardef\mhyphen="2D
\newcommand{\FUN}{\mathbf{2\mhyphen}\mathbf{Fun}}
\newcommand{\FUNbr}{\mathbf{2\mhyphen}\mathbf{Fun}_{\mathbf{br}}}
\newcommand{\FUNsym}{\mathbf{2\mhyphen}\mathbf{Fun}_{\mathbf{sym}}}
\newcommand{\AUT}{\mathbf{Aut}}
\newcommand{\AD}{\mathbf{Ad}}

\newcommand{\bG}{\mathbf{G}}
\newcommand{\bN}{\mathbf{N}}
\newcommand{\bA}{\mathbf{A}}

\newcommand{\X}{{\mathcal X}}

\begin{document}

\author{Alexei Davydov $^{a)}$, Dmitri Nikshych $^{b)}$}
\title{Braided Picard groups \\ and  graded extensions of braided tensor categories}

\maketitle

\begin{center}
$^{a)}$\ Department of Mathematics, Ohio University, Athens, OH 45701, USA\\
$^{b)}$\ Department of Mathematics and Statistics, University of New Hampshire,  Durham, NH 03824, USA
\end{center}
\date{}
\begin{abstract}
We classify various types of graded extensions of  a finite braided tensor category $\B$  in terms of 
its $2$-categorical Picard groups. In particular, we prove that braided extensions of $\B$ by a finite group $A$ correspond
to braided monoidal $2$-functors from $A$ to the  braided $2$-categorical Picard group of $\B$
(consisting of invertible central $\B$-module categories).  Such functors can be expressed in terms of the Eilnberg-Mac~Lane cohomology.   
We describe in detail braided $2$-categorical Picard groups of symmetric fusion categories 
and of pointed braided fusion categories.
\end{abstract}
\tableofcontents


\section{Introduction and synopsis}

\subsection{Extensions of tensor categories}

In this paper we work over an algebraically closed field $k$. All tensor categories are assumed to be $k$-linear and finite \cite{DGNO}.

Let $\B$ be a tensor category. An {\em extension} of $\B$ is an embedding of $\B$ into a tensor category $\C$
i.e. a fully faithful tensor functor  $\iota: \B \to \C$. We will  identify $\B$ with its image in $\C$ and use notation $\B \subset \C$ to
denote an extension.  An {\em isomorphism} between extensions  $\iota_1,\, \iota_2: \B \to \C$ is a tensor autoequivalence
$F: \C \to \C$ such that $F\circ\iota_1= \iota_2$. 

When $\B$ is braided (or symmetric) there are several types of extensions reflecting different ``amounts of commutativity" of $\C$.
Namely, we say that an extension $\iota :\B\to \C$ is
\begin{itemize}
\item {\em central} if there is a  lifting tensor functor $L: \B \to \Z(\C)$  such that $\iota$ coincides with
the composition $\B \xrightarrow{L} \Z(\C) \xrightarrow{\text{Forget}} \C$, where $\Z(\C)$ is the center of $\C$ and 
$\text{Forget}:\Z(\C)\to \C$ is the forgetful functor;
\item {\em braided} if $\C$ is braided;
\item {\em symmetric} if $\C$ is symmetric.
\end{itemize}

The extending tensor category $\C$ can be viewed as a $\B$-module category.  Furthermore,
the tensor product of $\C$ equips it with a structure of a pseudo-monoid \cite{DS} in a  monoidal $2$-category $\mathbf{M}$ consisting of certain
$\B$-modules.  Commutativity properties of $\B,\, \C$ and the type of an extension $\B\subset \C$ are reflected in the  
choice of $\B$-modules in $\mathbf{M}$ and in the properties of the pseudo-monoid $\C$. These properties are summarized in Table~\ref{table-1}.

\medskip
\begin{table}[h!]
\centering
\begin{tabular}{ |p{3cm}||p{7cm}|p{4cm}|  }
 \hline
 \hline
Extension $\B\subset \C$ & $2$-category $\mathbf{M}$ of $\B$-modules & $\C$\\
 \hline
tensor    & monoidal $2$-category $\Bimod(\B)$    &  pseudo-monoid  \\
central &  monoidal $2$-category $\Mod(\B)$  &  pseudo-monoid  \\
braided  & braided monoidal $2$-category $\Mod_{br}(\B)$ &  braided pseudo-monoid  \\
symmetric &  symmetric monoidal $2$-category $\Mod_{sym}(\B)$ & symmetric pseudo-monoid \\
 \hline
\end{tabular}
\caption{\label{table-1}Extensions $\B\subset \C$  as pseudo-monoids in a monoidal $2$-category $\mathbf{M}$.}
\end{table}
\medskip

In this paper, for a tensor category $\D$ we denote by $\Bimod(\D)$  the monoidal $2$-category of $\D$-bimodule categories \cite{ENO}.
For a braided tensor category $\B$
we denote by $\Mod(\B)$  the monoidal $2$-category of $\B$-module categories \cite{DN, ENO} (it can be viewed as
a monoidal $2$-subcategory of $\Bimod(\B)$) and $\Mod_{br}(\B)$ the braided monoidal $2$-category
of {\em braided} $\B$-module categories. For a symmetric tensor category $\E$  we denote $by \Mod_{sym}(\E)$ the symmetric monoidal $2$-category
of {\em symmetric} $\E$-module categories\footnote{For fusion categories, these 2-categories of module categories are fusion $2$-categories \cite{DR}.}. 
By definition \cite{E,B,BBJ},  a braided $\B$-module category is
equipped with a natural collection of isomorphisms coherently extending the braiding on $\B$, see Definition~\ref{pb definition}. 
Braided module categories  turn out to play an important role in 4d topological field theory and factorization homology. 
In Theorem~\ref{Modbr=ZMod} we show that the 2-category 
$\Mod_{br}(\B)$ of braided module categories is braided $2$-equivalent to the center (in the sense of \cite{BN}) of $\Mod(\B)$:
\begin{equation} 
\Mod_{br}(\B) \cong
 \mathbf{Z}(\Mod(\B)).
 \end{equation}

\subsection{Graded extensions  and monoidal $2$-functors to $2$-categorical groups}

We focus on extensions of finite braided tensor categories graded by finite groups.

Let $G$ be a group.  A {\em $G$-extension} of a tensor category $\D$  is an extension $\D \subset \C$ together with  a faithful 
$G$-grading of $\C$ such that $\D$ is the trivial component. In other words, $\C$ admits a decomposition
\begin{equation}
\label{G-extension def}
\C =\bigoplus_{g\in G}\, \C_g
\end{equation}
such that $\C_1=\D$ and the tensor multiplication of $\C$ maps $\C_x\times \C_y$ to $\C_{xy}$ for all $x,y\in G$.  
An {\em equivalence of $G$-extensions} is an equivalence of extensions respecting the grading.

In \cite{ENO} $G$-extensions of a tensor category $\D$  were classified by means of the Brauer-Picard $2$-categorical group $\uuBrPic(\D)$
of invertible $\D$-bimodule categories.  Namely,
it was shown that 
extensions \eqref{G-extension def} correspond to monoidal $2$-functors $G\to \uuBrPic(\D)$. 
As a result,  equivalence classes of such extensions can be described in terms of certain cohomology groups associated to a homomorphism $G\to \BrPic(\D)$. 

This paper  provides a classification  of  various types of $G$-extensions 
(where $G$ is an Abelian group) of a {\em braided} tensor category $\B$.

By a {\em $2$-categorical group} (respectively, {\em braided or symmetric $2$-categorical group}) we understand a monoidal 
(respectively,  braided or symmetric monoidal $2$ category)
in which every $0$-cell is invertible with respect to the tensor product, every $1$-cell is an  equivalence, and every $2$-cell is an isomorphism. 
For a monoidal $2$-category $\mathbf{M}$ the set of its invertible objects is a 2-categorical group which we will denote by $\INV(\mathbf{M})$.

For the monoidal $2$-categories $\Bimod(\B)$, $\Mod(\B)$, $\Mod_{br}(\B)$ and $\Mod_{sym}(\B)$ (for symmetric $\B$) introduced above the $2$-categorical  groups
of invertible objects 
\begin{eqnarray}
\label{InvBrPic}
\uuBrPic(\B) &=& \INV(\Bimod(\B)) \\
\label{InvPic}
\uuPic(\B) &=& \INV(\Mod(\B)),  \\
\label{InvPicbr}
\uuPicbr(\B) &=& \INV(\Mod_{br}(\B)), \\
 \label{InvPicsym}
\uuPicsym(\B) &=& \INV(\Mod_{sym}(\B)),
\end{eqnarray}
are called the {\em Brauer-Picard}, {\em Picard}, {\em braided Picard}, and {\em symmetric Picard} 2-categorical group, respectively.
These  2-categorical groups play the key role in our  study of extensions of tensor categories.

The main results of this paper concerning graded extensions (see Chapter~\ref{ext chapter}) can be stated as follows:
\begin{equation}
\label{groupoid iso}
\xymatrix{
{\left\{\begin{array}{c}\text{the groupoid of } \\ \text{$G$-extensions of $\B$}\\ \text{of a given type}
\end{array}\right\}}\ \ar@{=}[r] &\ {\left\{\begin{array}{c}  \text{the groupoid of}\\ \text{ corresponding monoidal $2$-functors }\\ \text{ between 2-categorical groups 
$G\to \mathbf{G}$}
\end{array}\right\}}  }
\end{equation}
for an appropriate $2$-categorical group $\mathbf{G}$. These categorical $2$-groups and
the correspondence between different types of $G$-extensions and monoidal $2$-functors $G\to \mathbf{G}$ are given in Table~\ref{table-2}.

\medskip
\begin{table}[h!]
\centering
\begin{tabular}{ |p{3cm}||p{7cm}|p{4cm}|  }
 \hline
 \hline
Extensions  $\B\subset \C$ & $2$-categorical group  $\mathbf{G}$  & $2$-functors $G\to \mathbf{G}$ \\
 \hline
tensor    & $2$-categorical group  $\uuBrPic(\B)$    &  monoidal  \\
central &  $2$-categorical group  $\uuPic(\B)$   &  monoidal  \\
braided  & braided $2$-categorical group  $\uuPicbr(\B)$    &  braided  \\
symmetric &  symmetric $2$-categorical group  $\uuPic(\B)$  & symmetric \\
 \hline
\end{tabular}
\caption{\label{table-2}$G$-extensions $\B\subset \C$  and corresponding  monoidal $2$-functors.}
\end{table}
\medskip

\subsection{Homotopy groups and invariants of $2$-categorical groups}

Let $\bG$ be a 2-categorical group with the identity object $\I$.  We introduce its homotopy groups as follows:
\begin{eqnarray}
\pi_0(\bG) &=& \text{the group of isomorphism classes of objects ($0$-cells) of $\bG$}, \\
\pi_1(\bG) &=& \Aut_\bG(\I),\text{ the group of isomorphism classes of $1$-automorphisms of  $\I$}, \\
\pi_2(\bG) &=& \Aut(\id_\I), \text{ the group of $2$-automorpshisms of the identity $1$-automorphism of $\I$}.
\end{eqnarray} 
The multiplication of $\pi_0(\bG)$ is given by the tensor product of $\bG$ and the multiplication of $\pi_1(\bG),\, \pi_2(\bG)$ 
is the composition of automorphisms. 

These homotopy groups  come equipped with additional structure, which we refer to as the {\em standard invariants}, namely
a $\pi_0(\bG)$-action on $\pi_{m}(\bG)$,
\begin{equation}
\label{wp}
\pi_0(\bG)\times \pi_m(\bG)\to \pi_{m}(\bG)\quad m=0,1,2
\end{equation}
given by the conjugation with $\id_\mathcal{X}$ for $\mathcal{X}\in \pi_0(\bG)$
(this action is used while making sense of the cohomology groups below) 
and the {\em first} and the {\em second canonical classes}
\begin{equation}
\label{standard canonical classes}
\alpha_\bG\in H^3(\pi_0(\bG),\pi_1(\bG))\qquad\text{and} \qquad q_\bG\in H^3_{br}(\pi_1(\bG),\pi_2(\bG)).
\end{equation}
Part of the properties of the standard invariants is that the second canonical class is invariant under the $\pi_0(\bG)$-action. 

Here and in what follows we denote by 
\begin{equation}
H^n_{br}(A,M) := H^{n+1}(A,2;M) \quad \text{and} \quad H^n_{sym}(A,M) := H^{n+3}(A,4;M)
\end{equation}
the Eilenberg-Mac Lane cohomology \cite{EM1} of level 2 and 4, respectively. Note that $H^3_{br}(A,M)$ is isomorphic
to the group of quadratic functions from $A$ to $M$.

For a braided 2-categorical group $\bG$ the $\pi_0(\bG)$-action \eqref{wp} is trivial. The canonical classes get promoted to 
\begin{equation}
\label{braided standard canonical classes}
\alpha_\bG\in H^3_{br}(\pi_0(\bG),\pi_1(\bG))\ \quad \text{and} \quad  q_\bG\in H^3_{sym}(\pi_1(\bG),\pi_2(\bG)).
\end {equation}
An additional structure is the {\em  Whitehead products} 
\begin{equation}
\lb{wp2}\pi_n(\bG)\times \pi_m(\bG)\to \pi_{n+m+1}(\bG),\qquad n,m=0,1,2.
\end{equation}
Note that the  product $\pi_0(\bG)\times \pi_0(\bG)\to \pi_1(\bG)$ is determined by the first canonical class 
(as the polarization of the quadratic function $\alpha$). 

For a symmetric 2-categorical group $\bG$ all Whitehead products are zero and the canonical classes are 
\begin{equation}
\label{symmetric standard canonical classes}
\alpha_\bG\in H^3_{sym}(\pi_0(\bG),\pi_1(\bG)) \quad \text{and} \quad q_\bG\in H^3_{sym}(\pi_1(\bG),\pi_2(\bG)).
\end{equation}

For a tensor category $\D$ the homotopy groups and standard invariants of the 2-categorical group 
$\uuBrPic(\D)$ were examined in \cite{ENO}.  One has
\begin{equation*}
\pi_0(\uuBrPic(\D)) = \BrPic(\D),  \quad \pi_1(\uuBrPic(\D)) = \Inv(\Z(\D)), \quad \pi_2(\uuBrPic(\D)) = k^\times.
\end{equation*}
It was shown there that the $\BrPic(\D)$-action on $\Inv(\Z(\D))$ (i.e. the $\pi_0$-action on $\pi_1$)
 comes from the isomorphism $\BrPic(\D)\simeq \Aut_{br}(\Z(\D))$  and that the second canonical class
 is given by the quadratic function 
 \[
 \pi_1=\Inv(\Z(\D)) \to \pi_2 =k^\times : Z \mapsto c_{Z,Z}, 
 \] 
where $c$ denotes the braiding of $\Z(\D)$.

The homotopy groups of $2$-categorical groups introduced in \eqref{InvPic} - \eqref{InvPicsym} are
\begin{align*}
 \pi_0(\uuPic(\B)) &= \Pic(\B), & \pi_1(\uuPic(\B)) &= \Inv(\B),  & \pi_2(\uuPic(\B)) &= k^\times, \\
 \pi_0(\uuPic_{br}(\B)) &= \Pic_{br}(\B), & \pi_1(\uuPic_{br}(\B)) &= \Inv(\Z_{sym}(\B)), & \pi_2(\uuPic_{br}(\B)) &= k^\times, \\
  \pi_0(\uuPic_{sym}(\E)) &= \Pic_{sym}(\E), & \pi_1(\uuPic_{br}(\E)) &= \Inv(\E), & \pi_2(\uuPic_{br}(\E)) &= k^\times, 
\end{align*}
where  $\B$ is a braided tensor category, and $\E$ is a symmetric tensor category.

We investigate the standard invariants of the braided 2-categorical group $\uuPic_{br}(\B)$ 
and of the symmetric 2-categorical group $\uuPic(\E)$. 
For a braided tensor $\B$ we describe the  Whitehead product (Proposition \ref{PB is well defined})
\beq\lb{swp}
\pi_0 \times \pi_1 = \Pic_{br}(\B)\times \Inv(\Z_{sym}(\B))\ \to\ \pi_2 = k^\times 
\eeq
and the first canonical class (viewed as a quadratic function) 
\beq\lb{bfc}
Q: \pi_0 =\Pic_{br}(\B)\ \to\  \pi_1=  \Inv(\Z_{sym}(\B)).
\eeq
For a symmetric tensor category $\E$ 
the first canonical class becomes a  
homomorphism
\beq\lb{sfc}
Q:\Pic_{sym}(\E)\ \to\  \Inv(\E)_2
\eeq
into the 2-torsion of the group of invertible objects of $\E$.\\

\subsection{Cohomological description of (braided) monoidal $2$-functors}

In view of the identification 
\eqref{groupoid iso} it is desirable to have a good description of various types of monoidal $2$-functors $G\to \bG$.
We present one  
 in Section~\ref{Chapter on cohomology} in terms of the Eilenberg-Mac Lane cohomology.

Let $\bG$ be a 2-categorical group (respectively, braided, symmetric 2-categorical group). 
Denote by $\FUN(G,\bG)$ (respectively, $\FUNbr(G,\bG)$, $\FUNsym(G,\bG)$)
the 2-groupoid of monoidal (respectively, braided, symmetric) $2$-functors $G\to \bG$. 
Such a functor restricts on  objects to a map from  $\pi_0(\FUN(G,\bG))$ (respectively, from $\pi_0(\FUNbr(G,\bG))$, 
$\pi_0(\FUNsym(G,\bG))$) to $\Hom(G,\,\pi_0(\bG))$), i.e.
from the set of isomorphism classes of $2$-functors to the set of group homomorphisms.
A homomorphism $\phi:G\to \pi_0(\bG)$ is in the image of this map (i.e. $\phi$ can be lifted
 to a monoidal (respectively, braided, symmetric) $2$-functor if and only if the following two obstructions vanish. 

The first obstruction is the image of $\phi$ under the homomorphism
\begin{equation}
\label{o3}
o_3: \Hom(G,\, \pi_0(\bG)) \to H^3(G,\pi_1(\bG))
\end{equation}
(respectively, $\Hom(G,\, \pi_0(\bG)) \to H^3_{br}(G,\pi_1(\bG))$, $\Hom(G,\, \pi_0(\bG)) \to H^3_{sym}(G,\pi_1(\bG))$),
given by   the pullback along $\phi$ of the first canonical class  $\alpha_\bG$ defined in \eqref{standard canonical classes}
(respectively, in  \eqref{braided standard canonical classes},
 \eqref{symmetric standard canonical classes}).
The obstruction $o_3(\phi)$ vanishes if and only if  $\phi$ can be lifted to a monoidal (respectively, braided, symmetric) functor from $G$
to the $1$-categorical truncation $\Pi_{\leq 1}(\bG)$ of $\bG$.

Suppose that a lifting $F: G \to \Pi_{\leq 1}(\bG) 
$ of $\phi$ is chosen. Then the  second obstruction  is the image of $F$ under the map
\begin{equation}
o_4: \Fun(G,\,  \Pi_{\leq 1}(\bG)  
\to H^4(G,\,\pi_1(\bG))
\end{equation}
(respectively, $\Fun_{br}(G,\, \Pi_{\leq 1}(\bG) 
) \to H^4_{br}(G,\,\pi_1(\bG))$, $\Fun_{sym}(G,\, \Pi_{\leq 1}(\bG) 
) \to H^4_{sym}(G,\,\pi_1(\bG))$).
The obstruction $o_4(F)$  measures the failure of extending $F$ to a monoidal (respectively, braided, symmetric) $2$-functor $G \to \bG$.
 When $o_4(F)$ vanishes, the equivalence classes of such $2$-functors extending $F$ form a torsor over the cokernel of a certain group homomorphism
 $H^1(G,\, \pi_1(\bG))\to H^3(G,\pi_2(\bG))$  (respectively, $H^1(G,\, \pi_1(\bG))\to H^3_{br}(G,\pi_2(\bG))$, 
 $H^1(G,\, \pi_1(\bG))\to H^3_{sym}(G,\pi_2(\bG))$) depending on $F$.

\subsection{Computation of standard invariants and groups of extensions}

For a non-degenerate braided fusion category $\B$ there is a monoidal $2$-equivalence
$\Mod(\Vect) =\Mod_{br}(\B)$, see Proposition~\ref{ModbrVect}.  In particular, 
the braided 2-categorical Picard group $\uuPic_{br}(\B)$ is ``trivial" in this case and so (as is well known)
is the extension theory: any braided graded extension of $\B$ splits into the tensor product of $\B$
and a pointed braided fusion category.

Thus, the most interesting braided Picard $2$-categorical  groups come from degenerate tensor categories.
In Section~\ref{symmetric section} we compute the homotopy groups and standard invariants of
symmetric fusion categories.   For example, the homotopy groups of the braided 2-categorical group $\uuPic_{br}(\Rep(G))$,
where $G$ is a finite group, are
\begin{equation*}
\pi_0 = H^2(G,\,k^\times)\times Z(G),\qquad \pi_1 = H^1(G,\,k^\times),\qquad \pi_2 = H^0(G,\, k^\times) =k^\times,
\end{equation*}
where $Z(G)$ denotes the center of $G$.
The first canonical class \eqref{bfc} is the quadratic function 
\[
H^2(G,\,k^\times)\times Z(G)\ \to\ H^1(G,k^\times),\qquad (\gamma,z) \mapsto \gamma_z(-) = \frac{\gamma(z,-)}{\gamma(-,z)}\ 
\]
and the second canonical class is trivial.

The Whitehead product \eqref{swp} is
\[
(H^2(G,\,k^\times)\times Z(G))\times H^1(G,\,k^\times)\ \to\ k^\times,\qquad (\gamma,z)\times \chi\mapsto \chi(z).
\]
We determine the corresponding homotopy groups and maps for a general (not necessarily Tannakian) symmetric 
fusion category in Section~\ref{symcase}.

We show that the groupoid of symmetric $A$-extensions of a  symmetric tensor category $\E$ has a structure of
a symmetric $2$-categorical group $\uuExsym(A,\,\E)$. We describe an exact sequence
that can be used to compute $\pi_0(\uuExsym(A,\,\E))$ in Section~\ref{Sect cat grp sym2fun}. We also determine the group of symmetric
extensions of a symmetric fusion category in Theorem~\ref{Ex groups computed}. 

\subsection{Organization} 

Section~\ref{Chapter on cohomology} contains the technical tools we need.  We include a detailed description of
the Eilneberg-Mac Lane cohomology \cite{EM1} in low degrees and the notions of braided and symmetric monoidal $2$-categories and $2$-functors between
them \cite{DS, KV2, S}.  An important observation is that the axioms of such categories and functors can be viewed as ``non-commutative versions"
of the higher Eilneberg-Mac Lane cocycle equations (e.g., compare equations \eqref{d13} -\eqref{d3}  with commuting 
polytopes \eqref{4cocycle 1} -  \eqref{4cocycle 4}).  This is parallel to the pentagon axiom of a monoidal category being
a non-commutative version of a $3$-cocycle equation.
Of a special use to us are (braided, symmetric)  $2$-categorical groups, characterized by invertibility of their cells
with respect to the tensor product. Monoidal (braided, symmetric) $2$-functors from a finite group (viewed as discrete $2$-categorical group) 
 to (braided, symmetric) $2$-categorical groups
can be obtained  as liftings of usual (braided, symmetric) monoidal functors, provided that certain cohomological obstructions vanish.
These obstructions for monoidal (respectively, braided, symmetric) $2$-functors and parameterization of liftings are described 
in Section~\ref{Sect Functors and cohomology}
(respectively,   Section~\ref{Sect Braided functors and cohomology},  Section~\ref{Sect Symmetric functors and cohomology}).
Symmetric monoidal $2$-functors as above form a symmetric $2$-categorical group. Its group of isomorphism classes of objects
fits into a certain exact sequence  (Theorem~\ref{exact sequence for Funsym groups}).

In Section~\ref{Chapter on module categories} we recall the $2$-category $\mathbf{Mod}(\B)$ of module categories
 over a finite tensor category $\B$.  When $\B$ is braided, $\mathbf{Mod}(\B)$ is a monoidal $2$-category. Its tensor product
 can be defined either by a universal property or by an explicit construction, see Remark~\ref{explicit btB}. 

Section~\ref{Chapter on Braided  module categories} deals with braided module categories over a braided tensor category $\B$  
introduced and studied by Enriques \cite{E}, Brochier \cite{B}, and Ben-Zvi, Brochier, and Jordan \cite{BBJ}. 
In such categories the action of $\B$ has an additional symmetry compatible with the braiding of $\B$ (Definition~\ref{pb definition}).
Equivalently, a module braiding on a $\B$-module category $\M$ is the same thing as a natural tensor  isomorphism between
the $\alpha$-inductions \cite{BEK} $\alpha^\pm_\M : \B^{\opp} \to \End_\B(\M)$ (Proposition~\ref{cbs}). 
The $2$-category $\Modbr(\B)$ of braided $\B$-module categories is   $2$-equivalent to
the $2$-center of $\mathbf{Mod}(\B)$ (Theorem~\ref{Modbr=ZMod}). In particular, $\Modbr(\B)$ is braided.
The easiest examples of braided $\B$-module categories
come from tensor automorphisms of $\id_\B$, we describe these in Section~\ref{basics of brmods}. We also prove 
in Proposition~\ref{ModbrVect} that  $\Modbr(\B)\cong \Modbr(\Vect)$ when $\B$ is a non-degenerate braided fusion category. 
Finally, module categories over  a symmetric  tensor category $\E$ can be equipped with the identity $\E$-module braiding 
and so they form a symmetric monoidal $2$-category $\Modsym(\E)$.
We prove in Proposition~\ref{inducing to center} that
the induction $\Modsym(\Z_{sym}(\B))\to \Modbr(\B)$  of braided module categories from the symmetric center of $\B$
is a braided monoidal $2$-functor. 

In Section~\ref{Chapter on Picards} we describe various $2$-categorical Picard groups associated to tensor categories. These are parts
of the corresponding monoidal $2$-categories consisting of invertible  module categories.  
The new ones are the braided Picard $2$-categorical group $\uuPicbr(\B)=\mathbf{Inv}(\Modbr(\B))$ of a braided tensor category $\B$
and the symmetric Picard $2$-categorical group $\uuPicsym(\E)=\mathbf{Inv}(\Modsym(\E))$ of a symmetric tensor category $\E$.
We describe their homotopy groups, canonical classes, and Whitehead brackets. Proposition~\ref{exs} provides an exact sequence
featuring the group $\pi_0(\uuPicbr(\B))$ that can be seen as a sequence of homotopy groups of a certain fibration. Here we also describe
Azumaya algebras in braided tensor categories, as they give a convenient description of invertible module categories.

Section~\ref{symmetric section} is dedicated to the braided $2$-categorical Picard group of a symmetric fusion category $\E$.
We recall the computation of $\Pic(\E)$  due to Carnovale \cite{Car} and use it to describe the braided categorical Picard group of $\E$
and its canonical classes. 

In Section~\ref{pointed section} we compute the braided categorical Picard group of a pointed braided fusion category $\B$.
We show that in this case there is a braided monoidal equivalence of braided categorical groups $\uPicbr(\B) \cong \uPicbr(\Z_{sym}(\B))$,
where $\Z_{sym}(\B)$ is the symmetric center of $\B$,  see Proposition~\ref{Picbr ptd}.

Finally, Section~\ref{ext chapter} contains a classification of graded extensions. Tensor (respectively, central, braided, and symmetric)
graded extensions are classified in Theorem~\ref{main ENO} (respectively, Theorem~\ref{secondary ENO}, Theorem~\ref{main DN},
and Theorem~\ref{main DN sym}). We compute the group of symmetric extensions of a symmetric fusion category in Theorem~\ref{Ex groups computed}.
Here we also explain that the zesting procedure studied in \cite{DGPRZ} can be understood 
as a deformation of  a braided monoidal functor $A\to \uPicbr(\B)$ and compute Pontryagin-Whitehead obstructions to existence 
of extensions in this case.

\subsection{Acknowledgements} 

We are grateful  to Pavel Etingof, C\'esar Galindo, Corey Jones, David Jordan, Liang Kong, Victor Ostrik, and Milen Yakimov for
many useful discussions.  We also thank the anonymous referee for comments and corrections.
The first author thanks the Simons foundation for partial support. 
The work of the second author was supported  by  the  National  Science  Foundation  under  
Grant No.\ DMS-1801198.  This material is based upon work supported by the National Science Foundation under Grant No. DMS-1440140, 
while the authors were in residence at the Mathematical Sciences Research Institute in Berkeley, California, during the Spring 2020 
semester.

\section{Higher categorical groups and group cohomology}
\label{Chapter on cohomology}

\subsection{Eilenberg-Mac Lane cohomology}
\label{Section EM cohomology}

We denote by $C^*(A,M)$ the normalised standard complex of the abelian group $A$ with coefficients in the trivial $A$-module $M$. 

\begin{remark}
We will refer to cochain complexes for the second, third, and fourth Eilenberg-Mac Lane cohomology groups as {\em braided},
{\em sylleptic}, and {\em symmetric}, respectively. This is justified since such cochains  give rise to  braided, sylleptic, and symmetric
$2$-categorical groups, see  Sections~\ref{Section higher categories} and \ref{2cg}.  The explicit descriptions of these complexes are recalled
below.
\end{remark}

We denote by $C^*_{br}(A,M) = C^{*+1}(K(A,2),M)$ the normalised standard complex computing the second Eilenberg-Mac Lane cohomology \cite{EM1}.
The first few terms of the cochain complex $C^*_{br}(A,M)$ are as follows:
\begin{align*}
C^0_{br}(A,M) &= C^0(A,M)=M,\qquad C^1_{br}(A,M)=C^1(A,M),\qquad C^2_{br}(A,M)=C^2(A,M),\\ 
C^3_{br}(A,M) &= C^3(A,M)\op C^2(A,M) = \{(a(-,-,-),a(- | -))\}, \\
C^4_{br}(A,M) &= C^4(A,M) \op C^3(A,M) \op C^3(A,M) = \{(a(-,-,-,-),a(-,-|-),a(-|-,-))\}, \\
C^5_{br}(A,M) &= C^5(A,M)\op C^4(A,M) \op C^4(A,M)\op C^4(A,M)\op C^3(A,M) \\
 &= \{(a(-,-,-,-,-),a(-,-,-|-),a(-,-|-,-),a(-|-,-,-),a(-|-|-)) \}
 \end{align*} 
with the differentials
\[
d :C^2_{br}(A,M)\to C^3_{br}(A,M) 
\]
\begin{eqnarray}
d(a)(x,y,z)  &=& a(y,z) - a(xy,z) + a(x,yz) - a(x,y),\\
d(a)(x|y) &=& a(y,x) - a(x,y); 
\end{eqnarray}
\[
d :C^3_{br}(A,M)\to C^4_{br}(A,M) 
\]
\begin{align}
\label{3cob1}
d(a)(x,y,z,w)  &=& a(y,z,w) - a(xy,z,w) + a(x,yz,w) - a(x,y,zw) + a(x,y,z), \\
\label{3cob2}
d(a)(x,y|z) &=& a(y|z) - a(xy|z) + a(x|z) + a(x,y,z) - a(x,z,y) + a(z,x,y), \\
\label{3cob3} 
d(a)(x|y,z)  &=&  a(x|y) - a(x|yz) + a(x|z) - a(x,y,z) + a(y,x,z) - a(y,z,x), 
\end{align}
and
\[
d : C^4_{br}(A,M)\to C^5_{br}(A,M)  
\]
\begin{eqnarray}
& & \begin{split}
\label{d5}
d(a)(x,y,z,w,u)  &=  a(y,z,w,u) - a(xy,z,w,u) + a(x,yz,w,u) \\
&\quad  - a(x,y,zw,u) + a(x,y,z,wu) - a(x,y,z,w), 
\end{split} \\
& & \begin{split}
\label{d13}
d(a)(x|y,z,w)  &=  a(x|z,w) - a(x|yz,w) + a(x|y,zw)- a(x|y,z) \\
& \quad - a(x,y,z,w) + a(y,x,z,w) - a(y,z,x,w) + a(y,z,w,x),
\end{split} \\
& & \begin{split}
 \label{d31}
d(a)(x,y,z|w)  &=  a(y,z|w) - a(xy,z|w) + a(x,yz|w)- a(x,y|w) - \\
&\quad a(x,y,z,w) + a(x,y,w,z) - a(x,w,y,z) + a(w,x,y,z), 
\end{split} \\
& & \begin{split}
\label{d22}
d(a)(x,y|z,w)  &=  a(y|z,w) - a(xy|z,w) + a(x|z,w) - a(x,y|w) + a(x,y|zw) - a(x,y|z)   \\
&\quad  +  a(x,y,z,w) - a(x,z,y,w) + a(z,x,y,w)  \\
&\quad + a(x,z,w,y) - a(z,x,w,y) + a(z,w,x,y), \\
\end{split} \\
\label{d3}
& & d(a)(x|y|z)  =  -a(x,y|z) + a(y,x|z) - a(x|y,z) + a(x|z,y),\qquad x,y,z,w\in A.
\end{eqnarray}

\begin{example}
The first few terms of the cochain complex $C^*_{br}(\mathbb{Z}/2\mathbb{Z},M)$ are 
$$\xymatrix{ M \ar[r]^0 & M \ar[r]^{d_1} & M \ar[r]^{0} & M^2 \ar[r]^{d_3} & M^3 \ar[r]^{d_4} & M^5 \ar[r] & ...}$$
where $M^n$ is the direct sum of $n$ copies of $M$ and 
$$\begin{array}{rcl}
d_1(m) & = & 2m \\
d_3(m,l) & = & (2m,2l+m,2l-m) \\
d_4(m,l,k) & = & (0,0,2(m-l+k),0,0) \\
\end{array}$$
Thus the first few braided cohomology groups are 
\begin{eqnarray*}
H^0_{br}(\mathbb{Z}/2\mathbb{Z},M) &=& M,\\
H^1_{br}(\mathbb{Z}/2\mathbb{Z},M) &=& M_2,\\
H^2_{br}(\mathbb{Z}/2\mathbb{Z},M) &=& M/2M, \\
H^3_{br}(\mathbb{Z}/2\mathbb{Z},M) &=& M_4,\\
H^4_{br}(\mathbb{Z}/2\mathbb{Z},M) &=& M_2\oplus M/4M.
\end{eqnarray*}
Here $M_s = \{m\in M|\ sm=0\}$. 
\end{example}

We denote by $C^*_{syl}(A,M) = C^{*+2}(K(A,3),M)$ the normalised standard complex computing the third Eilenberg-Mac Lane cohomology \cite{EM1}.
The first few terms of the cochain complex $C^*_{syl}(A,M)$ are as follows:
\begin{eqnarray*}
& &C^0_{syl}(A,M) =M,\quad C^1_{syl}(A,M)=C^1(A,M),\\
& & C^2_{syl}(A,M)=C^2(A,M),\quad C^3_{syl}(A,M) = C^3_{br}(A,M),\\
& & C^4_{syl}(A,M) = C^4_{br}(A,M) \op C^2(A,M) = \{(a(-,-,-,-),a(-,-|-),a(-|-,-),a(-||-))\},\\
& & \begin{split} C^5_{syl}(A,M) & = C^5_{br}(A,M)\op C^3(A,M)\op C^3(A,M)=\\
&= \{(a(-,-,-,-,-),a(-,-,-|-),a(-,-|-,-),a(-|-,-,-), \\
& \quad \quad a(-|-|-),a(-,-||-),a(-||-,-)) \}
\end{split}
\end{eqnarray*}
with the additional differentials
\[
d:C^3_{syl}(A,M)\to C^4_{syl}(A,M)
\]
\beq\label{}
d(a)(x||y) = a(x|y) + a(y|x),
\eeq
and 
\[
d:C^4_{syl}(A,M)\to C^5_{syl}(A,M)
\]
\beq\label{sym1}
d(a)(x||y,z) = a(x|y,z) + a(y,z|x) + a(x||y) + a(x||z) - a(x||yz),
\eeq
\beq \label{sym2}
d(a)(x,y||z) = a(x,y|z) +a(z|x,y) + a(x||z) + a(y||z) - a(xy||z).
\eeq

\begin{example}
The first few sylleptic cohomology groups of $\bZ/2\bZ$ are 
\begin{eqnarray*}
H^0_{syl}(\mathbb{Z}/2\mathbb{Z},M) &=& M,\\
H^1_{syl}(\mathbb{Z}/2\mathbb{Z},M) &=& M_2,\\
H^2_{syl}(\mathbb{Z}/2\mathbb{Z},M) &=& M/2M, \\
H^3_{syl}(\mathbb{Z}/2\mathbb{Z},M) &=& M_2,\\
H^4_{syl}(\mathbb{Z}/2\mathbb{Z},M) &=& M_2\oplus M/2M.
\end{eqnarray*}
\end{example}

We denote by $C^*_{sym}(A,M) = C^{*+3}(K(A,4),M)$ the normalised standard complex computing the fourth Eilenberg-Mac Lane cohomology \cite{EM1}.
The first few terms of the cochain complex $C^*_{sym}(A,M)$ are as follows:
\begin{eqnarray*}
& & C^0_{sym}(A,M) =M,\qquad C^1_{sym}(A,M)=C^1(A,M),\qquad C^2_{sym}(A,M)=C^2(A,M),\\ 
& & C^3_{sym}(A,M) = C^3_{br}(A,M),\qquad C^4_{sym}(A,M) = C^4_{syl}(A,M),\\
& & \begin{split} & C^5_{sym}(A,M) = C^5_{syl}(A,M)\op C^2(A,M) = \\
& \quad = \{ (a(-,-,-,-,-),a(-,-,-|-),a(-,-|-,-),a(-|-,-,-),\\
& \quad \quad a(-|-|-),a(-,-||-),a(-||-,-),a(-|||-)) \} 
\end{split}
\end{eqnarray*}
with the additional differential
\[
d:C^4_{sym}(A,M)\to C^5_{sym}(A,M)
\]
\beq\label{ssym}
d(a)(x|||y) = a(x||y) - a(y||x), \qquad x,y\in A.
\eeq

\begin{example}
The first few level 4 cohomology groups of $\bZ/2\bZ$ are the same as the symmetric cohomology
$$H^n_{sym}(\mathbb{Z}/2\mathbb{Z},M) = H^n_{syl}(\mathbb{Z}/2\mathbb{Z},M),\qquad n\leq 4.$$
\end{example}

\begin{example}
\label{cohomology 0-3} 
It is immediate from the definitions that 
\begin{eqnarray*}
& & H^0_ {br}(A,\,M) = H^0_ {syl}(A,\,M) = H^0_ {sym}(A,\,M) \cong M, \\ 
& & H^1_ {br}(A,\,M) = H^1_ {syl}(A,\,M) = H^1_ {sym}(A,\,M) \cong  \Hom(A,\,M), \\
& & H^2_ {br}(A,\,M) = H^2_ {syl}(A,\,M) = H^2_ {sym}(A,\,M) \cong  \Ext(A,\,M).
\end{eqnarray*}
It was shown in \cite{EM2} that there are isomorphisms
\begin{equation*}
H^3_ {br}(A,\,M) \cong \Quad(A,\,M),\qquad H^3_ {syl}(A,\,M) = H^3_ {sym}(A,\,M) \cong \Hom(A,\,M_2),
\end{equation*}
given by 
\[
(a(-,-,-),a(- | -))\mapsto q,\quad \text{ where } q(x) = a(x|x),\, x\in A. 
\]
Here $\Quad(A,M)$ is the group of quadratic maps and $M_2 = \{m\in M|\ 2m=0\}$ is the 2-torsion subgroup of $M$.
\end{example}

The $4$th cohomology groups are especially important for our purposes.
Let us now assume that $M$ is divisible.  The following results  are from \cite{EM2}.

\begin{example}
\label{cohomology 4} 
There is an isomorphism
\begin{equation}
\label{H4sym}
\theta_{sym} :H^4_{sym}(A,\, M) \xrightarrow{\sim } \Hom(A_2,\,M)
\end{equation}
assigning to a symmetric $4$-cocycle $(a(-,-,-,-),a(-,-|-),a(-|-,-),a(-||-))$ the homomorphism
\begin{equation}
\label{EM hom}
A_2\to M : x\mapsto a(x,x|x) - a(x|x,x) - a(x,x,x,x).
\end{equation}

There is an isomorphism
\begin{equation}
\label{H4syl}
\theta_{syl} :H^4_{syl}(A,\, M) \xrightarrow{\sim } \Hom(A_2,\,M) \oplus \Hom(\Extpower^2 A,\, M),
\end{equation}
whose first component   
is given by \eqref{EM hom} and the second component  assigns to a sylleptic $4$-cocycle
$(a(-,-,-,-),a(-,-|-),a(-|-,-),a(-||-))$ the homomorphism
\begin{equation}
\label{EM form}
\Extpower^2 A \to M : x \wedge y \mapsto a(x||y) - a(y||x),
\end{equation}
which is the obstruction for a sylleptic $4$-cocycle to be symmetric. 

Finally, there is a homomorphism 
\begin{equation}
\label{H4br}
\theta_{br} :H^4_{br}(A,\, M) \xrightarrow{ } \Ext(A,\, \Hom(A,\, M)),
\end{equation}
which is the obstruction for a braided $4$-cocycle to have a sylleptic  structure. It is defined as follows.
Let $(a(-,-,-,-),a(-,-|-),a(-|-,-))$ be a braided $4$-cocycle. For any $x\in A$ define a function $b_x\in C^2(A,\, M)$ by
\[
b_x(y,z)= a(x|y,z) + a(y,z |x),\qquad x,y,z\in A.
\]
It follows from formulas \eqref{d5} -\eqref{d3} and divisibility of $M$ that $b_x$ is a $2$-coboundary.
That is there exists a function  $a(-||-)\in C^2(A,\, M)$ such that \eqref{sym1} vanishes.
The function
\[
A^3\to M : (x,y,z) \mapsto a(x||y)-a(y||x) +a(x||z) -a(z||x) - a(x||yz) + a(yz||x)
\]
is multiplicative in $x$ and, hence, defines a symmetric $2$-cocycle $g$ on $A$ with values in $\Hom(A,\,M)$.
This $2$-cocycle is cohomologically trivial if and only if the given braided $4$-cocycle admits a sylleptic structure.
We set $\theta_{br}(a(-,-,-,-),a(-,-|-),a(-|-,-))$ to be the class of $g$ in $\Ext(A,\, \Hom(A,\, M))$. 

The kernel of \eqref{H4br} is isomorphic to $\Hom(A_2,\, M)$ via \eqref{EM hom}.
\end{example}

\subsection{Higher braided monoidal categories and functors}
\label{Section higher categories}
 
Semistrict monoidal $2$-categories were defined by Kapranov  and Voevodsky \cite{KV2}
and also by Day and Street \cite{DS} under the name of Gray monoids.   It was shown that every
monoidal $2$-category is equivalent to a semistrict one. We refer the reader to these papers 
and to \cite{Stay}
for basic definitions. All monoidal $2$-categories considered in this paper will be assumed semistrict.

Let  $\F,\H:\bM\to\bN$ be $2$-functors between 2-categories.  
Recall that a {\em pseudo-natural transformation} $P: \F\to \H$ is a collection of 
1-morphisms $P_\M: \F(\M)\to \H(\M)$  and  invertible 2-cells
\begin{equation}
\label{pi}
\xymatrix{
\F(\M) \ar[dd]_{\F(F)}_{}="a" \ar[rr]^{P_\M} && \H(\M) \ar[dd]^{\H(F)}="b" \\
&& \\
\F(\N) \ar[rr]^{P_\N} && \H(\N),
\ar@{}"a";"b"^(.25){}="b2"^(.75){}="c2" \ar@{=>}^{P_F}"b2";"c2"
}
\end{equation}
for all objects $\M$ and $1$-morphisms $F:\M \to \N$ in $\bM$  such that $P_{\id_\M} = \id_{P_\M}$ and 
\begin{equation}
\label{pi fg}
P_{F\circ G} = P_{F} \circ P_{G}
\end{equation}
 for all composable 1-morphisms $F$ and $G$. 

Let $P,Q: \F\to \H$  be pseudo-natural transformations between $2$-functors.
A {\em modification} $\eta:P\to Q$ is a collection of 2-cells
\begin{equation}
\label{modification}
\xymatrix{
\F(\M)   \ar@/^1pc/[rrrr]^{P_\M}_{}="a"      \ar@/^-1pc/[rrrr]_{Q_\M}^{}="d" &&&& \H(\M) ,
\ar@2{->}^{\eta_\M}"a";"d"
}
\end{equation}
for all objects $\M$ in $\bM$, natural in 1-morphisms in $\bM$.

\begin{definition}
A {\em (semistrict)
braided monoidal $2$-category} \cite{KV2, Cr, BN} consists of a (semistrict) monoidal $2$-category 
$(\mathbf{M},\, \bt,\, \mathcal{I})$, where $\bt$ is the tensor product, equipped with invertible $2$-cells
\begin{equation}
\label{btZW}
\xymatrix{
\M \bt \N \ar[dd]_{\M \bt W}_{}="a" \ar[rr]^{Z\bt \N} && \M' \bt \N \ar[dd]^{\M' \bt W}="b" \\
&& \\
 \M \bt \N'   \ar[rr]^{Z\bt \N'} && \M' \bt \N',
\ar@{}"a";"b"^(.25){}="b2"^(.75){}="c2" \ar@{=>}^{\bt_{Z,W}}"b2";"c2"
}
\end{equation}
for any $Z\in \mathbf{M}(\M,\,\M'),\, W\in \mathbf{M}(\N,\,\N')$, 
and $\mathcal{I}$ is  the unit object, 
\nl
together  with a pseudo-natural equivalence ({\em braiding})
\[
B_{\M,\,\N} : \M \bt  \N \to \N  \bt \M,\qquad \M,\N\in \mathbf{M},  
\]
invertible $2$-cells
\begin{equation}
\label{BZN and BMW}
\xymatrix{
\M \bt \N \ar[dd]_{B_{\M,\N}}_{}="a" \ar[rr]^{Z \bt \N} && \M' \bt \N \ar[dd]^{B_{\M',\N}}="b" 
&&
\M \bt \N \ar[dd]_{B_{\M,\N}}_{}="c" \ar[rr]^{\M \bt W} && \M \bt \N' \ar[dd]^{B_{\M',\N}}="d" 
\\
&& \\
 \N \bt \M   \ar[rr]^ {\N \bt Z} && \N\bt \M',
 &&
  \N \bt \M   \ar[rr]^{W \bt \M}  && \N' \bt \M,
\ar@{}"a";"b"^(.25){}="b2"^(.75){}="c2" \ar@{=>}^{B_{Z,\N}}"b2";"c2"
\ar@{}"c";"d"^(.25){}="b2"^(.75){}="c2" \ar@{=>}^{B_{\M,W}}"b2";"c2"
}
\end{equation}
satisfying $B_{Z_1,\N} \circ B_{Z_2,\N} = B_{Z_1\ot Z_2,\N}$ and $B_{\M,W_1} \circ B_{\M,W_2} = B_{\M,W_1\ot W_2}$,
and two invertible modifications\footnote{Below we omit the identity functors and the tensor product symbol $\bt$, so we write $\M\N$ for $\M \bt \N$.}
\begin{equation}
\label{betaxyz}
\xymatrix{
&& \L \N  \M  \ar[drr]^{B_{\L,\N}} && 
&  
&& \M \L \N \ar[drr]^{B_{\L,\N}} && \\
\L  \M  \N \ar[urr]^{ B_{\M,\N}} \ar[rrrr]_{B_{\L\M, \N}}^{}="B" &&&&  \N \L \M 
&
\L \M  \N \ar[urr]^{B_{\L,\M}} \ar[rrrr]_{B_{\L,\M\N}}^{}="b" &&&&  \M  \N  \L,  
\ar@{}"1,3";"B"^(.25){}="x"^(.85){}="y" \ar@2{->}^{\beta_{\L, \M \mid \N}}"x";"y"
\ar@{}"1,8";"b"^(.25){}="x"^(.85){}="y" \ar@2{->}^{\beta_{\L \mid  \M,\N}}"x";"y"
}
\end{equation}
satisfying the following axioms\footnote{Equalities of $2$-cell compositions in this paper can be used to represent 
commuting polytopes \cite{KV1}. These polytopes are recovered by gluing the  diagrams on both sides of equality
along the perimeter}:
\begin{equation}
\label{4cocycle 1}
\scalebox{1.0}{
\xymatrix{
\L  \K  \M  \N  \ar[dd]_{B_{\K,\M}}  &&& \K  \L  \M  \N  \ar[lll]_{B_{\K,\L}}  \ar[dd]^{B_{\K,\L\M\N}}_{}="c"  \ar[ddlll]_>>>>>>>>>>>>{B_{\K,\L\M}}_{}="a"^{}="b"  
&  & 
\L  \K  \M  \N \ar[dd]_{B_{\K,\M}}   \ar[ddrrr]_>>>>>>>>>>>>{B_{\K,\M\N}}_{}="A"^{}="B" &&& \K  \L  \M  \N   \ar[lll]_{B_{\K,\L}}  \ar[dd]^{B_{\K,\L\M\N}}_{}="C"  \\
&&&& = &&&& \\
\L  \M  \K  \N    \ar[rrr]_{B_{\K,\N}}   &&& \L  \M  \N  \K  & & \L  \M  \K  \N    \ar[rrr]_{B_{\K,\N}}    &&& \L  \M  \N  \K, 
\ar@{}"1,1";"a"^(.25){}="a1"^(.75){}="b1" \ar@2{->}^{\beta_{\K|\L,\M}}"a1";"b1"
\ar@{}"b";"c"^(.25){}="a1"^(.75){}="b1" \ar@2{->}_{\beta_{\K|\L\M,\N}}"a1";"b1"
\ar@{}"3,6";"A"^(.25){}="a1"^(.75){}="b1" \ar@2{->}_{\beta_{\K|\M,\N}}"a1";"b1"
\ar@{}"B";"C"^(.25){}="a1"^(.75){}="b1" \ar@2{->}^{\beta_{\K|\L,\M\N}}"a1";"b1"
}
}
\end{equation}

\begin{equation}
\label{4cocycle 2}
\scalebox{1.0}{
\xymatrix{
\K  \L  \N  \M  \ar[dd]_{B_{\L,\N}}  &&& \K  \L  \M  \N  \ar[lll]_{B_{\M,\N}}  \ar[dd]^{B_{\K\L\M,\N}}_{}="c"  \ar[ddlll]_>>>>>>>>>>>>{B_{\L\M,\N}}_{}="a"^{}="b"  
&  & 
\K  \L  \N  \M \ar[dd]_{B_{\L,\N}}   \ar[ddrrr]_>>>>>>>>>>>>{B_{\K\L,\N}}_{}="A"^{}="B" &&& \K  \L  \M  \N   \ar[lll]_{B_{\M,\N}}  \ar[dd]^{B_{\K\L\M,\N}}_{}="C"  \\
&&&& = &&&& \\
\K  \N  \L  \M    \ar[rrr]_{B_{\K,\N}}   &&&  \N  \K  \L  \M  & & \L  \M  \K  \N    \ar[rrr]_{B_{\K,\N}}    &&&  \N  \K  \L  \M, 
\ar@{}"1,1";"a"^(.25){}="a1"^(.75){}="b1" \ar@2{->}^{\beta_{\L,\M|\N}}"a1";"b1"
\ar@{}"b";"c"^(.25){}="a1"^(.75){}="b1" \ar@2{->}_{\beta_{\K,\L\M|\N}}"a1";"b1"
\ar@{}"3,6";"A"^(.25){}="a1"^(.75){}="b1" \ar@2{->}_{\beta_{\K,\L|\N}}"a1";"b1"
\ar@{}"B";"C"^(.25){}="a1"^(.75){}="b1" \ar@2{->}^{\beta_{\K\L,\M|\N}}"a1";"b1"
}
}
\end{equation}

\begin{equation}
\label{4cocycle 3}
\scalebox{1.0}{
\xymatrix{
 & \K  \L  \M  \N \ar[dl]_{B_{\L,\M}} \ar[dddl]^{}="b" \ar[dddr]^{}="d" \ar[dr]^{B_{\K\L,\M}} &              &&               & \K  \L  \M  \N \ar[dl]_{B_{\L,\M}} \ar[dr]^{B_{\K\L,\M}} ^{}="B" &       \\
 \K  \M  \L  \N \ar[dd]^{}="a"  && \M  \K  \L  \N  \ar[dd]^{}="c"                               &&      \K  \M  \L  \N \ar[dd]^{}="X" \ar[rr]^{}="A"  && \M  \K  \L  \N  \ar[dd]^{}="E"  \ar[dddl]^{}="C"   \\
    & &                                                                     &=&                     & &   \\
 \K  \M  \N  \L \ar[rr]^{}="e" \ar[dr] && \M  \N  \K  \L                                   &&                \K  \M  \N  \L  \ar[dr]^{}="F" && \M  \N  \K  \L    \\
 & \M  \K  \N  \L \ar[ur]^{}="f" &                                                     &&         & \M  \K  \N  \L, \ar[ur]^{}="D" &
\ar@{}"2,1";"b"^(.25){}="a1"^(.95){}="b1"\ar@2{->}^{\beta_{\L|\M,\N}}"a1";"b1"
\ar@{}"2,3";"d"^(.25){}="a1"^(1.05){}="b1"\ar@2{->}^{\beta_{\K\L|\M,\N}}"a1";"b1"
\ar@{}"4,1";"d"^(.15){}="a1"^(.95){}="b1"\ar@2{->}^{\beta_{\K,\L|\M\N}}"a1";"b1"
\ar@{}"5,2";"e"^(.15){}="a1"^(.95){}="b1"\ar@2{->}_{\beta_{\K|\M,\N}}"a1";"b1"
\ar@{}"2,5";"B"^(.25){}="a1"^(.95){}="b1"\ar@2{->}^{\beta_{\K,\L|\M}}"a1";"b1"
\ar@{}"X";"C"^(.20){}="a1"^(.80){}="b1" \ar@2{->}_{\bt_{B_{\K,\M},B_{\L,\N}}}"a1";"b1"
\ar@{}"3,7";"C"^(.25){}="a1"^(.95){}="b1" \ar@2{->}_{\beta_{\K,\L|\N}}"b1";"a1"
}
}
\end{equation}
\begin{equation}
\label{4cocycle 4}
\scalebox{1.0}{
\xymatrix{
& \K  \L  \M \ar[dl] \ar[dddl]^{}="b" \ar[dr]^{}="A" &                           &&         & \K  \L  \M \ar[dl]^{}="C" \ar[dddr]^{}="h" \ar[dr] &       \\
\L  \K   \M \ar[dd]^{}="a" && \K  \M  \L \ar[dd] \ar[dddl]^{}="d"         &&           \L  \K  \M \ar[dd] \ar[dddr]^{}="f"   && \K  \M  \L \ar[dd]^{}="g"   \\
&&                                                                               &=&            \\
\L  \M   \K \ar[dr]^{}="B" && \M  \K  \L \ar[dl]^{}="c"                                     &&            \L  \M  \K \ar[dr]^{}="e" && \M  \K  \L \ar[dl]^{}="D"   \\
& \M  \L  \K &                                                                         &&              & \M  \L  \K, &     
\ar@{}"2,1";"b"^(.25){}="a1"^(.85){}="b1"  \ar@2{->}^{\beta_{\K|\L,\M}}"a1";"b1"
\ar@{}"4,3";"d"^(.25){}="a1"^(1.10){}="b1"  \ar@2{->}_{\beta_{\K|\M,\L}}"a1";"b1"
\ar@{}"4,5";"f"^(.35){}="a1"^(1.05){}="b1"  \ar@2{->}^{\beta_{\L,\K|\M}}"a1";"b1"
\ar@{}"2,7";"h"^(.35){}="a1"^(1.05){}="b1"\ar@2{->}^{\beta_{\K,\L|\M}}"a1";"b1"
\ar@{}"A";"B"^(.15){}="a1"^(.85){}="b1" \ar@2{->}_{B_{\K,B_{\L,\N}}}"a1";"b1"
\ar@{}"C";"D"^(.15){}="a1"^(.85){}="b1" \ar@2{->}_{B_{B_{\K,\L}, \M}}"a1";"b1"
}
}
\end{equation}  
\end{definition}
for all $\K,\, \L,\,\M,\,\N \in \mathbf{M}$. 

\begin{definition}
A {\em sylleptiic} monoidal $2$-category is a braided monoidal $2$-category $\mathbf{M}$ with  an invertible {\em syllepsis}
modification
\begin{equation}
\label{symmetry 2-cell}
\xymatrix{
\M  \N    \ar@/^1pc/[rrrr]^{B_{\M,\N}}_{}="a"   &&&& \N \M \ar@/^1pc/[llll]^{B_{\N,\M}}_{}="d",
\ar@2{->}^{\tau_{\M,\N}}"a";"d"
}
\end{equation}
i.e. $\tau_{\M,\N}$ is an invertible modification between $B_{\N,\M} B_{\M,\N}$ and $\id_{\M\boxtimes \N}$ 
such that 
\begin{equation}
\label{1st symmetric}
\scalebox{0.85}{
\xymatrix{
&& \L  \N \M   \ar@/^3pc/[ddrr]^{B_{\L,\N}}  &&&&&& 
\L\N\M   \ar@/^3pc/[ddrr]^{B_{\L,\N}}_{}="d2" \ar[ddll]^{B_{\N,\M}}_{}="a2" 
&&
\\
&&&&& = &&&&&
\\
 \L  \M \N \ar@/^3pc/[uurr]^{B_{\M,\N}} \ar[rrrr]^{B_{\L\M,\N}}="a1"_{}="b1"
 &&&& \N\L\M  \ar@/^3pc/[llll]^{B_{\N,\L\M}}_{}="c1" && 
  \L  \M \N \ar@/^3pc/[uurr]^{B_{\M,\N}}_{}="b2" &&&& \N\L\M, \ar@/^3pc/[llll]^{B_{\N,\L\M}}_{}="e2" \ar[uull]^{B_{\N,\L}}_{}="c2"
 \ar@2{->}^{\beta_{\L,\M |\N}}"1,3";"a1"
 \ar@2{->}^{\tau_{\L\M,\N}}"b1";"c1"
 \ar@2{->}^{\beta_{\N|\L,\M }}"1,9";"e2"
 \ar@2{->}^{\tau_{\N,\M}}"a2";"b2"
  \ar@2{->}^{\tau_{\N,\L}}"c2";"d2"
}
}
\end{equation}

\begin{equation}
\label{2nd symmetric}
\scalebox{0.85}{
\xymatrix{
&& \M \L  \N    \ar@/^3pc/[ddrr]^{B_{\L,\N}}  &&&&&& 
\M\L\N   \ar@/^3pc/[ddrr]^{B_{\L,\N}}_{}="d2" \ar[ddll]^{B_{\M,\L}}_{}="a2" 
&&
\\
&&&&& = &&&&&
\\
 \L  \M \N \ar@/^3pc/[uurr]^{B_{\L,\M}} \ar[rrrr]^{B_{\L,\M\N}}="a1"_{}="b1"
 &&&& \M\N\L  \ar@/^3pc/[llll]^{B_{\M\N,\L}}_{}="c1" && 
  \L  \M \N \ar@/^3pc/[uurr]^{B_{\L,\M}}_{}="b2" &&&& \M\N\L, \ar@/^3pc/[llll]^{B_{\M\N,\L}}_{}="e2" \ar[uull]^{B_{\N,\L}}_{}="c2"
 \ar@2{->}^{\beta_{\L|\M,\N}}"1,3";"a1"
 \ar@2{->}^{\tau_{\L,\M\N}}"b1";"c1"
 \ar@2{->}^{\beta_{\M,\N|\L}}"1,9";"e2"
 \ar@2{->}^{\tau_{\M,\L}}"a2";"b2"
  \ar@2{->}^{\tau_{\N,\L}}"c2";"d2"
}
}
\end{equation}
commute for all objects $\L,\,\M,\,\N$  in $\mathbf{M}$.
\end{definition}

\begin{definition} 
A sylleptic braided monoidal $2$-category $\mathbf{M}$ is called {\em symmetric} if its syllepsis \eqref{symmetry 2-cell}
satisfies
\begin{equation}
\label{3rd symmetric}
\scalebox{0.85}{
\xymatrix{
\M\N    \ar@/^3pc/[rrrr]^{B_{\M,\N}}_{}="a"    \ar@/^-2pc/[rrrr]_{B_{\M,\N}}^{}="d" &&&& \N \M \ar[llll]_{B_{\N,\M}}="b"_{}="c"
& =&
\M \N    \ar@/^3pc/[rrrr]^{B_{\M,\N}}_{}="a1"    \ar@/^-2pc/[rrrr]_{B_{\M,\N}}^{}="d1" &&&& \N  \M 
\ar@2{->}^{\tau_{\M,\N}}"a";"b"
\ar@2{->}^{\tau_{\N,\M}}"c";"d"
\ar@{=>}_<<<<<<<<{\id_{\M\boxtimes \N}}"a1";"d1"
}
}
\end{equation}
for all objects $\M,\,\N$  in $\mathbf{M}$.
\end{definition}

\begin{definition}
A {\em monoidal  $2$-functor} $\F:\bM\to \bM'$ between monoidal 2-categories is a $2$-functor
along with a pseudo-natural  equivalence
\begin{equation}
\label{FNM}
F_{\M,\N}:\F(\M) \F(\N)\to \F(\M \N),
\end{equation}
an equivalence $U: \F(\I) \to \I$, and invertible  modifications
\begin{equation}
\label{alphaLMN}
\xymatrix{
 \F(\L) \F(\M) \F(\N) \ar[dd]_{F_{\L,\M}}_{}="a" \ar[rr]^{F_{\M,\N}} &&  \F(\L) \F(\M\N) \ar[dd]^{F_{\L,\M\N}}_{}="b" \\
&& \\
 \F(\L\M) \F(\N) \ar[rr]^{F_{\L\M,\N}} &&  \F(\L\M\N),
\ar@{}"a";"b"^(.25){}="a1"^(.75){}="b1" \ar@{=>}^{\alpha_{\L,\M,\N}}"a1";"b1"
}
\end{equation}
\begin{equation}
\label{lambda and rho}
\xymatrix{
\F(\I)\F(\M) \ar[dd]_{F_{\I,\M}}_{}="a" \ar[rr]^{U} && \I \F(\M) \ar[dd]^{L_{\F(\M)}}="b" && 
\F(\M)\F(\I) \ar[dd]_{F_{\M,\I}}_{}="A" \ar[rr]^{U} && \F(\M) \I, \ar[dd]^{R_{\F(\M)}}="B" 
\\
&&& \text{and} &&& \\
\F(\I\M) \ar[rr]^{\F(L_{\M})} && \F(\M), &&
\F(\M\I) \ar[rr]^{\F(R_{\M})} && \F(\M),
\ar@{}"a";"b"^(.25){}="b2"^(.75){}="c2" \ar@{=>}^{\lambda_{\M}}"b2";"c2"
\ar@{}"A";"B"^(.25){}="B2"^(.75){}="C2" \ar@{=>}^{\rho_{\M}}"B2";"C2"
}
\end{equation}
where $L$ and $R$ denote the unit constraints of $\bM$,
such that 
\begin{equation}
\label{mon2fun}
\scalebox{0.65}{
\xymatrix{
& \F(\K) \F(\L) \F(\M) \F(\N) \ar[dl]_{F_{\K,\L}} \ar[dr]^{F_{\M,\N}} & & &
& \F(\K) \F(\L) \F(\M) \F(\N) \ar[dl]_{F_{\K,\L}}_{}="A3" \ar[dr]^{F_{\M,\N}} \ar[dd]^{F_{\L,\M}}_{}="a3" 
 \\
 \F(\K\L) \F(\M) \F(\N) \ar[dd]_{F_{\K\L,\M}}_{}="a" \ar[dr]^{F_{\M,\N}} &&  \F(\K)\F(\L)\F(\M\N) \ar[dd]^{F_{\L,\M\N}}_{}="b" \ar[dl]_{F_{\K,\L}}_{}="B" &  &
 \F(\K\L) \F(\M) \F(\N) \ar[dd]_{F_{\K\L,\M}}_{}="b3"  &&  \F(\K)\F(\L)\F(\M\N)  \ar[dd]^{F_{\L,\M\N}}_{}="c3" 
  \\
&  \F(\K\L) \F(\M\N) \ar[dd]^{F_{\K\L,\M\N}}_{}="c" & & = &  
&  \F(\K) \F(\L\M) F(\N) \ar[dl]_{F_{\K,\L\M}}^{}="B3"  \ar[dr]^{F_{\L\M,\N}} & 
\\
 \F(\K\L\M) \F(\N) \ar[dr]_{F_{\K\L\M,\N}} && \F(\K)\F(\L\M\N) \ar[dl]^{F_{\K,\L\M\N}}^{}="C" &  &
 \F(\K\L\M) \F(\N) \ar[dr]_{F_{\K\L\M,\N}} && \F(\K)\F(\L\M\N) \ar[dl]^{F_{\K,\L\M\N}} \\
& \F(\K\L\M\N) & & &
& \F(\K\L\M\N) & 
\ar@{}"a";"c"^(.25){}="a1"^(.75){}="b1" \ar@{=>}^{\alpha_{\K\L,\M,\N}}"a1";"b1"
\ar@{}"B";"C"^(.25){}="b2"^(.75){}="c2" \ar@{=>}_{\alpha_{\K,\L,\M\N}}"b2";"c2"
\ar@{}"2,1";"2,3"^(.25){}="e2"^(.75){}="f2" \ar@{=>}^{\bt_{F_{\K,\L}, F_{\M,\N}}}"e2";"f2"
\ar@{}"A3";"B3"^(.25){}="a1"^(.75){}="b1" \ar@{=>}_{\alpha_{\K,\L,\M}}"a1";"b1"
\ar@{}"a3";"c3"^(.25){}="b2"^(.75){}="c2" \ar@{=>}^{\alpha_{\L,\M,\N}}"b2";"c2"
\ar@{}"4,5";"4,7"^(.25){}="e2"^(.75){}="f2" \ar@{=>}^{\alpha_{\K,\L\M,\N}}"e2";"f2"
}
}  
\end{equation}
and
\begin{equation}
\scalebox{0.7}{
\xymatrix{
& \F(\M) \F(\I) \F(\N) \ar[dl]_{F_{\M,\I}} \ar[dr]^{F_{\I,\N}}& & &
& \F(\M) \F(\I) \F(\N) \ar[dl]_{F_{\M,\I}}_{}="A" \ar[dr]^{F_{\I,\N}}_{}="C" \ar[dd]^{U}_{}="a" & \\
 \F(\M\I) \F(\N) \ar[dd]_{F(R_\M)}_{}="a" \ar[dr]^{F_{\M\I,\N}} && \F(\M)\F(\I\N) \ar[dd]^{\F(L_\N)}^{}="Z" \ar[dl]_{F_{\M,\I\N}} &  &
 \F(\M\I) \F(\N) \ar[dd]_{\F(R_\M)}_{}="b"  && \F(\M)\F(\I\N) \ar[dd]^{\F(L_\N)}_{}="c" \\
& \F(\M\I\N) \ar[dd]_{\F(R_\M)}_{}="f"^{F(L_\N)}_{}="Y" & &  = &
& \F(\M)\I \F(\N) \ar[dl]_{R_{\F(\M)}}^{}="B"  \ar[dr]^{L_{\F(\N)}}^{}="D"  & \\
 \F(\M) \F(\N)  \ar[dr]_{F_{\M,\N}} && \F(\M)\F(\N) \ar[dl]^{F_{\M,\N}} &  &
 \F(\M) \F(\N) \ar[dr]_{F_{\M,\N}} && \F(\K)\F(\M)\F(\N) \ar[dl]^{F_{\M,\N}} \\
& \F(\M\N) & & &
& \F(\M\N) & 
\ar@{}"a";"f"^(.25){}="a1"^(.75){}="b1" \ar@{=>}^{F_{R_\M,\id_N}}"a1";"b1"
\ar@{}"Z";"Y"^(.25){}="b2"^(.75){}="c2" \ar@{=>}_{F_{\id_\M,L_\N}}"b2";"c2"
\ar@{}"2,1";"2,3"^(.25){}="e2"^(.75){}="f2" \ar@{=>}^{\alpha_{\M,\I,\N}}"e2";"f2"
\ar@{}"A";"B"^(.25){}="a1"^(.75){}="b1" \ar@{=>}_{\rho_{\M}}"a1";"b1"
\ar@{}"C";"D"^(.25){}="b2"^(.75){}="c2" \ar@{=>}^{\lambda_{\N}}"b2";"c2"
\ar@{}"4,5";"4,7"^(.35){}="e2"^(.65){}="f2" \ar@{=}^{}"e2";"f2"
}
}
\end{equation}
for all $\K,\L,\M,\N\in\bM$.
\end{definition}

\begin{definition}
A {\em braided monoidal $2$-functor} $\F:\bM\to \bM'$ 
between braided monoidal 2-categories is a monoidal functor along with an invertible modification
\begin{equation}
\label{deltaMN}
\xymatrix{
\F(\M)\F(\N) \ar[dd]_{F_{\M,\N}}_{}="a" \ar[rr]^{B_{F(\M),F(\N)}} && \F(\N)\F(\M) \ar[dd]^{F_{\N,\M}}="b" \\
&& \\
\F(\M\N) \ar[rr]^{F(B_{\M,\N})} && \F(\N \M)
\ar@{}"a";"b"^(.25){}="b2"^(.75){}="c2" \ar@{=>}^{\delta_{\M,\N}}"b2";"c2"
}
\end{equation}
such that 
\begin{equation}
\label{braided2fun1}
\scalebox{0.7}{
\xymatrix{
& \F(\L)\F(\M\N) \ar[dr] &&& 
& 
& \F(\L)\F(\M\N) \ar[dr]  \ar[dddddd]_{}="d"^{}="e" &&& \\
\F(\L)\F(\M)\F(\N) \ar[ur]\ar[dr] \ar[dddddd]_{}="a" \ar[dddr] && \F(\L\M\N) \ar[dddr] &&
& 
\F(\L)\F(\M)\F(\N) \ar[ur]  \ar[dddddd]_{}="f" && \F(\L\M\N) \ar[dddr] \ar[dddddd]_{}="b"^{}="c" && \\
& \F(\L\M)\F(\N) \ar[ur]\ar[dr] &&& 
& 
&                     &&& \\
&& \F(\M\L)\F(\N) \ar[dr] && 
& 
&&     && \\
& \F(\M)\F(\L)\F(\N) \ar[ur]\ar[dr] \ar[dddl] && \F(\M\L\N)  \ar[dddl]
& = 
& &  && \F(\M\L\N)  \ar[dddl] &\\ 
&& \F(\M)\F(\L\N) \ar[ur] \ar[dddl]&& 
& 
&&   && \\
&                   &&& 
& 
& \F(\M\N)\F(\L)  \ar[dr] &&& \\
\F(\M)\F(\N)\F(\L) \ar[dr] && \F(\M\N\L) &&
& 
\F(\M)\F(\N)\F(\L) \ar[dr] \ar[ur]&& \F(\M\N\L) && \\
& \F(\M)\F(\N\L) \ar[ur] &&& 
& 
& \F(\M)\F(\N\L) \ar[ur] &&&
\ar@{}"3,2";"1,2"^(.25){}="a1"^(.75){}="b1" \ar@{=>}_{\alpha_{\L,\M,\N}^{-1}}"b1";"a1"
\ar@{}"4,3";"6,3"^(.25){}="c1"^(.75){}="d1" \ar@{=>}_{\alpha_{\M,\L,\N}}"c1";"d1"
\ar@{}"3,2";"5,2"^(.25){}="e1"^(.75){}="f1" \ar@{=>}_{\delta_{\L,\M}}"e1";"f1"
\ar@{}"2,3";"4,3"^(.25){}="g1"^(.75){}="h1" \ar@{=>}_{F_{{B_{\L,\M},\N}}}"g1";"h1"
\ar@{}"6,3";"8,3"^(.25){}="i1"^(.75){}="j1"  \ar@{=>}_>>>{F_{{\M,B_{\L,\N}}}}"i1";"j1"
\ar@{}"6,3";"8,1"^(.25){}="k1"^(.75){}="l1" \ar@{=>}_{\delta_{\L,\N}}"k1";"l1"
\ar@{=>}_>>>>>>>>>>>>{\beta'_{F(\L)|F(\M),F(\N)}}"5,2";"a"
\ar@{}"7,7";"9,7"^(.25){}="o1"^(.75){}="p1" \ar@{=>}_{\alpha_{\M,\N,\L}}"o1";"p1"
\ar@{=>}_>>>>>>>>>>>>{F(\beta_{\L|\M,\N})}"5,9";"b"
\ar@{}"d";"c"^(.25){}="q1"^(.75){}="r1" \ar@{=>}_{\delta_{\L,\M\N}}"q1";"r1"
\ar@{}"e";"f"^(.25){}="s1"^(.75){}="t1" \ar@{=>}_{B'_{\F(\L),F_{\M,\N}}}"s1";"t1"
}
}
\end{equation}
and
\begin{equation}
\label{braided2fun2}
\scalebox{0.7}{
\xymatrix{
& F(\L\M)F(\N) \ar[dr] &&& 
& 
&  F(\L\M)F(\N) \ar[dr]  \ar[dddddd]_{}="d"^{}="e" &&& \\
F(\L)F(\M)F(\N) \ar[ur]\ar[dr] \ar[dddddd]_{}="a" \ar[dddr] && F(\L\M\N) \ar[dddr] &&
& 
F(\L)F(\M)F(\N) \ar[ur]          \ar[dddddd]_{}="f"                && F(\L\M\N) \ar[dddr] \ar[dddddd]_{}="b"^{}="c" && \\
& F(\L)F(\M\N) \ar[ur]\ar[dr] &&& 
& 
&                     &&& \\
&& F(\L)F(\N\M) \ar[dr] && 
& 
&&     && \\
& F(\L)F(\N)F(\M) \ar[ur]\ar[dr] \ar[dddl] && F(\L\N\M)  \ar[dddl]
& = 
& &  && F(\L\N\M)  \ar[dddl] &\\ 
&& F(\L\N)F(\M) \ar[ur] \ar[dddl]&& 
& 
&&   && \\
&                   &&& 
& 
& F(\N)F(\L\M)  \ar[dr] &&& \\
F(\N)F(\L)F(\M) \ar[dr] && F(\N\L\M) &&
& 
F(\N)F(\L)F(\M) \ar[dr] \ar[ur]&& F(\N\L\M) && \\
& F(\N\L)F(\M) \ar[ur] &&& 
& 
&  F(\N\L)F(\M) \ar[ur] &&&
\ar@{}"3,2";"1,2"^(.25){}="a1"^(.75){}="b1" \ar@{=>}_{\alpha_{\L,\M,\N}}"b1";"a1"
\ar@{}"4,3";"6,3"^(.25){}="c1"^(.75){}="d1" \ar@{=>}_{\alpha_{\L,\N,\M}^{-1}}"c1";"d1"
\ar@{}"3,2";"5,2"^(.25){}="e1"^(.75){}="f1" \ar@{=>}_{\delta_{\M,\N}}"e1";"f1"
\ar@{}"2,3";"4,3"^(.25){}="g1"^(.75){}="h1" \ar@{=>}_{F_{\L,{B_{\N,\M}}}}"g1";"h1"
\ar@{}"6,3";"8,3"^(.25){}="i1"^(.75){}="j1"  \ar@{=>}_>>>{F_{{B_{\L,\N},\M}}}"i1";"j1"
\ar@{}"6,3";"8,1"^(.25){}="k1"^(.75){}="l1" \ar@{=>}_{\delta_{\L,\N}}"k1";"l1"
\ar@{=>}_>>>>>>>>>>>>{\beta'_{F(\L),F(\M)| F(\N)}}"5,2";"a"
\ar@{}"7,7";"9,7"^(.25){}="o1"^(.75){}="p1" \ar@{=>}_{\alpha_{\N,\L,\M}^{-1}}"o1";"p1"
\ar@{=>}_>>>>>>>>>>>>{F(\beta_{\L,\M|\N})}"5,9";"b"
\ar@{}"d";"c"^(.25){}="q1"^(.75){}="r1" \ar@{=>}_{\delta_{\L\M,\N}}"q1";"r1"
\ar@{}"e";"f"^(.25){}="s1"^(.75){}="t1" \ar@{=>}_{B'_{F_{\L,\M},F(\N)}}"s1";"t1"
}
}
\end{equation}
for all objects $\L,\M,\N$ in $\mathbf{M}$. Here $B,\, B'$ denote the braidings on $\bM,\, \bM'$. We omit $1$-cell 
labels to keep the diagrams readable.
\end{definition}

\begin{definition}
A braided monoidal $2$-functor $\F:\mathbf{M}\to \mathbf{M}'$ between sylleptic (respectively, symmetric) monoidal $2$-categories 
$\mathbf{M}$ and $\mathbf{M}'$ is {\em sylleptic} (respectively,  {\em symmetric}) if 
\begin{equation}
\label{sym2fun}
\scalebox{0.9}{
\xymatrix{
\F(\M)\F(\N)  \ar@/^1pc/[rrr]^{B'_{\F(\M),\F(\N)}}\ar[dd]_{F_{\M,\N}}_{}="a"  &&& F(\N) F(\M) \ar[dd]^{F_{\N,\M}}_{}="b" 
& & 
\F(\M)\F(\N) \ar@/^1pc/[rrr]^{B'_{F(\M),F(\N)}}_{}="g"  \ar[dd]_{F_{\M,\N}}_{}="c"  &&& 
\F(\N) \F(\M)  \ar[dd]^{F_{\N,\M}}_{}="d" \ar@/^1pc/[lll]^{B_{\F(\N),\F(\M)}}_{}="h" \\
&&&& = &&&& \\
\F(\M\N)   \ar@/^1pc/[rrr]^{\F(B_{\M,\N})}_{}="e"  &&& \F(\N\M) \ar@/^1pc/[lll]^{\F(B_{\N,\M})}_{}="f"
& & 
\F(\M\N)   &&& \F(\N\M), \ar@/^1pc/[lll]^{\F(B_{\N,\M})}
\ar@{}"a";"b"^(.25){}="a1"^(.75){}="b1" \ar@{=>}^{\delta_{\M,\N}}"a1";"b1"
\ar@{}"c";"d"^(.25){}="c1"^(.75){}="d1" \ar@{=>}^{\delta_{\N,\M}}"d1";"c1"
\ar@{=>}^{F(\tau_{\M,\N})}"e";"f"
\ar@{=>}^{\tau'_{F(\M),F(\N)}}"g";"h"
}
}
\end{equation}
for all $\M,\,\N$ in $\mathbf{M}$. Here $\tau,\, \tau'$ are the modification  defined in \eqref{symmetry 2-cell}. 
\end{definition}

\begin{definition}
\label{mon2trans}
Let $\F, \F':\bM\to\bM'$ be two monoidal functors.
A {\em monoidal  pseudo-natural transformation} $P: \F\to \F'$ is a pseudo-natural transformation along with 
an invertible modification
\begin{equation}
\label{muss}
\xymatrix{
\F(\L)\F(\M) \ar[dd]_{F_{\L,\M}}_{}="a" \ar[rr]^{P_\L P_\M} && \F'(\L)\F'(\M) \ar[dd]^{F'_{\L,\M}}="b" \\
&& \\
\F(\L\M) \ar[rr]^{P_{\L\M}} && \F'(\L \M)
\ar@{}"a";"b"^(.25){}="b2"^(.75){}="c2" \ar@{=>}^{\mu_{\L,\M}}"b2";"c2"
}
\end{equation}
such that 
\begin{equation}
\label{MT cube}
\scalebox{0.7}{
\xymatrix{
 & \F(\L\M)\F(\N) \ar[dr]^{F_{\L\M,\N}} & &&
 & \F(\L\M)\F(\N) \ar[dr]^{F_{\L\M,\N}}^{}="C" \ar[dd]_{P_{\L\M}\P_\N} &\\
\F(\L)\F(\M)\F(\N) \ar[ur]^{F_{\L,\M}}_{}="aa" \ar[dr]_{F_{\M,\N}}^{}="a"  \ar[dd]_{P_\L P_\M P_\N} && \F(\L\M\N) \ar[dd]^{P_{\L\M\N}}^{}="c" &&
\F(\L)\F(\M)\F(\N) \ar[ur]^{F_{\L,\M}}_{}="A" \ar[dd]_{P_\L P_\M P_\N} && \F(\L\M\N) \ar[dd]^{P_{\L\M\N}}  \\
 & \F(\L)\F(\M\N) \ar[ur]_{F_{\L,\M\N}}^{}="bb"_{}="c" \ar[dd]_{P_\L P_{\M\N}} &  &=&
 & \F'(\L\M)\F'(\N)  \ar[dr]_{F'_{\L\M,\N}}^{}="D"  &    \\
 \F'(\L)\F'(\M)\F'(\N)  \ar[dr]_{F'_{\M,\N}}^{}="b"   && \F'(\L\M\N) &&
 \F'(\L)\F'(\M)\F'(\N)  \ar[dr]_{F'_{\M,\N}}  \ar[ur]_{F'_{\L,\M}}^{}="cc"_{}="B"  && \F'(\L\M\N)  \\
& \F'(\L)\F'(\M\N) \ar[ur]_{F'_{\L,\M\N}}^{}="d" & &&
& \F'(\L)\F'(\M\N) \ar[ur]_{F'_{\L,\M\N}}^{}="dd" &
\ar@{}"aa";"bb"^(.25){}="x"^(.75){}="y" \ar@{=>}^{\alpha_{\L,\M,\N}}"x";"y"
\ar@{}"cc";"dd"^(.25){}="x"^(.75){}="y" \ar@{=>}^{\alpha'_{\L,\M,\N}}"x";"y"
\ar@{}"a";"b"^(.25){}="x"^(.75){}="y" \ar@{=>}_{\mu_{\M,\N}}"x";"y"
\ar@{}"c";"d"^(.25){}="x"^(.75){}="y" \ar@{=>}_{\mu_{\L,\M\N}}"x";"y"
\ar@{}"A";"B"^(.25){}="x"^(.75){}="y" \ar@{=>}_{\mu_{\L,\M}}"x";"y"
\ar@{}"C";"D"^(.25){}="x"^(.75){}="y" \ar@{=>}_{\mu_{\L\M,\N}}"x";"y"
}
}
\end{equation}
 for all $\L,\M,\N$ in $\bM$. Here $\alpha$ and $\alpha'$ denote the monoidal structures of $\F$ and $\F'$.
\end{definition}

\begin{definition}
A monoidal pseudo-natural transformation $P: \F\to \F'$ between braided monoidal $2$-functors is {\em braided} if 
\begin{equation}
\label{BT cube}
\scalebox{0.8}{
\xymatrix{
 & \F(\L\M) \ar[dr]^{P_{\L\M}} & &&
 & \F(\L\M) \ar[dr]^{P_{\L\M}}^{}="C" \ar[dd]_{\F(B_{\L,\M})}_{}="g" & \\
 \F(\L)\F(\M) \ar[ur]^{F_{\L,\M}}_{}="aa" \ar[dr]_{P_\L P_\M}^{}="a"  \ar[dd]_{B'_{\F(\L),\F(\M)}}_{}="E"  && \F'(\L\M) \ar[dd]^{\F'(B_{\L,\M})}^{}="c" &&
 \F(\L)\F(\M) \ar[ur]^{F_{\L,\M}}_{}="A" \ar[dd]_{B'_{\F(\L),\F(\M)}}&& \F'(\L\M) \ar[dd]^{\F'(B_{\L,\M})}^{}="h" \\
 & \F'(\L)\F'(\M) \ar[ur]_{F'_{\L,\M}}^{}="bb"_{}="c" \ar[dd]_{B_{\F'(\L),\F'(\M)}}_{}="F"&  &=&
 & \F(\M\L)  \ar[dr]_{P_{\M\L}}^{}="D"  &    \\
 \F(\M)\F(\L)  \ar[dr]_{P_\M P_\L}^{}="b"   && \F'(\M\L) &&
 \F(\M)\F(\L)  \ar[dr]_{P_\M P_\L}  \ar[ur]_{F_{\M,\L}}^{}="cc"_{}="B"  && \F'(\M\L)  \\
& \F'(\M)\F'(\L) \ar[ur]_{F'_{\M,\L}}^{}="d" & &&
& \F'(\M)\F'(\L) \ar[ur]_{F'_{\M,\L}}^{}="dd" &
\ar@{}"aa";"bb"^(.25){}="x"^(.75){}="y" \ar@{=>}^{\mu_{\L,\M}}"x";"y"
\ar@{}"cc";"dd"^(.25){}="x"^(.75){}="y" \ar@{=>}^{\mu_{\M,\L}}"x";"y"
\ar@{}"E";"F"^(.25){}="x"^(.75){}="y" \ar@{=>}^{B'_{P_\L,P_\M}}"x";"y"
\ar@{}"c";"d"^(.25){}="x"^(.75){}="y" \ar@{=>}_{\delta'_{\L,\M}}"x";"y"
\ar@{}"A";"B"^(.25){}="x"^(.75){}="y" \ar@{=>}_{\delta_{\M,\L}}"x";"y"
\ar@{}"g";"h"^(.25){}="x"^(.75){}="y" \ar@{=>}_{P_{B_{\L,\M}}}"x";"y"
}
}
\end{equation}
 for all $\L,\M$ in $\bM$. Here $\delta$ and $\delta'$ denote the braided structures on $F$ and $F'$.
\end{definition}

\begin{definition}
A modification $\eta:P \to Q$ between two monoidal pseudo-natural transformations $P,Q: \F\to \F'$ is {\em monoidal} if 
\begin{equation}
\label{MM cylinder}
\xymatrix{
\F(\M)\F(\N) \ar[r]^{F_{\M,\N}}_{}="A1" \ar@/^-2pc/[dd]_{P_\M P_\N}  & \F(\M\N)  \ar@/^2pc/[dd]^{Q_{\M\N}}_{}="b"   \ar@/^-2pc/[dd]_{P_{\M\N}}^{}="a" & & 
\F(\M)\F(\N) \ar[r]^{F_{\M,\N}}_{}="A2"  \ar@/^-2pc/[dd]_{P_\M P_\N}^{}="A"  \ar@/^2pc/[dd]^{Q_\M Q_\N}_{}="B" & \F(\M\N)  \ar@/^2pc/[dd]^{Q_{\M\N}}^{}  \\
& & = & & \\   
\F'(\M)\F'(\N) \ar[r]_{F'_{\M,\N}}^{}="B1"  & \F'(\M\N) & & \F'(\M)\F'(\N) \ar[r]_{F'_{\M,\N}}^{}="B2"  & \F'(\M\N),
\ar@{}"a";"b"^(.25){}="x"^(.75){}="y" \ar@{=>}^{\eta_{\M\N}}"x";"y"
\ar@{}"A";"B"^(.25){}="x"^(.75){}="y"  \ar@{=>}^{\eta_{\M}\eta_{\N}}"x";"y"
\ar@{}"A1";"B1"^(.25){}="x"^(.75){}="y" \ar@/^-1.5pc/@{=>}_{\mu_{\M,\N}}"x";"y"
\ar@{}"A2";"B2"^(.25){}="x"^(.75){}="y" \ar@/^1.5pc/@{=>}^{\nu_{\M,\N}}"x";"y"
} 
\end{equation}
for all $\M,\, \N \in \bM$, where $\mu$ and $\nu$ are modifications \eqref{muss} defining the  
monoidal structures on $P$ and $Q$, respectively.  

\end{definition}

\subsection{The center of a  monoidal $2$-category}
\label{2-center}

Let $\bM$ be a braided monoidal $2$-category.
Its {\em center} $\mathbf{Z}(\bM)$
is a braided monoidal $2$-category defined as follows \cite[Section 3]{BN}. Objects of  $\mathbf{Z}(\mathbf{M})$
are triples $(\N,\, S,\, \gamma)$ where $\N$ is an object of $\mathbf{M}$, $S$
is a pseudo-natural collection of equivalences (called a {\em half-braiding})
\[
S_\M: \M \N \to \N \M,\qquad \M\in \mathbf{A},
\]
and $\gamma$ is an invertible modification 
\begin{equation}
\label{tilde beta}
\xymatrix{
&& \L \N \M \ar[drr]^{ S_\L } && \\
\L  \M \N \ar[urr]^{\S_\M} \ar[rrrr]_{S_{\L\M}}^{}="b" &&&&  \N \L  \M 
\ar@{}"1,3";"b"^(.25){}="x"^(.75){}="y" \ar@2{->}^{\gamma_{\L,\M}}"x";"y"
}
\end{equation}
such that 
\begin{equation}
\label{central coherence}
\xymatrix{
\K  \L  \N  \M  \ar[dd]_{S_\L}  &&& \K  \L  \M  \N  \ar[lll]_{S_\M}  \ar[dd]^{S_{\K\L\M}}_{}="c"  \ar[ddlll]_{S_{\L\M}}_{}="a"^{}="b"  
&  & 
\K  \L  \N  \M \ar[dd]_{S_\L}   \ar[ddrrr]_{S_{\K\L}}_{}="A"^{}="B" &&& \K  \L  \M  \N   \ar[lll]_{S_{\M}}  \ar[dd]^{S_{\K\L\M}}_{}="C"  \\
&&&& = &&&& \\
\K  \N  \L  \M    \ar[rrr]_{S_{\K}}   &&&  \N  \K  \L  \M  & & \L  \M  \K  \N    \ar[rrr]_{S_{\K}}    &&&  \N  \K  \L  \M, 
\ar@{}"1,1";"a"^(.25){}="a1"^(.75){}="b1" \ar@2{->}^{\gamma_{\L,\M}}"a1";"b1"
\ar@{}"b";"c"^(.25){}="a1"^(.75){}="b1" \ar@2{->}_{\gamma_{\K,\L\M}}"a1";"b1"
\ar@{}"3,6";"A"^(.25){}="a1"^(.75){}="b1" \ar@2{->}^{\gamma_{\K,\L}}"a1";"b1"
\ar@{}"B";"C"^(.25){}="a1"^(.75){}="b1" \ar@2{->}^{\gamma_{\K\L,\M}}"a1";"b1"
}
\end{equation}
for all $\K,\,\L,\,\M\in \mathbf{M}$. 

A morphism between  $(\N,\, S,\, \gamma)$
and $(\N',\, S',\, \gamma')$ in $\mathbf{Z}(\mathbf{M})$ is  pair $(F,\, \sigma)$, where $F:\N\to \N'$
is a morphism in $\mathbf{M}$ and $\sigma$ is an invertible modification
\begin{equation}
\label{delta-def}
\xymatrix{
\M  \N  \ar[rr]^{S_\M} \ar[d]_{F}^{}="a"  && \N  \M \ar[d]^{F}_{}="b" \\
\M  \N'  \ar[rr]_{S'_\M}&&  \N'  \M
\ar@{}"a";"b"^(.25){}="a1"^(.75){}="b1" \ar@2{->}^{\sigma_\M}"a1";"b1"
}
\end{equation}
such that
\begin{equation*}
\scalebox{0.9}{
\xymatrix{
&& \L\N\M \ar[drr]^{S_\L} && 
& &
&& \L\N\M \ar[drr]^{S_\L}  \ar[d]_{F}_{}="B"  &&  \\
\L\M\N \ar[rrrr]_{S_{\L\M}}^{}="A1" \ar[urr]^{S_\M} \ar[d]_{F}_{}="a" &&&& \N\L\M \ar[d]^{F}_{}="b"
&=&
\L\M\N \ar[urr]^{S_\M}  \ar[d]_{F}_{}="A"  && \L\N'\M \ar[drr]_{S'_\L} && \N\L\M \ar[d]^{F}_{}="C"  \\
 \L\M\N' \ar[rrrr]_{S'_{\L\M}}   &&&& \N'\L\M
 & &
 \L\M\N' \ar[rrrr]_{S'_{\L\M}}^{}="A2" \ar[urr]_{S'_\M}  &&&& \N'\L\M.
\ar@{}"a";"b"^(.25){}="a1"^(.75){}="b1" \ar@2{->}_{\sigma_{\L\M}}"a1";"b1"
\ar@{}"A";"B"^(.25){}="a1"^(.75){}="b1" \ar@2{->}^{\sigma_{\M}}"a1";"b1"
\ar@{}"B";"C"^(.25){}="a1"^(.75){}="b1" \ar@2{->}^{\sigma_{\L}}"a1";"b1"
\ar@{}"1,3";"A1"^(.25){}="a1"^(.75){}="b1" \ar@2{->}^{\gamma_{\L,\M}}"a1";"b1"
\ar@{}"2,9";"A2"^(.25){}="a1"^(.75){}="b1" \ar@2{->}^{\gamma'_{\L,\M}}"a1";"b1"
}
}
\end{equation*}
for all $\L,\,\M$ in $\bM$.

A $2$-morphism in $\mathbf{Z}(\mathbf{M})$  between $(F,\, \sigma)$ and $(F',\, \sigma')$ is a $2$-cell 
\begin{equation}
\xymatrix{
\N \ar@/^1pc/[rrrr]^{F}_{}="a"      \ar@/^-1pc/[rrrr]_{F'}^{}="d" &&&&  \N',
\ar@2{->}^{\alpha}"a";"d"
}
\end{equation}
such that 
\begin{equation}
\xymatrix{
\M\N \ar[r]^{S_\M}_{}="A1" \ar@/^-2pc/[dd]_{F}  & \N\M  \ar@/^2pc/[dd]^{F'}_{}="b"   \ar@/^-2pc/[dd]^{F}^{}="a" & & 
\M\N  \ar[r]^{S_\M}_{}="A2"  \ar@/^-2pc/[dd]_{F}^{}="A"  \ar@/^2pc/[dd]_{F'}_{}="B" & \N\M \ar@/^2pc/[dd]^{F'}  \\
& & = & & \\   
\M\N' \ar[r]_{S'_\M}^{}="B1"  & \N'\M & & \M\N' \ar[r]_{S'_\M}^{}="B2"  & \N'\M,
\ar@{}"a";"b"^(.25){}="x"^(.75){}="y" \ar@{=>}^{\alpha\, \id_\M }"x";"y"
\ar@{}"A";"B"^(.25){}="x"^(.75){}="y"  \ar@{=>}^{\id_\M \alpha}"x";"y"
\ar@{}"A1";"B1"^(.25){}="x"^(.75){}="y" \ar@/^-1.5pc/@{=>}_{\sigma_\M}"x";"y"
\ar@{}"A2";"B2"^(.25){}="x"^(.75){}="y" \ar@/^1.5pc/@{=>}^{\sigma'_\M}"x";"y"
} 
\end{equation}
for all $\M$ in $\bM$.

The tensor product in $\mathbf{Z}(\mathbf{M})$ is given by
\begin{equation}
(\N,\, S,\, \gamma) \bt  (\N',\, S',\, \gamma') = (\N\N',\, SS',\, \gamma\gamma'),
\end{equation}
where $\N\N'$ is the tensor product in $\bM$, $(SS)'_\M$ is defined as the composition 
\[
(SS)'_\M : \M \N \N' \xrightarrow{S_\M\, N'} \N \M \N' \xrightarrow{\N\,S'_\M} \N\N'\M, 
\]
and $(\gamma\gamma')_{\L,\M}$ is given by the following composition of $2$-cells
\begin{equation}
\xymatrix{
&& \L\N\N'\M  \ar[dr]^{S_\L}  && \\
& \L\N\M\N' \ar[ur]^{S'_\M}  \ar[dr]_{S_\L} && \N\L\N'\M    \ar[dr]^{S'_\L} & \\
\L\M\N\N' \ar[ur]^{S_\M} \ar[rr]_{S_{\L\M}}^{}="a" && \N\L\M\N' \ar[ur]_{S'_\M} \ar[rr]_{S'_{\L\M}}^{}="b" && \N\N'\L\M
\ar@{}"2,2";"a"^(.25){}="x"^(.75){}="y" \ar@2{->}^{\gamma_{\L,\M}}"x";"y"
\ar@{}"2,4";"b"^(.25){}="x"^(.75){}="y" \ar@2{->}^{\gamma'_{\L,\M}}"x";"y"
\ar@{}"1,3";"3,3"^(.25){}="x"^(.75){}="y" \ar@2{->}^{\bt_{S'_\M,S_\L}}"x";"y"
}
\end{equation}
for all $\L,\M\in \bM$.

The braiding between $(\N,\, S,\, \gamma)$ and   $(\N',\, S',\, \gamma')$ is given by
\begin{equation}
(S'_\N,\, \Sigma): (\N,\, S,\, \gamma) \bt  (\N',\, S',\, \gamma') \to (\N',\, S',\, \gamma')  \bt (\N,\, S,\, \gamma),
\end{equation}
where $\Sigma_\M$ is the following composition  $2$-cell:
\begin{equation*}
\xymatrix{
\M\N\N' \ar[rr]^{S_\M} \ar[d]_{S'_\N} \ar[drr]^{S'_{\M\N}}^{}="a" && \N\M\N' \ar[rr]^{S'_\M}  \ar[drr]_{S'_{\N\M}}_{}="b"  && \N\N'\M \ar[d]^{S'_\N} \\
\M\N\N' \ar[rr]_{S'_\M} && \N'\M\N  \ar[rr]_{S_\M} && \N'\N\M. 
\ar@{}"2,1";"a"^(.35){}="x"^(.85){}="y" \ar@{=>}^{\gamma'_{\M,\N}}"x";"y"
\ar@{}"1,5";"b"^(.35){}="x"^(.85){}="y" \ar@{=>}^{\gamma'_{\N,\M}}"x";"y"
\ar@{}"1,3";"2,3"^(.25){}="x"^(.75){}="y" \ar@{=>}^{S'_{S_\M}}"x";"y"
}
\end{equation*}

Let $\bM$ be a braided monoidal $2$ category with braiding $B_{\M,\N}$ and structure modifications $\beta$ and $ \gamma$
\eqref{betaxyz}.  There is a braided monoidal $2$-functor 
\[
\F:\bM \to \mathbf{Z}(\mathbf{M}) : \N \mapsto (\N,\, B_{-, \N},\, \beta_{-,-|\N})
\]
with $F_{N,N'} : \F(\N)  \F(\N') \to \F(\N\N')$ given by $(\id_{\N\N'},\, \beta_{-\mid \N,\N'})$ and identity $2$-cells $\alpha$ \eqref{alphaLMN}
and $\delta$ \eqref{deltaMN}.  See  \cite{BN, Cr}  for details.

\subsection{$2$-categorical groups}
\lb{2cg}

Recall that a {\em categorical group} is a monoidal category in which  every object is invertible with respect
to the tensor product and each morphism is an isomorphism. 

We call an object $\P$ of a monoidal 2-category $\bM$ {\em invertible} if there is another object $\Q$ together with an equivalence $\P\bt\Q\to\I$,
where $\I$ is the unit object of $\bM$. 
Note that in this case the object $\Q$ is unique up to an equivalence. 

Note that the tensor products $-\bt\P$  and $\P\bt -$ with an invertible object $\P\in\bM$ are 2-autoequivalences of $\bM$.
In particular, each of them defines an equivalence of monoidal categories $\bM(\I,\I)\to \bM(\P,\P)$,
where $\I$ is the unit object of $\bM$. 

\begin{definition}
A {\em  $2$-categorical group} is a monoidal 2-category whose objects are invertible with respect to the tensor product, whose 1-morphisms are equivalences, and whose $2$-cells are isomorphisms. 
\end{definition}

\begin{example}
Let $\bA$ be a  2-category.
Then the  monoidal 2-category $\AUT(\bA)$  of autoequivalences of 
$\bA$ with pseudo-natural equivalences as 1-morphisms and isomorphisms as 2-morphisms is a 2-categorical group.
\end{example}

\begin{example}
Let $\bM$ be a monoidal 2-category.
Then the monoidal 2-category $\INV(\bM)$  of invertible objects in $\bM$ 
with equivalences as 1-morphisms and isomorphisms as 2-morphisms is a 2-categorical group.
\end{example}

Let $\bG$ be a braided 2-categorical group with the tensor product $\bt$ and unit object $\I$. 
Below we discuss some invariants of $\bG$.

Let $\varPi_{\leq 1}(\bG)$ denote the $1$-categorical truncation of $\bG$, i.e. the categorical group whose objects are objects of $\bG$ and morphisms
are isomorphism classes of $1$-cells in $\bG$.  Let  $\varPi_{1\leq}(\bG) =\bG(\I,\, \I)$ be the braided categorical group of autoequivalences of $\I$.
Its braiding is given by the naturality $2$-cells 
\begin{equation}
\xymatrix{
\I \bt \I  \ar[rr]^{\id \bt g} \ar[d]_{f\bt \id}_{}="a" && \I \bt \I \ar[d]^{f\bt \id}^{}="b" \\
\I \bt \I  \ar[rr]_{\id \bt g}  && \I \bt \I 
\ar@{}"a";"b"^(.25){}="x"^(.75){}="y" \ar@{=>}^{\bt_{f,g}}"x";"y"
}
\end{equation}
for all $f,g \in \bG(\I,\, \I)$.

\begin{definition}
The {\em homotopy groups} of $\bG$ are defined as follows.
\begin{itemize}
\item the 0th homotopy group $\pi_0(\bG)$ is the group of equivalence classes of objects of $\bG$,
\item the 1st homotopy group $\pi_1(\bG)$ is the group of isomorphism classes of autoequivalences in $\bG(\I,\, \I)$,
\item the 2nd homotopy group $\pi_2(\bG)$ is the group of automorphisms  of the identity 1-morphism $\id_\I$. 
\end{itemize}
\end{definition}

Since $\varPi_{1\leq}(\bG)$ is braided,  the homotopy groups $\pi_1(\bG), \pi_2(\bG)$ are abelian.

\begin{definition}
The {\em first} and {\em second canonical classes} of $\bG$,
\begin{equation}
\label{canonical classes}
\alpha_\bG\in H^3(\pi_0(\bG),\,\pi_1(\bG))\quad \text{and} \quad q_\bG\in H^3_{br}(\pi_1(\bG),\,\pi_2(\bG))
\end{equation}   
are, respectively,  the associator  of the categorical group $\varPi_{1\leq}(\bG)$ and the braided associator of the braided categorical group $\varPi_{1\leq}(\bG)$.
\end{definition}

\begin{proposition}
\label{prop: Pi0 action on Pi1}
There is a monoidal functor 
\begin{equation}
\label{Pi0 action on Pi1}
a: \varPi_{\leq 1}(\bG) \to \A ut_{br}(\varPi_{1\leq}(\bG))
\end{equation}
canonically defined up to a natural isomorphism.
\end{proposition}
\begin{proof}
For any object $\P$ in $\bG$  the corresponding  autoequivalence $a(\P)$ of $\bG(\I,\, \I)$ is given by composing the 
monoidal equivalence
\[
\bG(\I,\, \I) \to \bG(\P,\, \P) : f \mapsto f\bt \P,\, \alpha \mapsto \alpha\bt \P
\]
with the quasi-inverse of 
\[
\bG(\I,\, \I) \to \bG(\P,\, \P) : f \mapsto \P\bt f,\, \alpha \mapsto \P \bt \alpha. 
\]
That $a(\P)$ is a braided autoequivalence  and that $a$ is a monoidal functor follow  the naturality properties of $\bt$.
\end{proof}

\begin{remark}
The  action \eqref{Pi0 action on Pi1} of $\varPi_{\leq 1}(\bG)$ on $\varPi_{1\leq}(\bG)$ 
can also be recovered from the {\em adjoint action} of the 2-categorical group $\bG$
on itself, i.e. a monoidal 2-functor $\AD:\bG\to \AUT(\bG)$ characterized (up to an equivalence) by a coherent collection of equivalences $\P\bt\X \to \AD_\P(\X)\bt\P$, pseudo-natural in $\X\in\bG$. 
\end{remark}

The action \eqref{Pi0 action on Pi1} yields canonical group homomorphisms
\[
\pi_0(\bG) \to \Aut(\pi_1(\bG)),\quad \pi_0(\bG) \to \Aut(\pi_2(\bG)),\quad \text{and} \quad \pi_1(\bG) \to \Hom(\pi_1(\bG),\, \pi_2(\bG))
\]
corresponding, respectively, to the actions of objects of $\varPi_{\leq 1}(\bG)$ on objects and morphisms of $\varPi_{\leq 1}(\bG)$
and to the action of the group of automorphisms of $\I$ by automorphisms of $\id_{\varPi_{\leq 1}(\bG)}$.   We 
will refer to  the corresponding maps
\begin{equation}
\label{wb}
\pi_0(\bG) \times \pi_1(\bG) \to \pi_1(\bG),\quad  \pi_0(\bG) \times \pi_2(\bG) \to \pi_2(\bG),
\quad \text{and} \quad \pi_1(\bG) \times \pi_1(\bG) \to \pi_2(\bG)
\end{equation}
as the {\em Whitehead brackets}. 

Note that the first canonical class $\alpha_\bG$ is invariant with respect to the action of $\pi_0(\bG)$ and 
that the bimultiplicative pairing $\pi_1(\bG) \times \pi_1(\bG) \to \pi_2(\bG)$ is given by the polarization of the second canonical class $q_\bG$. 

Suppose that $\bG$ is a braided $2$-categorical group. In this case $\varPi_{\leq 1}(\bG)$ is a braided categorical
group and $\varPi_{1\leq}(\bG)$ is a symmetric categorical group. Hence, the canonical classes \eqref{canonical classes}
get promoted to
\begin{equation}
\label{braided canonical classes}
\alpha_\bG\in H^3_{br}(\pi_0(\bG),\,\pi_1(\bG))\quad \text{and} \quad q_\bG\in H^3_{sym}(\pi_1(\bG),\,\pi_2(\bG)).
\end{equation}
The braiding of $\bG$ gives a trivialization of the functor \eqref{Pi0 action on Pi1} which implies
that Whitehead brackets \eqref{wb} are trivial and yields a new bilinear pairing
\begin{equation}
\label{new pairing}
[\, , \, ]: \pi_0(\bG)\times \pi_1(\bG)\to \pi_2(\bG)
\end{equation}
constructed as follows.  For each object $\P$ in $\bG$, or an element of $\pi_0(\bG)$,  we have a canonical monoidal automorphism of 
$a(\P)=\id_\I$, i.e. a homomorphism $\pi_1(\bG) \to \pi_2(\bG)$. This gives a homomorphism $\pi_0(\bG) \to \Hom (\pi_1(\bG),\, \pi_2(\bG))$
identified with \eqref{new pairing}.

For a symmetric 2-categorical group $\bG$ the canonical classes are 
\begin{equation}
\label{symmetric canonical classes}
\alpha_\bG\in H^3_{sym}(\pi_0(\bG),\,\pi_1(\bG))\quad \text{and} \quad q_\bG\in H^3_{sym}(\pi_1(\bG),\,\pi_2(\bG)).
\end{equation}
All Whitehead brackets are trivial in this case.

\subsection{Monoidal $2$-functors between 2-categorical groups}
\label{Sect Functors and cohomology}

Let $G$ be a group. We consider it as a 2-categorical group with identity 1- and 2-morphisms.

Let  $\bG$ be a $2$-categorical group (viewed as a semistrict monoidal $2$-category) with the corresponding  canonical classes 
\[
\alpha_\bG\in H^3(\pi_0( \bG),\, \pi_1( \bG))
\quad\text{and}\quad
(\omega_\bG,\, c_\bG) \in H^3_{br}(\pi_1( \bG),\, \pi_2( \bG)).
\]

Let  $C: G \to \varPi_{\leq 1}(\bG): x\to \C_x$ be a monoidal functor. This means that there are $0$-cells 
$\C_x$ in $\bG$, $1$-isomorphisms $M_{x,y}: \C_x \C_y \to \C_{xy}$,
and invertible $2$-cells
\begin{equation}
\label{alpha fgh}
\xymatrix{
\C_x \C_y \C_z \ar[dd]_{M_{x,y}}_{}="a" \ar[rr]^{M_{y,z}} &&  \C_x \C_{yz} \ar[dd]^{M_{x,yz}}_{}="b" \\
&& \\
 \C_{xy} \C_z \ar[rr]^{M_{xy,z}} &&  \C_{xyz},
\ar@{}"a";"b"^(.25){}="a1"^(.75){}="b1" \ar@{=>}^{\alpha_{x,y,z}}"a1";"b1"
}
\end{equation}
for all $x,y,z\in G$. 

Note that $C$ gives rise to an action  of $G$ on $\pi_1(\bG)$ (obtained by composing the underlying group homomorphism
$G\to \pi_0(\bG)$ with the action of  $\pi_0(\bG)$ on $\pi_1(\bG)$). We denote this action by $(g,\,Z)\mapsto Z^g$.  

Following \cite{ENO} define a $4$-cochain
$\mathcal{p}^0_C: G^4\to \pi_2(\bG)$ by setting $\mathcal{p}^0_C(x,y,z,w)$
to be the composition of  faces of the following cube:
\begin{equation}
\label{cubes fghk}
\xymatrix{
&&& \C_x  \C_y  \C_z  \C_w  \ar[ddll]_{M_{x,y}} \ar@{-->}[dd]_{M_{y,z}}="bb" \ar[ddrr]^{M_{z,w}}  &&  \\
\\
&\C_{xy}  \C_z  \C_w \ar[dd]_{M_{xy,z}}="c" \ar[ddrr]_>>>>>>>>>>>>{M_{z,w}}  
&&    \C_{x}   \C_{yz}  \C_w  \ar@{-->}[ddll]_<<<<<<<<<<<{M_{x,yz}}="a" \ar@{-->}[ddrr]^<<<<<<<<<<{M_{yz,w}}_{}="b"
&& \C_x \C_y \C_{zw}  \ar[ddll]_>>>>>>>>>>{M_{x,y}} \ar[dd]^{M_{y,zw}}_{}="d"\\
 \mathcal{p}^0_C(x,y,z,w) =&&&&&\\
&\C_{xyz}  \C_w  \ar[ddrr]_{M_{xyz,w}}="e"
&& \C_{xy}  \C_{zw} \ar[dd]^{M_{xy,zw}}="x"
&& \C_{x}  \C_{yzw} \ar[ddll]^<<<<<<<<<<<{M_{x,yzw}}="y"\\
\\
&&& \C_{xyzw}, &&
\ar@{}"a";"y"^(.15){}="a2"^(.45){}="y2" \ar@2{-->}^{\alpha_{x,yz,w}}"a2";"y2" %
\ar@{}"c";"bb"^(.35){}="c1"^(.75){}="bb1" \ar@2{-->}^{\alpha_{x,y,z}}"bb1";"c1" %
\ar@{}"bb";"d"^(.25){}="b2"^(.75){}="d2" \ar@2{-->}^{\alpha_{y,z,w}}"b2";"d2" %
\ar@{}"c";"x"^(.35){}="e2"^(.85){}="x2" \ar@2{->}_{\alpha_{xy,z,w}}"e2";"x2" %
\ar@{}"x";"d"^(.15){}="x3"^(.65){}="y3" \ar@2{->} ^{\alpha_{x,y,zw}}"y3";"x3"
} 
\end{equation}
where the top face is given by $\bt_{M_{x,y},M_{z,w}}$.  That is, we view $\mathcal{p}^0_C(x,y,z,w)$ 
as a $2$-automorphism of the composition of morphisms between opposite corners, e.g., of $M_{xyz,w}M_{xy,z}M_{x,y},
\,x,y,z,w\in G$
(the $2$-automorphisms of other compositions are conjugate to this one).  We use this convention for other polytopes
in this paper.

\begin{proposition}
$\mathcal{p}^0_C$ is a $4$-cocycle whose cohomology class in  $H^4(G,\, \pi_2(\bG))$ depends only on
the isomorphism class of $C$.  A monoidal functor $C:G \to \varPi_{\leq 1}(\bG)$ extends to a monoidal $2$-functor $G\to \bG$
if and only if $\mathcal{p}^0_C=0$ in   $H^4(G,\, \pi_2(\bG))$.
\end{proposition}
\begin{proof}
Consider the following  polytope (its planar projection is pictured):  
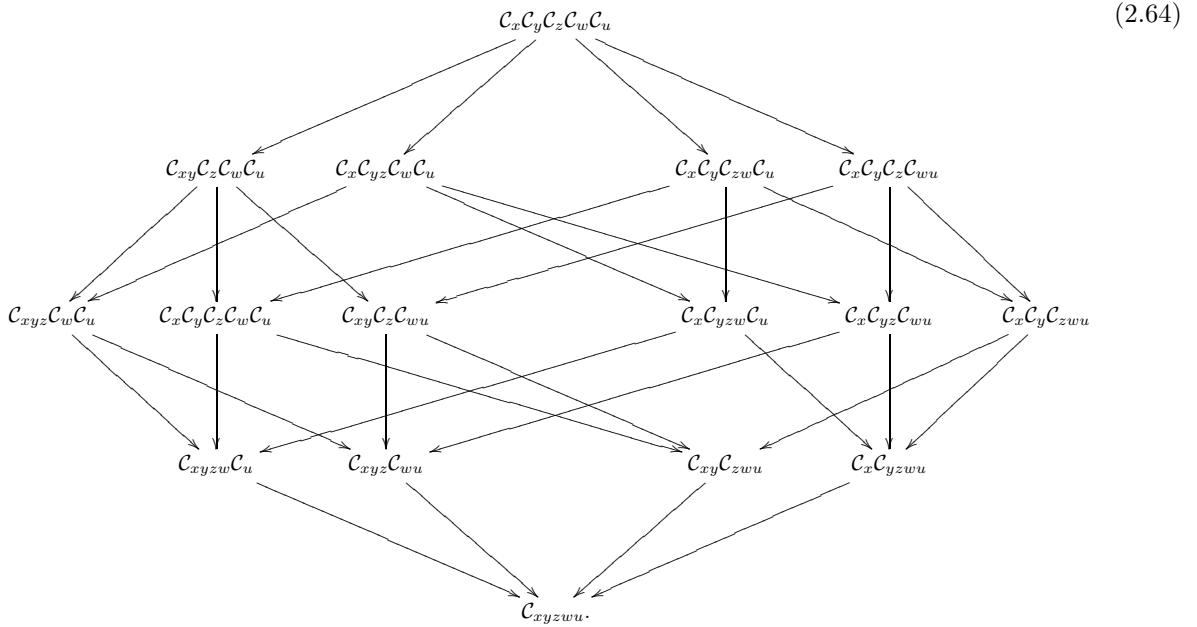
\begin{equation}
\label{15vert}
\scalebox{0.8}{
\xymatrix{
&&& \C_x \C_y  \C_z  \C_w  \C_u  \ar[ddll] \ar[ddl] \ar[ddr] \ar[ddrr] &&& \\
\\
& \C_{xy}  \C_z  \C_w  \C_u \ar[ddl] \ar[dd] \ar [ddr] 
& \C_x \C_{yz}  \C_w  \C_u  \ar[ddll] \ar[ddrr] \ar[ddrrr] 
&&  \C_x \C_y  \C_{zw}  \C_u \ar[ddlll] \ar[dd] \ar[ddrr] 
& \C_x \C_y  \C_z  \C_{wu}   \ar[ddlll] \ar[dd] \ar[ddr] & \\
\\
\C_{xyz}  \C_w  \C_u \ar[ddr] \ar[ddrr]
&  \C_x\C_y  \C_z\C_w  \C_u \ar[dd] \ar[ddrrr]
& \C_{xy}  \C_z \C_{wu}  \ar[dd] \ar[ddrr]
&& \C_x \C_{yzw}  \C_u \ar[ddlll] \ar[ddr]
& \C_x \C_{yz} \C_{wu} \ar[dd] \ar[ddlll]
& \C_x  \C_y  \C_{zwu} \ar[ddl] \ar[ddll]\\
\\
& \C_{xyzw}  \C_u \ar[ddrr]
& \C_{xyz}  \C_{wu} \ar[ddr]
&& \C_{xy}  \C_{zwu} \ar[ddl]
& \C_x \C_{yzwu} \ar[ddll] & \\
\\
&&& \C_{xyzwu}. &&&
}
}
\end{equation}
The edges of this polytope are isomorphisms $M_{x,y}$. The faces are cells  $\alpha_{x,y,z},\, x,y,z\in G,$
\eqref{alpha fgh} and $\bt_{f,g}$. The polytope \eqref{15vert} consists of $8$ cubes (four containing the top
vertex and four containing the bottom one) glued together in such a way that each of their $48$ faces belongs to exactly
two cubes (so that the boundary is empty).  Six of these cubes  are of the form \eqref{cubes fghk}; their composition is the differential 
of $\mathcal{p}^0_C$.  The remaining two cubes commute due to the naturality of the tensor product in $\bG$. Namely,
$M_{x,y}$ commutes with the $2$-cell $\alpha_{z,w,u}$ and $M_{w,u}$ commutes with the $2$-cell $\alpha_{x,y,z}$.
Thus, $d(\mathcal{p}^0_C)=1$.

A different choice of $2$-cells \eqref{alpha fgh}  results in multiplying $\mathcal{p}^0_C$ by a $4$-coboundary, so its class in 
$H^4(G,\, \pi_2(\bG))$ is well-defined.

Finally, $C$ extends to a monoidal $2$-functor if the $2$-cells \eqref{alpha fgh} can be chosen in such a way that 
\eqref{mon2fun} is satisfied. This is equivalent to commutativity of the cube \eqref{cubes fghk} 
(the latter is obtained by gluing the two sides of \eqref{mon2fun}), i.e. to $\mathcal{p}^0_C$  being cohomologically trivial. 
\end{proof}

For $L\in H^2(G,\, \pi_1(\bG))$ the monoidal functor  $L\cdot C: G \to \varPi_{\leq 1}(\G)$ is obtained by multiplying $M_{x,y}$
by $L_{x,y}$ for all $x,y\in G$.

Let $C:G \to \varPi_{\leq 1}(\bG)$ be a monoidal functor with the monoidal structure $M_{x,y}:\C_x\C_y \to \C_{xy},\,x,y\in G$. 

The group $\Aut(C)$ of automorphisms of $C$ is isomorphic to $H^1(G,\, \pi_1(\bG))$.
Explicitly, $P\in \Aut(C)$ corresponds to a  collection of 
equivalences $P_x:\C_x\to \C_x$ such that there are invertible $2$-cells
\begin{equation}
\label{muss xy}
\xymatrix{
\C_x\C_y \ar[dd]_{M_{x,y}}_{}="a" \ar[rr]^{P_x P_y} && \C_x\C_y \ar[dd]^{M_{x,y}}="b" \\
&& \\
\C_{xy} \ar[rr]^{P_{xy}} && \C_{xy}
\ar@{}"a";"b"^(.25){}="b2"^(.75){}="c2" \ar@{=>}^{\mu_{x,y}}"b2";"c2"
}
\end{equation}
for all $x,y\in G$.

Suppose that a monoidal functor $C:G \to \varPi_{\leq 1}(\bG)$  extends to a monoidal $2$-functor $\C: G\to \bG$.
That is, there is a choice of invertible $2$-cells \eqref{alpha fgh} such that the cubes \eqref{cubes fghk} commute,
i.e. $\mathcal{p}^0_C=1$. Let $P$ be a monoidal automorphism of $C$.

Define a  function $\mathcal{p}^1_C(P): G^3\to \pi_2(\bG)$ by 
\begin{equation}
\label{cubes for PW1}
\xymatrix{
&&& \C_x  \C_y  \C_z    \ar[ddll]_{M_{x,y}} \ar@{-->}[dd]_{P_xP_yP_z}="bb" \ar[ddrr]^{M_{y,z}}  &&  \\
\\
&\C_{xy}  \C_z   \ar[dd]_{P_{xy}P_z}="c" \ar[ddrr]_>>>>>>>>>>>> {M_{xy,z}}
&&    \C_{x}   \C_y  \C_z  \ar@{-->}[ddll]_<<<<<<<<<<<{M_{x,y}}="a" \ar@{-->}[ddrr]^<<<<<<<<<<{M_{y,z}}_{}="b"
&& \C_x \C_{yz}  \ar[ddll]^>>>>>>>>>>{M_{x,yz}} \ar[dd]^{P_xP_{yz}}_{}="d"\\
 \mathcal{p}^1_C(P)(x,y,z) = &&&&&& \\
&\C_{xy}  \C_z  \ar[ddrr]_{M_{xy,z}}="e"
&& \C_{xyz} \ar[dd]^{P_{xyz}}="x"
&& \C_{x}  \C_{yz} \ar[ddll]^<<<<<<<<<<<{M_{x,yz}}="y"\\
\\
&&& \C_{xyz}, &&
\ar@{}"a";"y"^(.15){}="a2"^(.45){}="y2" \ar@2{-->}^{\alpha_{x,y,z}}"a2";"y2" %
\ar@{}"c";"bb"^(.35){}="c1"^(.75){}="bb1" \ar@2{-->}^{\mu_{x,y}}"bb1";"c1" %
\ar@{}"bb";"d"^(.25){}="b2"^(.75){}="d2" \ar@2{-->}^{\mu_{yz,w}}"b2";"d2" %
\ar@{}"c";"x"^(.35){}="e2"^(.85){}="x2" \ar@2{->}_{\mu_{xy,z}}"e2";"x2" %
\ar@{}"x";"d"^(.15){}="x3"^(.65){}="y3" \ar@2{->} ^{\mu_{x,yz}}"y3";"x3"
} 
\end{equation}
for all $x,y,z\in G$.  Here the top and bottom faces are $\alpha_{x,y,z}$.

\begin{proposition}
\label{p1CP prop}
$\mathcal{p}^1_C(P)$ is a $3$-cocycle and the map
\begin{equation}
\label{p1(C)}
\mathcal{p}^1_C: \Aut(C) =H^1(G,\, \pi_1(\bG)) \to H^3(G,\, \pi_2(\bG)) : P \mapsto \mathcal{p}^1_C(P)
\end{equation}
is a well defined homomorphism.  The automorphism $P$ extends to a monoidal pseudo-natural automorphism of $\C$
if and only if  $\mathcal{p}^1_C(P)=0$ in $H^3(G,\, \pi_2(\bG))$.
\end{proposition}
\begin{proof}
Consider the following  polytope (its planar projection is pictured):  
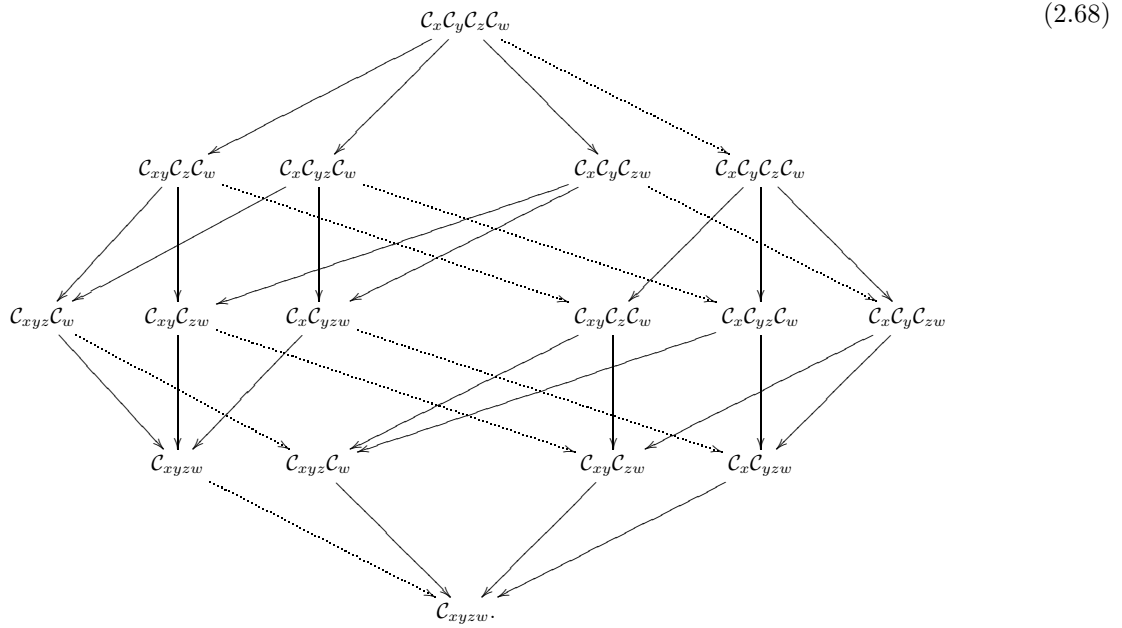
\begin{equation}
\label{15vert'}
\scalebox{0.8}{
\xymatrix{
&&& \C_x \C_y  \C_z  \C_w    \ar[ddll] \ar[ddl] \ar[ddr] \ar@{.>}[ddrr] &&& \\
\\
& \C_{xy}  \C_z  \C_w   \ar[ddl] \ar[dd] \ar@{.>}[ddrrr] 
& \C_x \C_{yz}  \C_w    \ar[ddll] \ar[dd] \ar@{.>}[ddrrr] 
&&  \C_x \C_y  \C_{zw}  \ar[ddlll] \ar[ddll] \ar@{.>}[ddrr] 
& \C_x \C_y  \C_z  \C_w   \ar[ddl] \ar[dd] \ar[ddr] & \\
\\
\C_{xyz}  \C_w   \ar[ddr] \ar@{.>}[ddrr]
&  \C_{xy}  \C_{zw}   \ar[dd] \ar@{.>}[ddrrr]
& \C_x  \C_{yzw}  \ar[ddl] \ar@{.>}[ddrrr]
&& \C_{xy}  \C_z  \C_w \ar[ddll] \ar[dd]
& \C_x \C_{yz}  \C_w   \ar[ddlll] \ar[dd]
&  \C_x \C_y  \C_{zw}  \ar[ddll] \ar[ddl]\\
\\
& \C_{xyzw}  \ar@{.>}[ddrr]
& \C_{xyz}  \C_w \ar[ddr]
&& \C_{xy}  \C_{zw} \ar[ddl]
& \C_x \C_{yzw} \ar[ddll] & \\
\\
&&& \C_{xyzw}. &&&
}
}
\end{equation}
The solid arrows are isomorphisms $M_{x,y}$ and the dotted ones are products of isomorphisms $P_x$. 
The faces are cells  $\alpha_{x,y,z}$
\eqref{alpha fgh}, $\mu_{x,y}$ \eqref{muss xy},
and $\bt_{x,y},\, x,y,z\in G$. 

The polytope \eqref{15vert'} consists of $8$ cubes (four containing the top
vertex and four containing the bottom one) glued together in such a way that each of their $48$ faces belongs to exactly
two cubes (so that the boundary is empty).  Five of these cubes  are of the form \eqref{cubes for PW1}; their composition is  the differential 
of $\mathcal{p}^1_C(P)$.  Two cubes consisting of solid arrows are the cubes \eqref{cubes fghk} and so they commute by assumption.
The remaining cube commutes due to the naturality of the tensor product of $\bG$. 
Thus, $d(\mathcal{p}^1_C(P))=1$.

A different choice of $2$-cells \eqref{muss xy}  results in multiplying $\mathcal{p}^1_C(P)$ by a coboundary, so its class in 
$H^3(G,\, \pi_2(\bG))$ is well-defined.

The equality 
\[
\mathcal{p}^1_C(PQ) = \mathcal{p}^1_C(P) \mathcal{p}^1_C(Q),\qquad  P,Q\in \Aut(C)
\]
is proved directly by gluing two cubes \eqref{cubes for PW1}  for $P$ and $Q$
along the face $\alpha_{x,y,z}$. 

Finally, $P$ extends to a monoidal pseudo-natural automorphism of $\C$ if $2$-cells \eqref{muss xy} can be chosen
in such a way that \eqref{MT cube} is satisfied. This is equivalent
to commutativity of the cube \eqref{cubes for PW1},  i.e. to $\mathcal{p}^1_C(P)$  being cohomologically trivial. 
\end{proof}

Let $\C:G \to \bG$ be a monoidal $2$-functor. For any $\omega \in Z^3(G,\, \pi_2(\bG))$ let $\C^\omega$ 
be a monoidal $2$-functor obtained from $\C$ by multiplying each $2$-cell $\alpha_{x,y,z}$ by $\omega(x,\,y,\,z),\,
x,y,z\in G$. The monoidal $2$-equivalence  class of $\C^\omega$ depends only on the cohomology class of $\omega$ in $H^3(G,\, \pi_2(\bG))$.
If $\C,\, \C'$ are extensions of the same monoidal functor  $C:G \to \varPi_0(\bG)$ if and only if 
$\C'\cong \C^\omega$  for some $\omega$. 

\begin{corollary}
\label{cokernel torsor}
Monoidal $2$-functors $\C^{\omega_1},\, \C^{\omega_2}: G \to \bG$  are isomorphic if and only if  
$\omega_2 = \mathcal{p}^1_C(P) \omega_1$ for some $P\in \Aut(C)=H^1(G,\, \pi_1(\bG))$.  
\end{corollary}
\begin{proof}
Let $\alpha_{x,y,z},\, x,y,z\in G,$ be the cells \eqref{alpha fgh} for $\C$.  A monoidal pseudo-natural isomorphism
between $\C^{\omega_1}$ and $\C^{\omega_2}$ consists of $1$-automorphisms $P_x: \C_x \to \C_x$ such that the cube 
\eqref{cubes for PW1} (with the top and bottom faces being, respectively, $\omega_1(x,y,z)\alpha_{x,y,z}$ and 
$\omega_2(x,y,z)\alpha_{x,y,z}$) commutes.  This is equivalent to $\omega_2/\omega_1 = \mathcal{p}^1_C(P)$,
where $C: G \to \varPi_{\leq 1}(\bG)$  is the underlying monoidal functor of $\C$. 
\end{proof}

\begin{example}
Let $I:G \to \varPi_{\leq 1}(\bG) : x\mapsto \I$ denote the trivial monoidal functor.  Then 
\begin{equation}
\label{p1 formula}
\mathcal{p}^1_I(P)(x,y,z)= \omega_\bG(P_x,\, P_y,\, P_z),\qquad x,y,z\in G.
\end{equation}
\end{example}

The next Corollary summarizes  our description of monoidal $2$-functors $G \to \bG$.

\begin{corollary}
\label{summary monoidal fun}
Let $C:G \to \varPi_0(\bG)$ be a monoidal functor. An extension of $C$ to a  monoidal $2$-functor $\C: G \to \bG$ exists if and only 
if $ \mathcal{p}^0_C=0$ in $H^4(G,\, \pi_2(\bG))$. 
Equivalence classes  of such extensions of $C$ form a torsor over 
$\Coker \left( \mathcal{p}^1_C: H^1(G,\, \pi_1(\bG))  \to H^3(G,\, \pi_2(\bG))   \right)$.
\end{corollary}

\subsection{Braided monoidal $2$-functors between 2-categorical groups}
\label{Sect Braided functors and cohomology}

Let $A$ be an Abelian group. We consider it as a 2-categorical group with identity 1- and 2-morphisms.

Let  $\bG$ be a braided $2$-categorical group with the corresponding  canonical classes 
\[
\alpha_\bG\in H^3_{br}(\pi_0( \bG),\, \pi_1( \bG))
\quad\text{and}\quad
q_\bG =(\omega_\bG,\, c_\bG) \in H^3_{sym}(\pi_1( \bG),\, \pi_2( \bG)).
\]

Let  $C: A \to \varPi_{\leq 1}(\bG): x\to \C_x$ be a braided monoidal functor. This means that there is a $0$-cell 
$\C_x$ in $\bG$ for  each $x\in G$, $1$-isomorphisms $M_{x,y}: \C_x \C_y \to \C_{xy}$,
invertible $2$-cells $\alpha_{x,y,z}$ \eqref{alpha fgh}, and invertible $2$-cells
\begin{equation}
\label{deltas}
\xymatrix{
\C_x\bt_\B \C_y  \ar[rrrr]^{B_{x,y}}="a"  \ar[drr]_{M_{x,y}} &&&& \C_y\bt_\B \C_x \ar[dll]^{M_{y,x}}_{}="b" \\
&& \C_{xy} &&
\ar@{}"2,3";"a"^(.25){}="x"^(.75){}="y" \ar@2{->}^{\delta_{x,y}}"x";"y"
}
\end{equation}
for all $x,y,z\in G$. Let $\beta_{x|y,z}$ and $\beta_{x,y|z}$ denote the invertible modifications \eqref{betaxyz} 
with $\L=\C_x,\, \M=\C_y,\ \N=\C_z$.

Define a braided $4$-cochain $\mathcal{p}^0_C\in C^4_{br}(A,\, \pi_2(\bG))$ by taking
$\mathcal{p}_0^{br}(C)(x,y,z,w)$ from \eqref{cubes fghk},
\begin{equation}
\label{oct theta}
\xymatrix{
& \C_{x} \bt_\B \C_y \bt_\B \C_z  \ar@/^-8pc/[ddddrrrr]_{B_{x\bt y,z}}^<<<<<<<{}="c"^{}="l"^(.20){}="X"
\ar[rr]^{M_{y,z}}^{}="b"  \ar[ddr]_{M_{x,y}}="g" \ar@/^3pc/[rrrr]^{B_{y,z}}_{}="a" &&
\C_x \bt_\B \C_{yz} \ar[dd]^{M_{x,yz}}_{}="d" && 
\C_x\bt_\B \C_z\bt_\B \C_y \ar[ll]_{M_{z,y}} \ar[dd]_{M_{x,z}}^{}="e" \ar@/^4pc/[dddd]^{B_{x,z}}_{}="f" \\
\\
\mathcal{p}^0_C(x,y|z) = &&\C_{xy}\bt_\B \C_z \ar[r]^{M_{xy,z}} \ar[ddr]_{B_{xy,z}}="m"_{}="m1" &
\C_{xyz} && 
\C_{xz} \bt_\B \C_y \ar[ll]_{M_{xz,y}}="h" \\
\\
&&& \C_z\bt_\B \C_{xy}  \ar[uu]_{M_{z,xy}}="k" &&
\C_{z} \bt_\B \C_x \bt_\B \C_y,  \ar[ll]_{M_{x,y}} \ar[uu]^{M_{z,x}}="s"
\ar@{}"1,4";"a"^(.25){}="x"^(.75){}="y" \ar@2{->}_{\delta_{y,z}}"x";"y" 
\ar@{}"a";"X"^(.1){}="x"^(.9){}="y"  \ar@2{-->}_{\beta_{x,y|z}}"x";"y"
\ar@{}"3,6";"f"^(.45){}="x"^(.95){}="y" \ar@2{->}^{\delta_{x,z}}"x";"y"
\ar@{}"3,4";"m1"^(.25){}="x"^(.95){}="y" \ar@2{->}_{\delta_{xy,z}}"x";"y"
\ar@{}"g";"d"^(.45){}="x"^(.95){}="y"  \ar@2{->}^{\alpha_{x,y,z}}"x";"y"
\ar@{}"e";"d"^(.25){}="x"^(.75){}="y" \ar@2{->}_{\alpha_{x,z,y}}"x";"y"
\ar@{}"s";"k"^(.25){}="x"^(.75){}="y" \ar@2{->}_{\alpha_{z,x,y}}"x";"y"
\ar@{}"l";"m"^(.15){}="x"^(.85){}="y" \ar@2{->}^{B_{M_{x,y},\C_z}}"x";"y"
}
\end{equation}
and
\begin{equation}
\label{oct psi}
\xymatrix{
& \C_{x} \bt_\B \C_y \bt_\B \C_z  \ar@/^-8pc/[ddddrrrr]_{B_{x, y\bt z}}^<<<<<<<{}="c"^{}="l"^(.20){}="X"
\ar[rr]^{M_{x,y}}^(0.75){}="b"  \ar[ddr]_{M_{y,z}} \ar@/^3pc/[rrrr]^{B_{x,y}}_{}="a" &&
\C_{xy} \bt_\B \C_{z} \ar[dd]^{M_{xy,z}}_{}="d" && 
\C_y\bt_\B \C_x\bt_\B \C_z \ar[ll]_{M_{y,x}}="s" \ar[dd]_{M_{x,z}}^{}="e" \ar@/^4pc/[dddd]^{B_{x,z}}_{}="f" \\
\\
\mathcal{p}^0_C(x|y,z) =&&\C_{x}\bt_\B \C_{yz} \ar[r]^{M_{x,yz}}="g" \ar[ddr]_{B_{x,yz}}="m"_{}="m1" &
\C_{xyz} && 
\C_{y} \bt_\B \C_{xz} \ar[ll]_{M_{y,xz}}="h" \\
\\
&&& \C_{yz}\bt_\B \C_{x},  \ar[uu]_{M_{yz,x}}="k" &&
\C_{y} \bt_\B \C_z \bt_\B \C_x.  \ar[ll]_{M_{y,z}}="t" \ar[uu]^{M_{z,x}}
\ar@{}"1,4";"a"^(.25){}="x"^(.75){}="y" \ar@2{->}_{\delta_{x,y}}"x";"y" 
\ar@{}"a";"X"^(.1){}="x"^(.9){}="y"  \ar@2{-->}_{\beta_{x|y,z}}"x";"y"
\ar@{}"3,6";"f"^(.45){}="x"^(.95){}="y" \ar@2{->}^{\delta_{x,z}}"x";"y"
\ar@{}"3,4";"m1"^(.25){}="x"^(.95){}="y" \ar@2{->}_{\delta_{x,yz}}"x";"y"
\ar@{}"b";"g"^(.25){}="x"^(.75){}="y" \ar@2{->}_{\alpha_{x,y,z}}"x";"y"
\ar@{}"s";"h"^(.25){}="x"^(.75){}="y"  \ar@2{->}^{\alpha_{y,x,z}}"x";"y"
\ar@{}"t";"h"^(.25){}="x"^(.75){}="y"  \ar@2{->}_{\alpha_{y,z,x}}"x";"y"
\ar@{}"l";"m"^(.15){}="x"^(.85){}="y" \ar@2{->}^{B_{\C_x,M_{y,z}}}"x";"y"
}
\end{equation}
where the plane projections of the octahedra are pictured.

\begin{remark}
The octahedra \eqref{oct theta} and \eqref{oct psi} are special cases of those from the definition of
a braided pseudomonoid in a braided Gray monoid  \cite[Definition 13]{DS}.
\end{remark}

\begin{proposition}
\label{br obstruction}
$\mathcal{p}^0_C$
is a braided $4$-cocycle whose cohomology class in $H^4_{br}(A,\, \pi_2(\bG))$  depends only on the isomorphism class of $C$.  
A braided monoidal functor $C: A\to \varPi_{\leq 1}(\bG)$ extends to a braided monoidal $2$-functor
$A\to \bG$ if and only if  $\mathcal{p}^0_C =0$ in $H^4_{br}(A,\, \pi_2(\bG))$.
\end{proposition} 
\begin{proof}
We need to verify vanishing of the shuffle differentials \eqref{d5} -- \eqref{d3}.  
That the differential \eqref{d5} is zero follows from the construction of polytope \eqref{15vert}. 
Vanishing of  the differentials \eqref{d13} (respectively,  \eqref{d31}, \eqref{d22}, and \eqref{d3}) is proved in a similar way. 
Namely, we form polytopes by gluing the octahedra \eqref{oct theta}, \eqref{oct psi} to cubes \eqref{cubes fghk} 
and  to the commuting polytopes \eqref{4cocycle 1} (respectively,  \eqref{4cocycle 2}, \eqref{4cocycle 3}, and \eqref{4cocycle 4}) 
so that the faces of the octahedra labelled by  $\beta$'s  in octahedra and polytopes are glued to each other. 
In the resulting large polytopes  each face labelled by $\alpha$ or $\delta$ is glued to its inverse. This implies commutativity of
the polytopes, i.e. vanishing of the differentials. 

A different choice of $2$-cells \eqref{alpha fgh}  and \eqref{deltas} results in multiplying $\mathcal{p}^0_C$ by a braided 
$4$-coboundary, so its class in  $H^4(G,\, \pi_2(\bG))$ is well-defined.

Note that $C$ extends to a braided monoidal $2$-functor  $A\to \bG$  if the cells $\alpha_{x,y,z}$ and $\delta_{x,y}$
can be chosen in such a way that \eqref{mon2fun}, \eqref{braided2fun1}, and  \eqref{braided2fun2} are satisfied.
This is equivalent to commutativity of cubes \eqref{cubes fghk} and octahedra \eqref{oct theta} and \eqref{oct psi}, i.e. to $\mathcal{p}^0_C=1$.
Indeed, these polytopes are obtained  by gluing the two sides of \eqref{mon2fun}, \eqref{braided2fun1}, and  \eqref{braided2fun2}.   
\end{proof}

Suppose that a braided monoidal functor $C:A \to \varPi_{\leq 1}(\bG)$  extends to a braided monoidal $2$-functor $\C: A\to \bG$.
That is, there is a choice of invertible $2$-cells \eqref{alpha fgh}  and \eqref{deltas}
such that the cubes \eqref{cubes fghk} and octahedra \eqref{oct theta} and \eqref{oct psi} commute,
i.e. $\mathcal{p}^0_C=1$. Let $P$ be a monoidal automorphism of $C$.

Let $\mathcal{p}^1_C(P)(x,y,z)$ be defined by \eqref{cubes for PW1}
and let $\mathcal{p}^1_C(P)(x\,|\,y)$ be the composition of the faces of the prism 
\begin{equation}
\label{prism x|y}
\xymatrix{
&&& \C_x\C_y  \ar[ddll]_{P_xP_y} \ar@{-->}[dd]_{B_{x,y}}="a" \ar[ddrrrr]^{M_{x,y}}="c" &&  &&  \\
&&& && && \\
 \mathcal{p}^1_C(P)(x\,|\,y) =
& \C_x\C_y  \ar[dd]_{B_{x,y}}="b"  \ar[ddrrrr]^{M_{x,y}}="d" &&  \C_y \C_x \ar@{-->}[ddll]_{P_yP_x} \ar@{-->}[rrrr]_{M_{y,x}}="e"  && && \C_{xy}  \ar[ddll]^{P_{xy}}  \\
 &&& && && \\
& \C_y\C_x   \ar[rrrr]_{M_{y,x}}="f" && && \C_{xy}, &&
\ar@{}"a";"b"^(.10){}="x"^(.60){}="y" \ar@2{-->}^{B_{P_x,P_y}}"x";"y"
\ar@{}"c";"d"^(.25){}="x"^(.75){}="y" \ar@2{->}_{\mu_{x,y}}"x";"y"
\ar@{}"c";"e"^(.25){}="x"^(.75){}="y" \ar@2{-->}^{\delta_{x,y}}"x";"y"
\ar@{}"d";"f"^(.25){}="x"^(.75){}="y" \ar@2{->}^{\delta_{x,y}}"x";"y"
\ar@{}"e";"f"^(.10){}="x"^(.60){}="y" \ar@2{-->}_{\mu_{y,x}}"x";"y"
}
\end{equation}
for all $x,y\in A$.

\begin{proposition}
\label{br p1CP}
$\mathcal{p}^1_C(P)$ is a braided $3$-cocycle and the map 
\begin{equation}
\label{p1br(C)}
\mathcal{p}^1_C: \Aut_\ot(C) =H^1(A,\, \pi_1(\bG)) \to H^3_{br}(A,\, \pi_2(\bG)) : P \mapsto \mathcal{p}^1_C(P)
\end{equation}
is a well defined group homomorphism. 

The natural automorphism $P$ extends to a braided monoidal pseudo-natural automorphism of $\C$ if and only if
$\mathcal{p}^1_C(P)=0$ in $H^3_{br}(A,\, \pi_2(\bG))$.
\end{proposition}
\begin{proof}
The differential \eqref{3cob1} vanishes by Proposition~\ref{p1CP prop}.
The vanishing of the differential \eqref{3cob2} (respectively,\eqref{3cob3})  is established by gluing the faces of three cubes
\eqref{cubes for PW1},  three prisms \eqref{prism x|y}, and two copies of the octahedron~\eqref{oct theta} (respectively, \eqref{oct psi}) 
in such a way that the result has the empty boundary.

A different choice of $2$-cells \eqref{muss xy}  results in multiplying $\mathcal{p}^1_C(P)$ by a braided coboundary, so its class in 
$H^3_{br}(G,\, \pi_2(\bG))$ is well-defined.

The multiplicative property of $\mathcal{p}^1_C$ is a direct consequence of the definition of prisms \eqref{prism x|y}.

Finally, a monoidal pseudo-natural automorphism of $\C$ obtained by extending $P$ is braided
if $2$-cells \eqref{muss xy} can be chosen in such a way that the cube \eqref{BT cube} is satisfied.
But this cube becomes a prism \eqref{prism x|y} in our situation since the structural $2$-cells of $A$, viewed
as a braided $2$-categorical group, are trivial. So the braided property is equivalent
to the braided cohomology class of $\mathcal{p}^1_C(P)$ being trivial.
\end{proof}

Let $\C:A \to \bG$ be a monoidal $2$-functor. For any $(\omega,\, c) \in Z^3_{br} (A,\, \pi_2(\bG))$ let $\C^{(\omega,c)}$ 
be a monoidal $2$-functor obtained from $\C$ by multiplying each $2$-cell $\alpha_{x,y,z}$ by $\omega(x,\,y,\,z)$
and each $2$-cell $\delta_{x,y}$ by $c(x,y)$ for all $x,y,z\in A$.
The isomorphism class of $\C^{(\omega,c)}$ depends only on the class of $(\omega,\,c)$ in $H^3_{br}(A,\, \pi_2(\bG))$.
If $\C,\, \C'$ are extensions of the same braided monoidal functor  $C:A \to \varPi_{\leq 1}(\bG)$ if and only if 
$\C'\cong \C^{(\omega,c)}$  for some $(\omega,\,c)$. 

\begin{corollary}
\label{cokernel torsor braided}
Braided monoidal $2$-functors $\C^{(\omega_1,c_1)},\, \C^{(\omega_2, c_2)}: A \to \bG$  are isomorphic if and only if  
$(\omega_2,\, c_2) = (\omega_1,\, c_1)\,  \mathcal{p}^1_C$ for some $P\in \Aut(C)=H^1(A,\, \pi_1(\bG))$.  
\end{corollary}
\begin{proof}
This is similar to the proof of Corollary~\ref{cokernel torsor}, where a criterion for isomorphism of monoidal $2$-functors
$\C^{\omega_1}$ and $\C^{\omega_2}$ was established.  The braided property of such an  isomorphism translates to
commutativity the prism \eqref{prism x|y}  with the front and back faces being, respectively, $c_1(x,y,z)\delta_{x,y}$ and 
$c_2(x,y)\delta_{x,y}$.  This is equivalent to $c_2/c_1 = \mathcal{p}^1_C(P)$,
where $C: A \to \varPi_{\leq 1}(\bG)$  is the underlying monoidal functor of $\C$. 
\end{proof}

The next Corollary summarizes  our description of braided monoidal $2$-functors $A \to \bG$.

\begin{corollary}
\label{cohomology description braided}
Let $C:G \to \varPi_{\leq 1}(\bG)$ be a braided monoidal functor. An extension of $C$ to a  monoidal $2$-functor $A \to \bG$ exists if and only 
if $ \mathcal{p}^0_C=0$ in $H^4_{br}(A,\, \pi_2(\bG))$. 
Equivalence classes  of such extensions form a tosrsor over 
$\Coker \left( \mathcal{p}^1_C: H^1(A,\, \pi_1(\bG)) \to H^3_{br}(A,\, \pi_2(\bG))   \right)$.
\end{corollary}

\subsection{Symmetric monoidal $2$-functors between 2-categorical groups}
\label{Sect Symmetric functors and cohomology}

Let $A$ be an Abelian group. Here we will consider it as a symmetric 2-categorical group with identity 1- and 2-morphisms.

Let  $\bG$ be a symmetric $2$-categorical group with the corresponding  canonical classes 
\[
\alpha_\bG\in H^3_{sym}(\pi_0( \bG),\, \pi_1( \bG))
\quad\text{and}\quad
q_\bG \in H^3_{sym}(\pi_1( \bG),\, \pi_2( \bG)).
\]
One can extend the obstruction theory  from Section~\ref{Sect Braided functors and cohomology} to symmetric
monoidal $2$-functors $A\to \bG$. Namely, let $C: A\to\varPi_{\leq 1}(\bG)$  be a symmetric monoidal functor. 
We define $\mathcal{p}_0^C\in C^4_{sym}(A,\, \pi_2(\bG))$ by extending the braided $4$-cocycle from
Proposition~\ref{br obstruction} as follows. The components $\mathcal{p}_0^C(-,\,-,\,-,\,-)$, $\mathcal{p}_0^C(-\, |\,-,\,-)$, and 
$\mathcal{p}_0^C(-,\,-,\, | \,-)$  are given by \eqref{cubes fghk} , \eqref{oct theta}, and \eqref{oct psi}, respectively, and 
\begin{equation}
\label{cone}
\xymatrix{
& \C_x\bt_B \C_y  \ar@/^-2pc/[rrrr]^{B_{x,y}}^{}="a" \ar[dddrr]_{M_{x,y}}="A" &&&& 
\C_y\bt_B \C_x  \ar@/^-2pc/[llll]_{B_{y,x}}^{}="b"^(.4){}="Z"    \ar[dddll]^{M_{y,x}}="B"^(.8){}="W" \\
\mathcal{p}_0^C(x\, || \,y) = && &&& \\
&&& && \\
&&& \C_{xy},&&
\ar@{}"a";"b"^(.25){}="x"^(.75){}="y" \ar@2{->}_{\tau_{x,y}} "y";"x"  
\ar@{}"4,4";"a"^(.2){}="x"^(.8){}="y" \ar@2{->}^{\delta_{x,y}} "x";"y"
\ar@{}"W";"Z"^(.2){}="x"^(.8){}="y" \ar@2{-->}_{\delta_{y,x}} "x";"y"
}
\end{equation}
for all $x,y\in A$.

\begin{proposition}
\label{sym obstruction}
The above $\mathcal{p}^0_C$
is a symmetric  $4$-cocycle whose cohomology class in $H^4_{sym}(A,\, \pi_2(\bG))$  depends only on the isomorphism class of $C$.  
A symmetric monoidal functor $C: A\to \varPi_{\leq 1}(\bG)$ extends to a symmetric  monoidal $2$-functor
$A\to \bG$ if and only if  $\mathcal{p}^0_C =0$ in $H^4_{sym}(A,\, \pi_2(\bG))$.
\end{proposition} 
\begin{proof}
We need to check vanishing of the differentials \eqref{sym1}, \eqref{sym2}, and \eqref{ssym}. 
Vanishing of the differentials  \eqref{sym1} and  \eqref{sym2} is checked by gluing  polytopes \eqref{oct theta} and \eqref{oct psi}
along their associativity faces $\alpha$  and gluing their braiding faces $\delta$ to two sides of cones \eqref{cone}.
The differential  \eqref{ssym} vanishes thanks to axiom \eqref{3rd symmetric} of a symmetric monoidal $2$-category.

Proposition~\ref{br obstruction} gives a criterion for $\C$ to have an extension to a braided monoidal $2$-functor.
This extension admits a  symmetric monoidal  $2$-functor structure if and only if  the cells \eqref{deltaMN} are  chosen in such a  way that the 
cone \eqref{cone} commutes. This is equivalent to $\mathcal{p}_0^C$ being trivial in $H^4_{sym}(A,\, \pi_2(\bG))$.
\end{proof}

Suppose that a symmetric monoidal functor $C:A \to \varPi_{\leq 1}(\bG)$  extends to a symmetric monoidal $2$-functor $\C: A\to \bG$.
For any $P\in \Aut_\ot(C)$ the braided $3$-cocycle from Proposition~\ref{br p1CP} is symmetric, 
i.e. 
\[
\mathcal{p}_1^C(P)(x\,|\,y)\,\mathcal{p}_1^C(P)(y\,|\,x) =1
\] 
for all $x,y\in A$. This can be seen gluing boundaries of two prisms \eqref{prism x|y} and two cones \eqref{cone}.

\begin{corollary}
\label{cokernel torsor symmetric}
Let $(\omega_1,c_1),\,(\omega_2, c_2)\in Z^3_{sym}(A, \pi_1(\bG))$ be symmetric $3$-cocycles.
Symmetric monoidal $2$-functors $\C^{(\omega_1,c_1)},\, \C^{(\omega_2, c_2)}: A \to \bG$  are isomorphic if and only if  
$(\omega_2,\, c_2) = (\omega_1,\, c_1)\,  \mathcal{p}_1^C(P)$ for some $P\in \Aut(C)=H^1(A,\, \pi_1(\bG))$.  
\end{corollary}
\begin{proof} 
This is the same as Corollary~\ref{cokernel torsor braided}, 
since there is no difference between isomorphisms of braided and symmetric monoidal $2$-functors.
\end{proof}

The next Corollary summarizes  our description of braided monoidal $2$-functors $A \to \bG$.

\begin{corollary}
\label{cohomology description symmetric}
Let $C:A \to \varPi_{\leq 1}(\bG)$ be a symmetric monoidal functor. An extension of $C$ to a  monoidal $2$-functor $A \to \bG$ exists if and only 
if $ \mathcal{p}^0_C=0$ in $H^4_{sym}(A,\, \pi_2(\bG))$. 
Equivalence classes  of such extensions form a torsor over 
$\Coker \left( \mathcal{p}^1_C: H^1(A,\, \pi_1(\bG)) \to H^3_{sym}(A,\, \pi_2(\bG))   \right)$.
\end{corollary}

\subsection{The symmetric $2$-categorical group of symmetric monoidal $2$-functors}
\label{Sect cat grp sym2fun}

Let $A$ be a finite Abelian group and let $\bG$ be a  symmetric $2$-categorical group.

Let $C,\,C': A \to \varPi_{\leq 1}(\bG)$ be symmetric monoidal $2$-functors, where $C$ is given by  $x\mapsto \C_x$  with the monoidal
structure $M_{x,y}: \C_x\bt \C_y \xrightarrow{\sim} \C_{xy}$  and $C'$ is given by  $x\mapsto \C'_x$  with the monoidal
structure $M'_{x,y}: \C'_x\bt \C'_y \xrightarrow{\sim} \C'_{xy}$, $x,y\in A$. 

Define   a symmetric monoidal functor 
\begin{equation}
\label{tildeM}
\tilde{C}:=
C\bt C': A \to  \varPi_{\leq 1}(\bG): \qquad x\mapsto \C_x \bt_\B \C'_x.
\end{equation}
with the monoidal  structure 
\begin{equation}
\tilde{M}_{x,y}:  \C_x \bt \C'_x \bt  \C_y \bt \C'_y \xrightarrow{B_{x',y}} \C_x \bt \C_y \bt  \C'_x \bt  \C'_y 
\xrightarrow{M_{x,y} \bt  M'_{x,y}}  \C_{xy} \bt   \C'_{xy},\qquad x,\,y\in A.
\end{equation}
Here and below we denote $B_{x',y}$ the braiding between $\C'_x$ and $\C_y$. 

Suppose that $C$ and $C'$ extend to symmetric monoidal $2$-functors $\C,\, \C': A \to \bG$.
The  associativity and braiding  $2$-cells \eqref{alphaLMN}
and \eqref{deltaMN} for $\C$ and $\C'$ will be denoted $\alpha,\, \alpha'$ and $\delta,\, \delta'$,  respectively. 

Our goal is to construct a canonical braided monoidal $2$-functor  $\tilde{\C}$ extending $\tilde{C}$. 

Define the  associativity 
$2$-cells $\tilde{\alpha}_{x,y,z}\, (x,y,z\in A)$
by 
\begin{equation}
\label{tildealpha}
\scalebox{1.15}{
\xymatrix{
\C_x  \C'_x  \C_y  \C'_y  \C_z  \C'_z    \ar[rr]^{B_{y',z}}  \ar[dd]_{B_{x',y}} &&
\C_x  \C'_x  \C_y  \C_z  \C'_y  \C'_z  \ar[rr]^{M_{y,z}  M'_{y,z}}^(.4){}="e"    \ar[dd]^{B_{x',y\bt z}}_{}="B"  \ar[dl]_{B_{x',y}} &&
\C_x  \C'_x  \C_{yz}  \C'_{yz }  \ar[dd]^{B_{x',yz}} \\
& \C_x  \C_y   \C'_x  \C_z  \C'_y  \C'_z   \ar[dr]^{B_{x',yz}}&&&\\
\C_x  \C_y  \C'_x   \C'_y  \C_z  \C'_z  \ar[rr]^{B_{(x\bt y)',z}}="A"   \ar[dd]_{M_{x,y}  M'_{x,y}}^{}="c"  \ar[ur]^{B_{y',z}} && 
\C_x   \C_y  \C_z  \C'_x  \C'_y  \C'_z  \ar[rr]^{M_{y,z}  M'_{y,z}}^(.4){}="f" \ar[dd]_{M_{x,y}  M'_{x,y}}="d"^{}="a"  &&
\C_x  \C_{yz}  \C'_x  \C'_{yz}   \ar[dd]^{M_{x,yz}  M'_{x,yz}}_{}="b" \\ 
\\
\C_{xy}  \C'_{xy}  \C_z  \C'_z \ar[rr]^{B_{(xy)',z}}  &&
\C_{xy}  \C_z   \C'_{xy}  \C'_z   \ar[rr]_{M_{xy,z}  M'_{xy,z}}  &&
\C_{xyz} \C'_{xyz}
\ar@{}"a";"b"^(.25){}="x"^(.75){}="y" \ar@2{->}^{\alpha_{x,y,z}\alpha'_{x,y,z}}"x";"y"
\ar@{}"c";"d"^(.25){}="x"^(.75){}="y" \ar@2{->}^{\bt_{M_{x,y}, B_{M'_{x,y},\C_z}}}"x";"y"
\ar@{}"e";"f"^(.25){}="x"^(.75){}="y"  \ar@2{->}^{\bt_{B_{\C'_x,M_{y,z}}, M'_{y,z}}}"x";"y"
\ar@{}"1,1";"2,2"^(.25){}="x"^(.75){}="y"  \ar@2{->}^{\bt_{B_{x',y},B_{y',z}} }"x";"y"
\ar@{}"2,2";"A"^(.15){}="x"^(.85){}="y" \ar@2{->}^{\beta_{x',y'|z}}"x";"y"
\ar@{}"2,2";"B"^(.45){}="x"^(.95){}="y" \ar@2{->}^{\beta_{x'|y,z}}"x";"y"
} 
}
\end{equation}
and the braiding $2$-cells $\tilde{\delta}_{x,y}\, (x,y\in A)$ by 
\begin{equation}
\label{tildedelta}
\scalebox{1.15}{
\xymatrix{
\C_x  \C'_x \C_y   \C'_y    \ar[rrrr]^{B_{x\bt x', y\bt y'}}_{}="C" \ar[dd]_{B_{x',y}}  \ar[ddrr]^>>>>>>>>>>>{B_{x\bt x', y}}_{}="A" 
&& &&   
\C_y   \C'_y \C_x  \C'_x  \ar@/^3pc/[dd]^{B_{y',x}} _{}="E"\\
\\
 \C_x  \C_y  \C'_x   \C'_y \ar[rr]_{B_{x,y}} \ar[ddrr]_{M_{x,y}  M'_{x,y}} &&  
\C_y  \C_x  \C'_x   \C'_y \ar[rr]_{B_{x', y'}}  \ar[uurr]^{B_{x\bt x', y'}}_{}="B"  &&
\C_y  \C_x  \C'_y   \C'_x  \ar[uu]^{B_{x, y'}}_{}="D" \ar[ddll]^{M_{y,x}  M'_{y,x}}   \\
\\
&& \C_{xy}    \C'_{xy}. &&
\ar@{}"5,3";"3,3"^(.25){}="x"^(.75){}="y"  \ar@2{->}_{\delta_{x,y}\delta'_{x,y}}"x";"y"
\ar@{}"3,1";"A"^(.45){}="x"^(.95){}="y"  \ar@2{->}_>>>>>>{\beta_{x,x'|y}}"x";"y"
\ar@{}"3,5";"B"^(.45){}="x"^(.95){}="y"   \ar@2{->}^>>>>>>{\beta_{x,x'|y'}}"x";"y"
\ar@{}"3,3";"C"^(.45){}="x"^(.95){}="y"  \ar@2{->}^>>>>>>{\beta_{x\bt x'|y,y'}}"x";"y"
\ar@{}"D";"E"^(.15){}="x"^(.85){}="y"  \ar@2{->}^{\tau_{x,y'}}"x";"y"
}
} 
\end{equation}
Here we  write $\beta_{x\bt x'|y,y'}$ as a shorthand  for $\beta_{\C_x\bt \C'_x|  \C_y,\C'_y}$ etc. 

\begin{proposition}
\label{tildeM is braided}
The $2$-cells \eqref{tildealpha} and \eqref{tildedelta} make $\tilde{\C}= \C\bt  \C'$ a symmetric monoidal $2$-functor.
\end{proposition}
\begin{proof}
The proof is tedious but straightforward. It extends the corresponding argument for symmetric monoidal functors
and consists of decomposing the  commuting cube \eqref{cubes fghk} and
octahedra \eqref{oct theta},  \eqref{oct psi} formed by $2$-cells \eqref{tildealpha} and \eqref{tildedelta} 
into unions of commuting polytopes glued together.

For the cube \eqref{cubes fghk} for $\tilde{\C}$ one gets commuting 
polytopes obtained by gluing both sides of \eqref{4cocycle 1} - \eqref{4cocycle 4}, the polytopes  commuting due to the naturality
of braiding and the naturality of cells $\beta$ and $\tau$, and  cubes  \eqref{cubes fghk} for $\C$ and $\C'$. 
For the octahedra \eqref{oct theta},  \eqref{oct psi} 
one gets commuting polytopes as above, the corresponding  polytopes
for $\C$ and $\C'$, and the symmetry polytopes   \eqref{1st symmetric},  \eqref{2nd symmetric}, and \eqref{3rd symmetric})
of $\bG$.  It follows that  $\tilde{\C}$ is a braided monoidal $2$-functor. 

The cone \eqref{cone} corresponding to the property of $\tilde{\C}$ being symmetric is comprised
from ${\delta}_{x,y},\, {\delta}'_{x,y}$ and $\tau_{x,y}$ for $x,y\in A$. This cone
decomposes into the union of several commuting polytopes, namely the pair of 
corresponding cones for $\C$ and $\C'$ and the symmetry polytopes 
\eqref{1st symmetric},  \eqref{2nd symmetric}, and \eqref{3rd symmetric}.  Hence, it
commutes.
\end{proof}

\begin{proposition}
The above product of functors turns  $\FUNsym(A,\, \bG)$  into a symmetric $2$-categorical group. 
\end{proposition}
\begin{proof}
For $\C,\,\C',\,\C''\in  \FUNsym(A,\, \bG)$  there is a pseudo-natural equivalence between
$(\C \bt \C') \bt \C''$ and $\C \bt (\C'\bt \C'')$. This can be seen to be monoidal by comparing  the associativities
\eqref{tildealpha} for both $2$-functors. The unit object of  $\FUNsym(A,\, \bG)$ is the trivial symmetric $2$-functor (with $\C_x=\I$ 
for all $x\in A$ and all structure morphisms and cells being identities). 

The braiding  of $\C$ and $\C'$ is  a pseudo-natural  isomorphism given by 
\begin{equation}
\label{B CC'}
B_{\C,\,\C'}(x) : \C_x \bt \C'_x \xrightarrow{B_{x,x'}} \C'_x \bt \C_x,\qquad x\in A, 
\end{equation}
with  $2$-cells \eqref{mon2trans} being the following compositions:
\begin{equation}
\xymatrix{
\C_x\C'_x\C_y\C'_y \ar@/^2pc/[rr]^{B_{x',y}}_{}="a" \ar[dd]_{B_{yy'}}  && 
\C_x\C_y\C'_x\C'_y \ar[rr]^{M_{x,y}\bt M_{x',y'}} \ar[dddd]^{B_{x\bt y, x'\bt y'}}_{}="c" 
\ar[ddll]^{B_{y, x'\bt y'}}  \ar[ll]^{B_{y,x'}}_{}="b" &&
\C_{xy}  \C'_{xy} \ar[dddd]^{B_{xy, x'y'}} 
\\ 
&& && \\
\C_x\C'_x\C'_y  \C_y \ar[dd]_{B_{x,x'}}  \ar[ddrr]^{B_{x, x'\bt y'}}  && && \\
&& && \\
 \C'_x\C_x\C'_y  \C_y  \ar[rr]^{B_{x,y'}}  &&  \C'_x\C'_y  \C_x \C_y \ar[rr]^{M_{x',y'}\bt \M_{x,y}} && \C'_{xy}  \C_{xy}.
\ar@{}"a";"b"^(.25){}="a1"^(.75){}="b1" \ar@{=>}^{\tau_{x',y}}"a1";"b1"
\ar@{}"1,1";"2,2"^(.35){}="c1"^(.95){}="d1" \ar@{=>}_{\beta_{y|x',y'}}"c1";"d1"
\ar@{}"3,1";"c"^(.25){}="e1"^(.75){}="f1" \ar@{=>}^{\beta_{x,y|x'\bt y'}}"e1";"f1"
\ar@{}"5,1";"4,2"^(.35){}="g1"^(.95){}="h1" \ar@{=>}^{\beta_{x|x',y'}}"g1";"h1"
\ar@{}"1,4";"5,4"^(.25){}="g11"^(.75){}="h11" \ar@{=>}^{\delta_{x,y}\bt \delta'_{x,y}}"g11";"h11"
}
\end{equation}
The  $2$-cells \eqref{betaxyz} are $\beta_{x,x'|x''}$ and $\beta_{x|x'\,x''},\, x\in A$. 
One can directly verify commutativity of the cubes \eqref{MT cube} and \eqref{BT cube}.

Finally, the symmetry $2$-cell $\tau$ of $\bG$ provides an invertible modification
between $B_{\C',\C}\circ B_{\C,\C'}$ and $\id_{\C\bt \C'}$ satisfying \eqref{1st symmetric},
\eqref{2nd symmetric}, and \eqref{3rd symmetric}.
\end{proof}

\begin{theorem}
\label{exact sequence for Funsym groups}
There is an  exact sequence of group homomorphisms:
\begin{multline}
\label{exact sequence Funsym}
H^1(A,\,\pi_1(\bG)) \to H^3_{sym}(A,\, \pi_2(\bG)) \to \pi_0(\FUNsym(A,\, \bG))
  \to \pi_0(\Fun_{sym}(A,\, \varPi_{\leq 1}(\bG))) \to H^4_{sym}(A,\,\pi_2(\bG)).
\end{multline}
\end{theorem}
\begin{proof}
The first three arrows are described in Section~\ref{Sect Symmetric functors and cohomology}. 
That they are homomorphisms follows from the definition of the tensor product in $\FUNsym(A,\, \bG)$.

We need to check that
\[
\pi_0(\Fun_{sym}(A,\, \varPi_{\leq 1}(\bG))) \to H^4_{sym}(A,\,\pi_2(\bG)) : C \mapsto \mathcal{p}^0_C,
\]
where the components of the symmetric $4$-cocycle $\mathcal{p}^0_C$ are given by the values
of polytopes \eqref{cubes fghk}, \eqref{oct theta}, \eqref{oct psi}, and \eqref{cone}, 
is a group homomorphism. This is achieved by  decomposing each of these polytopes for $\C\bt \C'$, where
$\C,\C'\in \FUNsym(A,\, \bG)$, into the union of the corresponding polytopes for $\C$ and $\C'$ and 
commuting polytopes satisfied by the structure $2$-cells of  $\bG$ as well as those of $\C,\, \C'$, glued together
in such a way that the resulting boundary is empty.

The exactness of this sequence follows from Corollary~\ref{cohomology description symmetric}.
\end{proof}

\section{Module categories}
\label{Chapter on module categories}

\subsection{Module categories over a tensor category}

Le $\C$ be a tensor category with the associativity constraint $a_{X,Y,Z}: (X\ot Y)\ot Z  \xrightarrow{\sim} X\ot (Y \ot Z)$. 
Let $\C^{\opp}$ denote the tensor category with the opposite multiplication $X\ot^\opp Y = Y\ot X$ and the associativity
constraint $a_{X,Y,Z}^\opp = a_{Z,Y,X}^{-1}: Z \ot (Y \ot X) \xrightarrow{\sim} (Z\ot Y)\ot X$ for $X,Y,Z\in \C$.
Below we recall definitions from \cite{O1, EGNO}.

\begin{definition}
A (left) {\em $\C$-module category}   is a finite  Abelian  $k$-linear category $\M$ together with a bifunctor
\[
\C\times\M\to\M,\quad (X,M)\mapsto X*M,
\]
exact in each variable, and a collection of isomorphisms ({\em $\C$-module associativity constraint})
$$
m_{X,Y,M}:(X\ot Y)*M \xrightarrow{\sim} X*(Y*M),
$$
natural in $X,Y\in\C, M\in M$ and such that the diagram 
\begin{equation}
\label{cmc}
\xymatrix{
& ((X\ot Y)\ot Z)*M \ar[dr]^{m_{X\ot Y,Z,M}} \ar[dl]_{a_{X,Y,Z}*\id_M} &  \\
(X\ot (Y\ot Z))*M  \ar[d]_{m_{X,Y\ot Z,M}} && (X\ot Y)*(Z*M) \ar[d]^{m_{X,Y,Z*M}}  \\
X*((Y\ot Z)*M)  \ar[rr]^{\id_X*m_{Y,Z,M}} && X*( Y* (Z*M))
}
\end{equation}
commutes for all $X,Y,Z\in\C,\ M\in\M$. 

A {\em right} $\C$-module category  is a $\C^\opp$-module category.  A {\em $\C$-bimodule} category is a $(\C \bt \C^\opp)$-module category.
\end{definition}

\begin{remark}
A  $\C$-bimodule category $\M$ can be equivalently described as a category with both left and right $\C$-module structures and a
collection of isomorphisms (a {\em middle associativity constraint}) 
\begin{equation}
m_{X,M,Y}: (X * M) * Y \xrightarrow{\sim}  X * (M * Y)
\end{equation}
 natural in $X,Y\in \C,\, M\in \M$ compatible
in a certain way \cite[Definition 7.1.7]{EGNO}.
\end{remark}

\begin{definition}
A {\em $\C$-module}  functor $F:\M\to\N$ between $\C$-module categories is a functor 
along with a collection of isomorphisms $F_{X,M}:X*F(M)\xrightarrow{\sim}  F(X*M)$   natural in $X\in\C,\ M\in\M$ 
such that the following diagram 
\begin{equation}
\label{cmf}
\xymatrix{
& (X\ot Y)*F(M) \ar[dr]^{F_{X\ot Y, M}} \ar[dl]_{m_{X,Y,F(M)}} &  \\
X*(Y*F(M))  \ar[d]_{\id_X*F_{Y,M}}  && F((X\ot Y)*M) \ar[d]^{F(m_{X,Y,M})} \\
X*F(Y*M)  \ar[rr]^{F_{X, Y*M}} && F(X*(Y*M))
}
\end{equation}
commutes for $X,Y\in\C,\ ,M\in\M$. 
\end{definition}

\begin{definition}
A natural {\em $\C$-module transformation} between $\C$-module functors $F,F':\M\to\N$ is  a natural transformation $\eta:F\to F'$  
such that
\beq\lb{cmt}
\xymatrix{X*F(M)\ar[rr]^{F_{X,M}} \ar[d]_{\id_X*\eta_M} && F(X*M) \ar[d]^{\eta_{X*M}}\\X*F'(M)\ar[rr]^{F'_{X,M}} && F'(X*M)}
\eeq
commutes for all $X\in\C,\, M\in\M$. 
\end{definition}

Let $F:\L \to \M$ and $F":\M \to \N$
be  $\C$-module functors  then $F'\circ F$ has a canonical structure of $\C$-module functor 
\begin{equation}
\label{comf}
(F'\circ F)_{X,M}:  X*F'(F(M)) \xrightarrow{F'_{X,F(M)}}  F'(X*F(M)) \xrightarrow{F'(F_{X,M})}  F'(F(X*M)),\quad X\in \C,\, M\in \M.
\end{equation}
Thus, $\C$-module categories, $\C$-module functors, and $\C$-module natural transformations form a strict 2-category $\Mod(\C)$.

\subsection{Tensor product of module categories}
\label{sect: tp tensor}

Let $\C$ be a tensor category, let $\M$ be a right $\C$-module category, and let $\N$ be a left $\C$-module category.
The {\em (relative) tensor product} $\M \bt_\C \N$ \cite{ENO} is an abelian category $\M \bt_\C \N$ along with a functor $\M \times \N \to \M \bt_\C \N$
universal among $\C$-balanced and right exact in each variable  functors 
 from $\M \times \N$ to Abelian categories.

An explicit description is given as follows.
Objects of $\M \bt_\C \N$ are pairs $(V,\gamma)$, where $V\in \M\bt \N$ and
\begin{equation}
\gamma_X: V*(X\bt \be) \xrightarrow{\sim} (\be \bt X)*V,
\end{equation} 
is a  {\em balancing isomorphism} natural in $V \in \M \bt \N, X\in \C$ and such that the following diagram
\begin{equation}
\label{MbtBN}
\xymatrix{
V*((X\ot Y)\bt \be)  \ar[rr]^{\gamma_{X\ot Y}} \ar[d]_{m_{V,X,Y}} && (\be \bt (X\ot Y))*V  \ar[d]^{n_{X,Y, V}} \\ 
(V*(X\bt \be))*(Y\bt \be) \ar[r]^{\gamma_X}  &  (\be \bt X)* (V*(Y \bt  \be) )   \ar[r]^{\gamma_Y} &  (\be \bt X)* ((\be \bt  Y)*V ) , 
}
\end{equation}
commutes. Here $m$ and $n$ are the module associativity constraints in $\M$ and $\N$.

A morphism between $(V,\, \{\gamma_X\}_{X\in \C})$ and $(V',\, \{\gamma'_X\}_{X\in \C})$ 
in $\M \bt_\C \N$ is a  morphism $f: V \to V'$ in $\M\bt \N$ such that
the diagram
\begin{equation}
\xymatrix{
V*(X\bt \be) \ar[rr]^{f* (X\bt \be)} \ar[d]_{\gamma_X} && V' *(X\bt \be)\ar[d]_{\gamma'_X}\\
(\be \bt X)*V \ar[rr]^{(\be\bt X)*f} && (\be \bt X) *V'
}
\end{equation}
commutes for all $X\in \C$.

If $\M$ is a $\C$-bimodule category then   $\M \bt_\C \N$ inherits the left $\C$-module category structure from $\M$:
\begin{equation}
\label{left module on tp}
Y* (V,\, \{\gamma_X\}) = ((Y\bt \be)*V ,\, \{\overline{\gamma}_X\}),
\end{equation}
where
\begin{equation}
\xymatrix{
((Y\bt \be)*V)*(X\bt \be)  \ar[rr]^{\overline{\gamma}_X} \ar[d]_{m^{-1}_{Y,V,X}} && (\be\bt X)*((Y\bt \be)*V)  \\
(Y\bt \be)*(V*(X\bt \be))   \ar[rr]^{\gamma_X} && (Y\bt \be)*((\be\bt X)*V)  \ar@{=}[u]
}
\end{equation}
for all $X,Y\in \C$. Similarly, if $\N$ is a $\C$-bimodule category then $\M \bt_\C \N$ inherits the right $\C$-module category 
structure from $\N$.

Thus, there is a  monoidal $2$-category $\Bimod(\C)$ of
 $\C$-bimodule categories. Its 1-cells are $\C$-bimodule functors and 2-cells are natural transformations of $\C$-bimodule functors.  
The regular $\C$-bimodule category $\C$ is the identity for $\bt_\C$.

\subsection{From module categories over a braided tensor category to bimodule categories}
\label{sect: from mod to bimod}

Let  $\B$ be a braided tensor category with the braiding 
\[
c_{X,Y} : X \ot Y \xrightarrow{\sim} Y \ot X,\qquad X,\,Y \in \B.
\] 
The braiding of $\B$  allows to turn a left $\B$-module category $\M$ into a $\B$-bimodule category as follows. 
Let $m_{X,Y,M}: (X\ot Y)\ot M \xrightarrow{\sim} X\ot(Y\ot M) $ denote the 
left $\B$-module associativity constraint of $\M$.  
Define the right action of $\B$ on $\M$ by $M* X := X * M$ for all $X\in \B$ and $M\in \M$. 
The right
$\B$-module associativity constraint  is given by the composition
\begin{equation}
\label{right associativity}
\xymatrix{
M *(X\ot Y) \ar@{=}[d] \ar[rrrr]^{m_{M,X,Y}} &&&&  (M*X)* Y \ar@{=}[d] \\
(X\ot Y) *M  \ar[rr]^{c_{X,Y}}  && (Y\ot X)*M  \ar[rr]^{m_{Y,X,M}} && Y * (X * M)
%
}
\end{equation}
and the middle associativity constraint is given by
\begin{equation}
\label{middle associativity}
\xymatrix{
(X*M)*Y \ar@{=}[d] \ar[rrrrrr]^{m_{X,Y,M}} &&&&&& X*(M*Y)  \ar@{=}[d]  \\
Y * (X * M) \ar[rr]^{m_{Y,X,M}^{-1}}  && (Y\ot X) *M  \ar[rr]^{c_{Y,X}} && (X\ot Y)*M \ar[rr]^{m_{X,Y,M}} && X *(Y* M),
%
}
\end{equation}
for all $X,\,Y\in \B$ and $M\in \M$.

\begin{remark}
Since $\B$-module functors  and their $\B$-module natural transformations extend 
to $\B$-bimodule functors  and $\B$-bimodule transformations  in an obvious way,
there is a $2$-embedding $\Mod(\B)\to \Bimod(\B)$.
\end{remark}

Using the  $\B$-bimodule structure of $\M$ we define the tensor product $\M \bt_\B \N$ 
of $\B$-module categories $\M,\, \N$, as in Section~\ref{sect: tp tensor}.  It has a canonical structure
of a left $\B$-module category.  This makes $\Mod(\B)$ a monoidal $2$-category.

\begin{remark}
\label{explicit btB}
From  \eqref{MbtBN} we see that objects of $\M \bt_\B \N$ are pairs $(V,\, \{\gamma_X\}_{X\in \B})$, where $V\in \M\bt \N$ and
\[
\gamma_X: (X\bt \be)*V \to (\be \bt X)*V,\qquad V \in \M \bt \N, X\in \B,
\] 
is a natural balancing isomorphism satisfying  
\begin{equation}
\label{braided MbtBN}
\xymatrix{
((X\ot Y)\bt \be)*V \ar[rrrr]^{\gamma_{X\ot Y}}  \ar[d]_{c_{X,Y}} &&&& (\be \bt(X\ot Y))*V \ar[dd]^{n_{X,Y, V}} \\
((Y\ot X)\bt \be)*V  \ar[d]_{m_{Y,X, V}} &&&&  \\
(Y\bt \be)\ot ((X\bt \be)*V) \ar[rr]^>>>>>>>>>>>{\gamma_X}& &(Y\bt X)*V \ar[rr]^<<<<<<<<<<{\gamma_Y} && (\be \bt Y) \ot ((\be \bt X)*V).
}
\end{equation}
The vertical composition on the left side is the right $\B$-module associativity constraint of $\M$.
\end{remark}

\begin{proposition}
\label{module induction}
Let $\C\subset \B$ be a tensor subcategory. The induction
\begin{equation}
\label{modC to modB}
\Mod(\C)\to \Mod(\B) : \N \mapsto \B \bt_\C \N
\end{equation}
is a monoidal $2$-functor.
\end{proposition}
\begin{proof}
The monoidal structure on the $2$-functor \eqref{modC to modB} is given by the canonical equivalence
\[
 (\B \bt_\C \M) \bt_\B  (\B \bt_\C \N)  \cong \B \bt_\C (\M \bt_\B \B) \bt_\C \N \cong  \B \bt_\C (\M \bt_\C \N),\qquad \M,\N \in \Mod(\C).
\]
The verification of axioms is straightforward and is left to the reader.
\end{proof}

\section{Braided  module categories}
\label{Chapter on Braided  module categories}

\subsection{Module braiding on module categories}

Let $\B$ be a braided tensor category.  The following definition appeared in \cite{E, B,BBJ}.

\begin{definition}
\label{pb definition}
A {\em braided} $\B$-module category is a  pair $(\M,\, \sigma)$, where $\M$ is a  $\B$-module category 
and $\sigma=\{ \sigma_{X,M}:X*M\to X*M\}_{X\in \B,\, M\in \M}$  is a  natural isomorphism (called a {\em $\B$-module braiding}) 
with $\sigma_{\be,M} = 1_M$  such that the diagrams
\begin{equation}
\label{module braiding diagrams}
\xymatrix{
X*(Y*M) \ar[rr]^{\sigma_{X,Y*M}} \ar[d]_{m_{X,Y,M}^{-1}} && X*(Y*M)  \ar[d]^{m_{X,Y,M}^{-1}} \\
(X\ot Y)*M \ar[d]_{c_{X,Y}} && (X\ot Y)*M \ar[d]^{c_{Y,X}^{-1}} \\
(Y\ot X)*M \ar[d]_{m_{Y,X,M}} && (Y\ot X)*M \ar[d]^{m_{Y,X,M}} \\
Y*(X*M)  \ar[rr]^{\sigma_{X,M}} && Y*(X*M)  \\
}\qquad\qquad
\xymatrix{
(X\ot Y)*M \ar[rr]^{\sigma_{X\ot Y,M}} \ar[d]_{c_{Y,X}^{-1}} && (X\ot Y)*M \ar[d]^{c_{X,Y}} \\
(Y\ot X)*M  \ar[d]_{m_{Y,X,M}} && (Y\ot X)*M \ar[d]^{m_{Y,X,M}}\\
Y*(X*M)  \ar[dr]_{\sigma_{X,M}} && Y*(X*M)  \\
& Y*(X*M) \ar[ur]_{\sigma_{Y,X*M}}
}
\end{equation}
commute for all $X, Y\in \B$ and $M\in\M$. 
\end{definition}

\begin{definition}
A $\B$-module functor $F:\M\to\N$ between  braided $\B$-module categories is {\em  braided} if the diagram
\begin{equation}
\label{bmf diagram}
\xymatrix{
X*F(M)  \ar[rr]^{\sigma_{X, F(M)}} \ar[d]_{F_{X,M}}  && X*F(M) \ar[d]^{F_{X,M}} \\
F(X*M) \ar[rr]^{F(\sigma_{X,M})} && F(X*M)
}
\end{equation}
commutes for all $X\in \B$ and $M\in\M$.  

A morphism  between braided $B$-module functors is a $\B$-module natural transformation.
\end{definition}

Let $\Modbr(\B)$ denote the $2$-category of braided $\B$-module categories.

\begin{example} 
\label{braided extensions BC}
Let $\C$ be a braided tensor category containing $\B$. Then $\C$  is a braided $\B$-module category
with the $\B$-module braiding 
\begin{equation}
\label{sigma =cc}
\sigma_{X,Y}= c_{Y,X}c_{X,Y},\qquad X,Y\in \B, 
\end{equation}
where $c$ denotes the braiding of $\C$.
The commutativity of diagrams in Definition~\ref{pb definition} follows directly from the hexagon identities and naturality of braiding.
\end{example}

Recall that the {\em symmetric center} $\Z_{sym}(\B)$  of a braided tensor category $\B$ is  the full subcategory  of $\B$
whose objects $V$ satisfy $c_{XV}c_{VX}=\id_{X\ot V}$ for all $X$ in $B$.  Clearly, $\Z_{sym}(\B)$ is a symmetric
tensor category.

\begin{example}
\label{regular pb}
A special case of the previous example is $\C=\B$,
the regular $\B$-module category, with the module braiding \eqref{sigma =cc}.
The category of braided module endofunctors of $\B$ is braided equivalent to  $\Z_{sym}(\B)$. 

\end{example}

\begin{remark}
The name {\em  braided} in Definition~\ref{pb definition} 
is justified as follows.  Recall that the Artin braid group of type $B$ is the group $B_n$
generated by elements $\varsigma_1,\dots, \varsigma_{n}$ and relations
\begin{align*}
\varsigma_{n-1}\varsigma_n \varsigma_{n-1}\varsigma_n &= \varsigma_n \varsigma_{n-1}\varsigma_n \varsigma_{n-1},\\
\varsigma_i \varsigma_j &= \varsigma_i \varsigma_j ,\qquad |i-j|\geq 2, \\
\varsigma_i\varsigma_{i+1} \varsigma_i\ &= \varsigma_{i+1} \varsigma_i\varsigma_{i+1},\qquad i=1,\dots,n-1.
\end{align*}
Equivalently,  $B_n$ is the braid group of a once punctured disk.

Let $\sigma$ be an element of $B_n$. We will use the same letter $\sigma$ to denote the induced permutation in $S_{n-1}$.
Given objects $X_1,\dots, X_{n-1}$ in a braided tensor category $\B$ and an object $M$ in a braided $\B$-module category $\M$,
there are isomorphisms
\[
X_1\ot \cdots \ot X_{n-1} * M \to X_{\sigma(1)}\ot \dots, \ot X_{\sigma(n-1)} * M, \qquad \sigma\in B_n,
\]
compatible with the composition of braids. In particular, for any $X\in \B$ there is a homomorphism from the pure braid
group of type $B$ to   $\Aut_\M(X^{\ot (n-1)} * M)$.

\end{remark}

\begin{definition}
We say that a braided $\B$-module category is {\em indecomposable}
if it is indecomposable as a $\B$-module category.
\end{definition}

The {\em $\alpha$-inductions} \cite{BEK} for a a left $\B$-module category $\M$ are tensor functors 
\begin{equation}
\label{alpha inductions}
\alpha^\pm_\M : \B^{\opp} \to \End_\B(\M),\qquad   \alpha^\pm(X)(M) = X\ast M,\qquad X\in \B,\ M\in \M\ .
\end{equation}
Here $\End_\B(\M)$ is the category of right exact $\B$-module endofunctors of $\M$,
The $\B$-module structures on $\alpha^\pm(X)($ are given by the compositions
$$
\xymatrix{
 Y\ot  \alpha^+_\M(X)(M)   \ar[rr]^{\alpha^+_\M(X)_{Y,M}} \ar@{=}[d] &&  \alpha^+_\M(X)(Y* M) \ar@{=}[d] \\
Y* (X * M) \ar[d]_{m_{Y,X,M}^{-1}} &&  X* (Y * M) \\
(Y\ot X) * M   \ar[rr]^{c_{X,Y}^{-1}*\id_M} && (X\ot Y) * M,  \ar[u]_{m_{X,Y,M}} 
}\qquad\qquad
\xymatrix{
Y\ot   \alpha^-_\M(X)(M)  \ar[rr]^{\alpha^-_\M(X)_{Y,M}} \ar@{=}[d] &&  \alpha^-_\M(X)(Y* M) \ar@{=}[d] \\
 Y* (X * M)  \ar[d]_{m_{Y,X,M}^{-1}} && X* (Y * M)  \\
(Y\ot X) * M \ar[rr]^{c_{Y,X}*\id_M} &&  (X\ot Y) * M,  \ar[u]_{m_{X,Y,M}} 
}
$$
for all $X,Y\in\C,\, M\in  \M$ respectively.
\nl
The monoidal structures of $\alpha^\pm_\M$ are
\begin{gather*}
\alpha^+_\M(Y)(\alpha^+_\M(X)(M)) = Y* (X * M)  \xrightarrow{m_{Y,X,M}^{-1}}  (Y\ot X) * M \xrightarrow{c_{Y,X}} (X\ot Y) * M  =  \alpha^+_\M(X\ot Y)(M), \\
\alpha^-_\M(Y)(\alpha^-_\M(X)(M)) = Y* (X * M)  \xrightarrow{m_{Y,X,M}^{-1}}  (Y\ot X) * M \xrightarrow{c_{X,Y}^{-1}} (X\ot Y) * M  =  \alpha^-_\M(X\ot Y)(M).
%
\end{gather*}

\begin{remark}
\label{F apm square}
For every $\B$-module functor $F:\M\to \N$ there are natural transformations of $\B$-module functors
\begin{equation}
\label{Fa+ to a+F}
\xymatrix{
\M \ar[dd]_{\alpha_\M^\pm(X)}_{}="a" \ar[rr]^{F} && \N \ar[dd]^{\alpha_\N^\pm(X)}="b" \\
&& \\
\M \ar[rr]^{F} && \N,
\ar@{}"a";"b"^(.25){}="b2"^(.65){}="c2" \ar@{=>}^{F_{X,-}^\pm}"b2";"c2"
}
\end{equation}
for all $X\in \B$.
\end{remark}

\begin{remark}
Let $\M$ be a $\B$-module category.
The $\B$-bimodule category  $\M$ constructed in Section~\ref{sect: from mod to bimod}
can be conveniently described by means of  the functor $\alpha^+_\M:\B^\opp \to \End_\B(\M)$. 
Indeed, this functor  turns a canonical $(\B \bt \End_\B(\M))$-module category $\M$
into a $\B$-bimodule category.  Note that the functor $\alpha^-_\M$ gives rise to a  different $\B$-bimodule category 
obtained from $\M$ using
the reverse braiding of $\B$.  
\end{remark}

Let $\mathbf{A}(\B)$ denote the $2$-category
whose objects are pairs $(\M,\, \eta)$, where $\M$ is a $\B$-module category
and $\eta: \alpha^+_\M\xrightarrow{\sim} \alpha^-_\M$ is an isomorphism of tensor functors,
1-cells are $\B$-module functors $F:\M\to \N$ such that 
\begin{equation}
\label{K1 condition}
\xymatrix{
\M \ar[rr]^{F}_{}="A1" \ar@/^-2pc/[dd]_{\alpha^+_\M(X)}  && \N  \ar@/^2pc/[dd]^{\alpha^-_\N(X)}_{}="b"   \ar@/^-2pc/[dd]_{\alpha^+_\N(X)}^{}="a" & & &&
\M \ar[rr]^{F}_{}="A2"  \ar@/^-2pc/[dd]_{\alpha^+_\M(X)}^{}="A"  \ar@/^2pc/[dd]^{\alpha^-_\M(X)}_{}="B" && \N \ar@/^2pc/[dd]^{\alpha^-_\N(X)}  \\
& && &  = & & &&  \\   
\M \ar[rr]_{F}^{}="B1"  && \N & & &&\M \ar[rr]_{F}^{}="B2"  && \N,
\ar@{}"a";"b"^(.25){}="x"^(.75){}="y" \ar@{=>}^{\eta_X }"x";"y"
\ar@{}"A";"B"^(.25){}="x"^(.75){}="y"  \ar@{=>}^{\eta_X}"x";"y"
\ar@{}"A1";"B1"^(.25){}="x"^(.75){}="y" \ar@/^-2.5pc/@{=>}_{F^+_{X,-}}"x";"y"
\ar@{}"A2";"B2"^(.25){}="x"^(.75){}="y" \ar@/^2.5pc/@{=>}^{F^-_{X,-}}"x";"y"
} 
\end{equation}
for all $X\in \B$, where $F_{X,-}^\pm$ are natural isomorphisms from \eqref{Fa+ to a+F}, and $2$-cells are $B$-module natural transformations.

\begin{proposition}
\lb{cbs}
There is canonical $2$-equivalence $\Modbr(\B) \cong \mathbf{A}(\B)$.
\end{proposition}
\begin{proof}
A module braiding $\sigma_{X,M}$ on $\M$ is the same thing as
a natural isomorphism $\eta :\alpha^+_\M \xrightarrow{\sim} \alpha^-_\M$ via
\[
\eta_X(M) =\sigma_{X,M} : \alpha^+_\M(X)(M) = X*M \to X*M = \alpha^-_\M(X)(M),\qquad X\in\B,\, M\in \M.
\]
The first diagram in \eqref{module braiding diagrams} is equivalent to 
$\eta_X:\alpha^+_\M(X) \xrightarrow{\sim} \alpha^-_\M(X)$ being an isomorphism of left $\B$-module functors
and the second diagram expresses the tensor property of the natural isomorphism $\eta_\M$.  On the level of $1$-cells,
the commuting square \eqref{bmf diagram} is equivalent to the identity \eqref{K1 condition}. 
\end{proof}

\begin{remark}
A  version of Proposition~\ref{cbs} was proved by Safonov in \cite[Proposition 2.7]{S}.
\end{remark}

\subsection{$\Modbr(\B)$ as a braided monoidal $2$-category}

\begin{theorem}
\label{Modbr=ZMod}
There is a canonical $2$-equivalence $\Modbr(\B) \cong \mathbf{Z}(\Mod(\B))$. In particular, $\Modbr(\B)$ has a canonical structure
of a braided monoidal $2$-category.
\end{theorem}
\begin{proof}
In view of  Proposition~\ref{cbs} it suffices to construct a  $2$-equivalence $\mathbf{A}(\B) \cong \mathbf{Z}(\Mod(\B))$.

We construct a $2$-functor $\mathbf{A}(\B)  \to \mathbf{Z}(\Mod(\B))$ as follows. Let $(\N,\, \eta :\alpha^+_\N \xrightarrow{\sim} \alpha^-_\N)$ 
be an object of  $\mathbf{A}(\B)$. Let $A$ be an algebra in $\B$ and let $\M=\text{Mod}_\B(A)$ be the category of $A$-modules in $\B$ 
(any $\B$-module category is of this form). Then $\M\bt_\B \N \cong \text{Mod}_\N(\alpha^+_\N(A))$, where $\alpha^+_\N(A)$ is an algebra
 in $\End_\B(\N)$ and its module in $\N$ is an object $N\in \N$ along with an action $\alpha^+_\N(A)*N\to N$ satisfying usual axioms.
Similarly, $\N \bt_\B \M  \cong \text{Mod}_\N(\alpha^-_\N(A))$. Hence, the isomorphism  $\eta_A :\alpha^+_\N(A) \xrightarrow{\sim} \alpha^-_\N(A)$
of algebras in $\End_\B(\N)$ yields a pseudo-natural  $\B$-module equivalence $S_\M:\M\bt_\B \N \xrightarrow{\sim}  \N \bt_\B \M$.  

Let $\L =\text{Mod}_\B(A_1)$ and  $\M =\text{Mod}_\B(A_2)$. The invertible 
modification $\gamma_{\L,\M}$ \eqref{tilde beta} comes from the commutative diagram of algebra isomorphisms
\begin{equation}
\xymatrix{
 \alpha^+_\N(A_1) \ot  \alpha^+_\N(A_2) \ar[rrr]^{\eta_{A_1} \ot  \eta_{A_2} } \ar[d]_{(\alpha^+_\N)_{A_1 \ot A_2}} &&&  \alpha^-_\N(A_1) \ot  \alpha^-_\N(A_2) 
 \ar[d]^{(\alpha^-_\N)_{A_1 \ot A_2}}  \\
  \alpha^+_\N(A_1 \ot A_2) \ar[rrr]^{\eta_{A_1\ot A_2 } }  &&&  \alpha^-_\N(A_1 \ot A_2).
} 
\end{equation}
Note that since $\alpha^\pm_\N$ is a central functor, $\alpha^\pm_\N(A_1) \ot  \alpha^\pm_\N(A_2)$ are algebras in $\End_\B(\N)$ and 
$\eta_{A_1} \ot  \eta_{A_2}$ is an algebra isomorphism. The coherence condition \eqref{central coherence} follows from the tensor property
of $\eta$.  Thus, we have an object $(\N,\, S=\{S_\M\},\, \gamma=\{\gamma_{\L,\M}\})$ of $\mathbf{Z}(\Mod(\B))$, see Section~\ref{2-center}. 
This gives rise to  a $2$-functor
\begin{equation}
\label{A to Modbr}
\mathbf{A}(\B)  \to \mathbf{Z}(\Mod(\B)) : (\N,\, \eta) \mapsto (\N,\, S,\, \gamma).
\end{equation}

To construct a $2$-functor in the opposite direction, note that for any $X\in\End_\B(\I_{\Mod(\B)}) \cong \B^\opp$ and $\N \in \Mod(\B)$ 
the tensor functors 
\begin{eqnarray}
\label{X btB N}
&& \B^\opp \to \End_\B(\N) : X \mapsto L_\N  \circ (X \bt_\B \id_\N) \circ L_\N^{-1},  \\ 
\label{N btB X}
&&  \B^\opp  \to \End_\B(\N) : X \mapsto R_\N  \circ (\id_\N \bt_\B X) \circ R_\N^{-1},
\end{eqnarray}
where $L_\N,\, R_\N$ are the unit constraint $1$-cells in $\Mod(\B)$, are isomorphic to $\alpha_\N^+$ and $\alpha_\N^-$,
respectively.

For an object  $(\N,\, S,\, \gamma)$  in $\mathbf{Z}(\Mod(\B))$ consider the following composition of invertible $2$-cells:
\begin{equation}
\label{2-cell eta}
\xymatrix{
 && \N \ar[dll]_{L_\N^{-1}}^{}="a" \ar[drr]^{R_\N^{-1}}^{}="b"   && \\
 \B\bt_\B \N  \ar[d]_{X \bt_\B \id_\N}^{}="c" \ar[rrrr]^{S_\B} &&&& \N\bt_\B \B \ar[d]^{\id_N \bt_B X}_{}="d"  \\
  \B\bt_\B \N \ar[drr]_{L_\N}^{}="e"  \ar[rrrr]^{S_\B}  &&&& \N\bt_\B \B  \ar[dll]^{R_\N}_{}="f" \\
  && \N, &&
  \ar@{}"a";"b"^(.25){}="x"^(.75){}="y" \ar@{=>} "x";"y"
  \ar@{}"e";"f"^(.25){}="x"^(.75){}="y" \ar@{=>} "x";"y"
   \ar@{}"c";"d"^(.25){}="x"^(.75){}="y" \ar@{=>}^{S_X} "x";"y"
} 
\end{equation}
where the top and bottom triangles are canonical $2$-cells coming from the unit constraints of $\Mod(\B)$
and $S_X$ is the half-braiding $2$-cell. The outside compositions are \eqref{X btB N} and \eqref{N btB X}. Thus, $2$-cells
\eqref{2-cell eta} give an isomorphism of $\B$-module functors $\eta_X:  \alpha_\N^+(X) \xrightarrow{\sim} \alpha_\N^-(X),\, X\in \B$. 
The  multiplicative property \eqref{pi fg} of the pseudo-natural transformation $S$ implies that $\eta$ is an isomorphism of tensor functors.

This gives a  $2$-functor
\begin{equation}
\label{Modbr to A}
\mathbf{Z}(\Mod(\B))  \to \mathbf{A}(\B)   : (\N,\, S,\, \gamma) \mapsto (\N,\, \eta_\N)
\end{equation}
quasi-inverse to \eqref{A to Modbr}.

This  resulting $2$-equivalence $\Modbr(\B) \cong \mathbf{Z}(\Mod(\B))$  obtained using  Proposition~\ref{cbs} 
induces on $\Modbr(\B)$ a structure of a braided monoidal $2$-category.
\end{proof}

For braided $\B$-module categories  $\M:=(\M,\, \sigma^\M) $ and $\N:=(\N,\, \sigma^\N)$  the braiding
\[
B_{\M,\N} :\M \bt_\B \N \xrightarrow{\sim} \N \bt_\B \M
\]
in $\Modbr(\B)$ is given by  the half braiding of $\N$.

\begin{proposition}
A braided tensor functor $F:\B \to \C$ induces a braided monoidal $2$-functor 
\[
\Modbr(\C)\to \Modbr(\B) : (\M,\, \sigma) \mapsto (\tilde{\M},\, \tilde{\sigma}),
\]
where $\tilde{\M}=\M$ with the action $X*M =F(X)*M$ and $\tilde{\sigma}_{X,M} = \sigma_{X,M},\, X\in \B,\, M\in \M$.
\end{proposition}
\begin{proof}
This is verified directly using the $\C$-module braiding axioms of $\sigma$ and the braided property of~$F$.
\end{proof}

\begin{remark}
The braided monoidal $2$-category structure on  $\Modbr(\B)$ constructed in  Theorem~\ref{Modbr=ZMod} can also be described explicitly
as follows. Let $(\M,\, \sigma^\M) $ and $(\N,\, \sigma^\N)$  be braided $\B$-module categories.
Recall that objects of $\M\bt_\B \N$ are pairs $(V,\, \gamma)$, where $V\in \M\bt \N$ and 
\[
\gamma_X: (X \bt \be)*V \xrightarrow{\sim} (\be \bt X)*V,\, X\in \B,
\]
is a balancing isomorphism satisfying \eqref{MbtBN}.

The tensor product of $\Modbr(\B)$ is
\[
(\M,\, \sigma^\M)  \bt_\B (\N,\, \sigma^\N) = (\M \bt_\B \N,\, \sigma^{\M \bt_\B \N}),
 \]
where 
\[
\sigma^{\M \bt_\B \N}_{X,(V,\gamma)}: (X \bt \be)* (V,\,\gamma) \to (X \bt \be)*(V,\,\gamma),\qquad X\in \B,\, (V,\,\gamma) \in \M \bt_\B \N,
\]
is given by the composition
\[
(X\bt 1)\ot V \xrightarrow{\gamma_X} (1\bt X)\ot V 
 \xrightarrow{\sigma^\N_{X,V}}  (1\bt X)\ot V \xrightarrow{\gamma_X^{-1}} (X\bt 1)\ot V \xrightarrow{\sigma^\M_{X,V}} (X\bt 1)\ot V 
\] 
The unit object for this tensor product is the regular braided $\B$-module category from Example~\ref{regular pb}.

The braiding is 
\[
B_{(\M,\, \sigma^\M), (\N,\, \sigma^\N)} : (\M,\, \sigma^\M)  \bt_\B (\N,\, \sigma^\N)  \xrightarrow{\sim} (\N,\, \sigma^\N)  \bt_\B (\M,\, \sigma^\M),\quad
(V,\, \gamma) \mapsto (V^t,\, \tilde\gamma), 
\]
where $\M\bt \N \to \N \bt \M : V \mapsto V^t$ is   the transposition functor, i.e. $V^t= N \bt M$ for $V =M \bt N$ (this 
extends to $\M \bt \N$ thanks to the universal property of $\bt$) and 
\[
\tilde\gamma_X : (X \bt \be) * V^t \xrightarrow{\sigma^\N_{X,V}} (X \bt \be) * V^t \xrightarrow{(\gamma_X^t)^{-1}} (\be\bt X)* V^t,\quad X\in \B.
\]
\end{remark}

\subsection{Examples and basic properties of braided  module categories}
\label{basics of brmods}

Let $\B$ be a braided tensor category. 

\begin{example}
\label{ModbrO}
The  regular braided $\B$-module category $\B$ from Example~\ref{regular pb}
is the unit object of $\Modbr(\B)$. It generates a braided monoidal $2$-category $\ModbrO(\B)$ whose objects are 
direct sums of copies of $\B$ (identified with natural numbers), $1$-cells are matrices of objects in $\Z_{sym}(\B)$,
and $2$-cells are matrices of morphisms in $\B$. The tensor product is given by the Kronecker product
of such matrices while the braiding $2$-cells are given by the braiding of  $\Z_{sym}(\B)$.
\end{example}

\begin{example}
\label{ModbrI}
Note that $\B$ can have other $\B$-module braidings, in addition to one from Example~\ref{regular pb}.
Namely, it follows from the first diagram  in \eqref{module braiding diagrams}  that a $\B$-module braiding  $\sigma$ 
on $\B$  satisfies 
\[
\sigma_{X,Y} =\sigma_{X,\be} c_{Y,X}c_{X,Y},\, X,Y\in \B.
\]
The second diagram  in \eqref{module braiding diagrams} 
is equivalent to $\sigma_{X,\be}$  being a tensor automorphism of $\id_\B$. Conversely, any $\nu \in \Aut_\ot(\id_\B)$ yields
a module braiding 
\begin{equation}
\sigma^\nu_{X,Y}= (\nu_{X} \ot \be) c_{Y,X}c_{X,Y},\, X,Y\in \B.
\end{equation}
Let $\B^\nu := (\B,\, \sigma^\nu)$ denote the corresponding braided $\B$-module category.

There is an exact sequence \cite[3.3.4]{DGNO} of groups
\begin{equation}
1 \to \Inv(\Z_{sym}(\B)) \to \Inv(\B) \xrightarrow{\alpha}  \Aut(\id_\B),
\end{equation}
where 
\begin{equation}
\label{DGNO a}
\alpha(Z)_X =c_{Z,X} c_{X,Z}  \in \Aut(X\ot Z)= k^\times
\end{equation}
 for every simple object $X\in \B$. Two braided $\B$-module categories $\B^{\nu_1}$
and $\B^{\nu_2},\, \nu_1,\nu_2 \in \Aut(\id_\B),$ are equivalent  if and only if $\nu_2 = \nu_1 \alpha(Z)$ for some $Z\in \Inv(\B)$. Thus, 
the group of equivalence classes of  braided $\B$-module categories of the form  $\B^\nu$  is isomorphic 
to  $\Coker(\Inv(\B) \xrightarrow{\alpha}  \Aut(\id_\B))$.

Let $\ModbrI(\B)$ denote the full  braided monoidal $2$-subcategory of $\Modbr(\B)$ generated by braided $\B$-module categories  $\B^\nu$.
\end{example}

\begin{example}
\label{Bnu}
Let $\M$ be an exact $\B$-module category. The $2$-categorical half-braiding  
\begin{equation}
\label{SMnu}
\M \simeq \M \bt_\B \B^\nu \xrightarrow{S_\M^\nu}  \B^\nu  \bt_\B \M \simeq \M
\end{equation}
is identified with  the image of $\nu$ under the composition
\begin{equation}
\label{SMnu defined}
\Aut_\ot(\id_\B) \xrightarrow{\iota} \Inv(\Z(\B)) \cong \Inv(\Z(\End_\B(\M)) \to \Inv(\End_\B(\M)) = \Aut_\B(\M),
\end{equation}
where $\iota(\nu) =\be$ as an object of $\B$ with the half-braiding
\[
\nu_X \id_X: X \cong X \ot \iota(\nu) \xrightarrow{\sim} \iota(\nu)  \ot X \cong X,\quad X\in \B.
\]
\end{example}

For the trivial tensor category $\B=\Vect$ we have $\Modbr(\Vect) = \ModbrO(\Vect) = \ModbrI(\Vect)$, the 2-category of 2-vector spaces. The objects of this category
are natural numbers, $1$-cells are matrices of vector spaces, and $2$-cells are matrices of linear transformations.

The following result was established in \cite{LKW} using different methods and terminology. 

\begin{proposition}
\label{ModbrVect}
Let $\B$ be a non-degenerate braided fusion category. There is an equivalence of braided monoidal $2$-categories
\begin{equation}
\Modbr(\B)  \simeq \Modbr(\Vect). 
\end{equation} 
\end{proposition}
\begin{proof} 
Let $\M$ be an indecomposable  $\B$-module category. The tensor functors $\alpha_\M^\pm: \B \to \End_\B(\M)$
are given by the compositions
\begin{equation}
\label{BtoEnd}
\B \to \Z(\B) \simeq \Z(\End_\B(\M)) \to \End_\B(\M),
\end{equation}
where the first functor is the embedding of $\B$ (respectively, $\B^\rev$) into $\Z(\B)$, and the last one
is the forgetful functor.  The images of $\B$ and $\B^\rev$ generate $\Z(\End_\B(\M))$. If $\M$ has a $\B$-module braiding,
it follows from Proposition~\ref{cbs} that  the full images of
$\alpha_\M^+(\B)$ and $\alpha_\M^-(\B)$ in $\End_\B(\M)$ coincide. Since the forgetful functor is surjective, we have
$\alpha_\M^\pm(\B) = \End_\B(\M)$. Thus, $\M$ is an invertible $\B$-module category such that the  braided autoequivalence
$\partial_\M:= (\alpha_\M^+)^{-1}  \circ  \alpha_\M^-\in \Aut^{br}(\B)$ is trivial. It follows from \cite{DN, ENO} that $\M \simeq \B$
as a $\B$-module category, i.e.  $\Modbr(\B)  \simeq  \ModbrI(\B)$.  Since the homomorphism $\alpha$ from Example~\ref{ModbrI}
is an isomorphism for a non-degenerate category $\B$, 
we have $\ModbrI(\B)  \simeq  \ModbrO(\B)$.  Since $\Z_{sym}(\B)=\Vect$, $\ModbrO(\B)$  is $2$-equivalent to $\Modbr(\Vect)$,
and  the statement follows.
\end{proof}

\subsection{The symmetric monoidal $2$-category of symmetric  module categories}

Let $\E$ be a symmetric tensor category.  For any $\E$-module category $\M$ we have $\alpha_\M^+ = \alpha_\M^-$.
In particular any $\E$-module category $\M$ has  the identity module braiding $\id_{X\ot M}$. 

\begin{definition}
A braided $\E$-module category $(\M,\, \sigma)$ is called {\em symmetric} if $\sigma_{X,M}=\id_{X\ot M}$ for all
$X\in \E$ and $M\in \M$.
\end{definition}

\begin{example} 
Let $\C$ be a symmetric braided tensor category containing $\E$. Then $\C$ is a symmetric $\E$-module category.
\end{example}

Clearly, the tensor product of symmetric module categories is symmetric.
We will denote   $\Modsym(\E)$ the symmetric monoidal $2$-category of  symmetric $\E$-module categories
(its double braiding $2$-cells \eqref{symmetry 2-cell} are identities). Note that $\Modsym(\E) =\Mod(\E)$ as a monoidal $2$-category
and can also be viewed as a braided monoidal $2$-subcategory of $\Modbr(\E)$.

\begin{remark}
Let $\B$ be a braided tensor category. It follows from Example~\ref{Bnu} that
the braided monoidal $2$-category $\ModbrI(\B)$ has a symmetric structure.
\end{remark}

\begin{proposition}
\label{inducing to center}
Let  $\E$ be a tensor subcategory of $\Z_{sym}(\B)$. The induction
\begin{equation}
\label{braided module induction}
\Mod(\E)\to \Modbr(\B) : \N \mapsto \B \bt_\E \N
\end{equation}
is a braided monoidal $2$-functor.
\end{proposition}
\begin{proof}
Let $\N= \text{Mod}_\E(A)\in \Mod(\E)$ for some algebra $A\in \E$. Then $ \B \bt_\E \N = \text{Mod}_\B(A)$.  
For any $\M \in \Mod(\B)$  the composition of $\B$-module equivalences
\[
 \M \bt_\B (\B \bt_\E \N) \cong \text{Mod}_\M(\alpha_\M^-(A)) = \text{Mod}_\M(\alpha_\M^+(A))  \cong (\B \bt_\E \N) \bt_\B \M,
\]
where the equality in the middle is due to the fact that $A\in \Z_{sym}(\B)$, defines a half braiding on  $\B \bt_\E \N$.
It follows that $\B \bt_\E \N$ is a braided $\B$-module category and the monoidal induction $2$-functor from 
 Proposition~\ref{module induction} lifts to a braided monoidal $2$-functor
\eqref{braided module induction}.
\end{proof}

\section{$2$-categorical Picard groups }
\label{Chapter on Picards}

In this Section we describe categorical $2$-groups of module categories over tensor categories in terms introduced in Section~\ref{2cg}.

\subsection{The $2$-categorical Brauer-Picard group of a tensor category}

Let $\D$ be a tensor category. Recall from \cite[Section 4.1]{ENO} that a $\D$-bimodule category $\M$
is invertible with respect to $\bt_\D$ if and only if  $\M^o \bt_\D \M \cong \D$ and $\M \bt_\D \M^o \cong \D$, where $\M^o$ is the opposite
Abelian category of $\M$ with the left (respectively, right) $\D$-module actions of $X\in \D$ given by the right (respectively, left) 
actions of ${}^*X$ (see also \cite{DSS}).

The {\em $2$-categorical  Brauer-Picard group} \cite{ENO} of a tensor category $\D$ is 
\begin{equation}
\uuBrPic(\D)=\uuInv(\mathbf{Bimod}(\D)).
\end{equation}
Its objects are invertible $\D$-bimodule categories, $1$-cells are $\D$-bimodule equivalences, and $2$-cells
are isomorphisms of $\D$-bimodule equivalences. The tensor product is $\bt_\D$ and the unit
object is the regular $\D$-bimodule category. Let $\uBrPic(\D)$ denote the categorical group obtained by truncating 
$\uuBrPic(\D)$ and let $\BrPic(\D)$ denote the group of isomorphism classes of objects.

The homotopy groups of $\uuBrPic(\D)$  are
\begin{eqnarray}
&& \pi_0(\uuBrPic(\D)) = \BrPic(\D) \cong \Aut^{br}(\Z(\D)), \\
&& \pi_1(\uuBrPic(\D)) =  \Inv(\Z(\D)),\\
&& \pi_2(\uuBrPic(\D)) =  k^\times.
\end{eqnarray}

The $1$-categorical truncations of $\uuBrPic(\D)$ are 
\begin{eqnarray}
&& \varPi_{\leq 1}(\uuBrPic(\D)) = \uBrPic(\D) \cong \uAut^{br}(\Z(\D)), \\
&& \varPi_{1\leq}(\uuBrPic(\D)) = \uInv(\Z(\D)).
\end{eqnarray}

The first canonical class is the associator $\alpha_{\uuBrPic(\D)} \in H^3(\Aut^{br}(\Z(\D)),\, \Inv(\Z(\D)))$ of the categorical group $\uAut^{br}(\Z(\D))$.
The second canonical class is the braided associator  $q_{\uuBrPic(\D)} \in H^3_{br}( \Inv(\Z(\D)),\, k^\times)$
of  the braided categorical group $\uInv(\Z(\D))$, 
corresponding to  the quadratic form  
\[
q_{\uuBrPic(\D)}:  \Inv(\Z(\D)) \to k^\times : Z\mapsto c_{Z,Z}.
\]

The monoidal functor $\varPi_{\leq 1}(\uuBrPic(\D))  \to \A ut^{br}(\varPi_{1\leq}(\uuBrPic(\D)))$ 
coincides with the composition
$\uBrPic(\D) \cong \A ut^{br}(\Z(\D)) \to \A ut^{br}(\uInv(\Z(\D)))$ \cite{DN, ENO}. 

The non-trivial Whitehead brackets are the maps $\pi_0\times \pi_1 \to \pi_1$ and $\pi_1\times \pi_1 \to \pi_2$ given by
\begin{eqnarray}
&& \Aut^{br}(\Z(\D)) \times  \Inv(\Z(\D)) \to  \Inv(\Z(\D)) : (F,\, Z)\mapsto F(Z), \\
&&\Inv(\Z(\D)) \times  \Inv(\Z(\D))  \to k^\times : (Z,\, W) \mapsto c_{W,Z} c_{Z,W}.
\end{eqnarray}
Here $c$ denotes the braiding of $\Z(\D)$.

\subsection{The $2$-categorical Picard group of a braided tensor category}
\label{Section Pic(B)}

Let $\B$ be a braided tensor category.
Recall \cite{DN, ENO} that a $\B$-module category $\M$ is invertible if and only if the $\alpha$-induction tensor 
functors $\alpha_\M^\pm: \B^\opp \to \End_\B(\M)$, see \eqref{alpha inductions},  are equivalences. Here 
$\End_\B(\M)$ denotes the category of right exact $\B$-module endofunctors of $\M$.

Recall that a $\B$-module category $\M$ is {\em exact} \cite{EO1} if for any projective object $P\in \B$ and any object
$M\in \M$ the object $P\ot M\in \M$ is projective. For an exact $\M$, the dual category $\End_\B(\M)$ is a multitensor category.
The tensor product of functors is their composition and the left and right duals of a $\C$-module functor $F:\M\to \M$ are its left and right adjoints.

The following result was explained to us by Victor Ostrik.

\begin{proposition}
\label{inv--> exact}
An invertible $\B$-module category is exact.
\end{proposition}
\begin{proof}
Let $\M$ be an invertible $\B$-module category. It is equivalent to $\text{Mod}_\B(A)$ for some algebra $A$ in $\B$ and
\[
\text{Bimod}_\B(A) \cong \End_\B(\M) \cong \B^\opp.
\]
So the tensor product over $A$
is exact on the category of $A$-bimodules. This implies that it is exact
for right modules tensored with left modules (as any right $A$-module $M$ can be made into a bimodule $A\otimes M$ and similarly for left modules).
By \cite[Proposition 7.9.7(1) and Example 7.9.8]{EGNO}, this is equivalent to exactness of the internal $\Hom$ functor
\[
\M \to \B: M \mapsto \uHom( N,\, M)\qquad \text{for all $N\in  \text{Mod}_\B(A)$}
\]
and, hence, to exactness of $\M$.
\end{proof}

\begin{corollary}
\label{inv Cbimod is exact}
Let $\C$ be a finite tensor category. An invertible $\C$-bimodule category is exact.
\end{corollary}
\begin{proof}
A canonical monoidal $2$-equivalence between $\uuPic(\Z(\C))$ and $\uuBrPic(\C)$ preserves exactness by 
\cite[Theorem 3.31]{EO1}.
\end{proof}

The {\em $2$-categorical Picard group} \cite{DN, ENO} of $\B$ is 
\begin{equation}
\uuPic(\B)=\uuInv(\Mod(\B)).
\end{equation}
Its objects are invertible $\B$-module categories, $1$-cells are $\B$-module equivalences, and $2$-cells
are isomorphisms of $\B$-module equivalences. The tensor product is $\bt_\B$ and the unit
object is the regular $\B$-module category. Let $\uPic(\B)$ denote the categorical group obtained by modding out
$\uuPic(\B)$ by $2$-morphisms and let $\Pic(\B)$ denote the group of isomorphism classes of objects.

The homotopy groups of $\uuPic(\B)$  are
\begin{eqnarray}
&& \pi_0(\uuPic(\B)) = \Pic(\B), \\
&& \pi_1(\uuPic(\B)) =  \Inv(\B),\\
&& \pi_2(\uuPic(\B)) =  k^\times.
\end{eqnarray}

The $1$-categorical truncations of $\uuPic(\B)$ are 
\begin{eqnarray}
&& \varPi_{\leq 1}(\uuPic(\B)) = \uPic(\B), \\
&& \varPi_{1\leq}(\uuPic(\B)) = \uInv(\B).
\end{eqnarray}

The first canonical class is the associator $\alpha_{\uuPic(\B)} \in H^3(\Pic(\B),\, \Inv(\B))$ of the categorical group $\uPic(\B)$.
The second canonical class is the braided associator $q_{\uuPic(\B)} \in H^3_{br}( \Inv(\B),\, k^\times)$
of the braided categorical group $\uInv(\B)$,
corresponding to the quadratic form  
\[
q_{\uuPic(\B)}: \Inv(\B) \to k^\times : Z\mapsto c_{Z,Z}.
\]

The monoidal functor $\varPi_{\leq 1}(\uuPic(\B))  \to \Aut^{br}(\varPi_{1\leq}(\uuPic(\B)))$ 
coincides with the composition 
$ \uPic(\B) \to \A ut^{br}(\B) \to \A ut^{br}(\uInv(\B))$ \cite{DN, ENO}. 

The non-trivial Whitehead brackets are the maps $\pi_0\times \pi_1 \to \pi_1$ and $\pi_1\times \pi_1 \to \pi_2$ given by
\begin{eqnarray}
&& \Aut^{br}(\B) \times  \Inv(\B) \to  \Inv(\B) : (F,\, Z)\mapsto F(Z), \\
&&\Inv(\B) \times  \Inv(\B)  \to k^\times : (Z,\, W) \mapsto c_{W,Z} c_{Z,W}.
\end{eqnarray}
Here $c$ denotes the braiding of $\B$.

For any tensor category $\D$ there is a monoidal $2$-equivalence $\uuBrPic(\D)\cong \uuPic(\Z(\D))$ \cite[Theorem 5.2]{ENO}.
Thus, $2$-categorical  Picard groups generalize Brauer-Picard groups.

\subsection{The braided $2$-categorical Picard group of a braided tensor category}
\label{section: pb and AB}

The  {\em braided $2$-categorical Picard group} of a braided tensor category $\B$ is 
\begin{equation}
\uuPicbr(\B)=\uuInv(\Modbr(\B)) \cong \uuInv(\mathbf{Z}(\Mod(\B))),
\end{equation}
where the last $2$-equivalence is by Theorem~\ref{Modbr=ZMod}.
Its objects are invertible braided $\B$-module categories, $1$-cells are braided $\B$-module equivalences, and $2$-cells
are natural isomorphisms of $\B$-module equivalences. The tensor product is $\bt_\B$  and the unit
object is the regular braided $\B$-module category (see Example~\ref{pb definition}). 
Let $\uPicbr(\B)$ denote the braided categorical group obtained by by modding out
$\uuPicbr(\B)$ by $2$-morphisms  and let $\Picbr(\B)$ denote the group of isomorphism classes of objects.

The homotopy groups of $\uuPicbr(\B)$  are
\begin{eqnarray}
&& \pi_0(\uuPic(\B)) = \Picbr(\B), \\
&& \pi_1(\uuPic(\B)) =  \Inv(\Z_{sym}(\B)),\\
&& \pi_2(\uuPic(\B)) =  k^\times.
\end{eqnarray}

The $1$-categorical truncations of $\uuPicbr(\B)$ are 
\begin{eqnarray}
&& \varPi_{\leq 1}(\uuPic(\B)) = \uPicbr(\B), \\
&& \varPi_{1\leq}(\uuPic(\B)) = \uInv(\Z_{sym}(\B)).
\end{eqnarray}

The  first canonical class is the braided associator $\alpha_{\uuPicbr(\B)} \in H^3_{br}(\Picbr(\B),\,  \Inv(\Z_{sym}(\B)))$ 
of the braided categorical group $\uuPicbr(\B)$ corresponding to  the quadratic function  
\[
Q_{\uuPicbr(\B)}:  \Picbr(\B) \to  \Inv(\Z_{sym}(\B)) : \M \mapsto B_{\M,\M},
\]
where $B$ denotes the braiding of $\uuPicbr(\B)$.  The second canonical class is the
symmetric associator $q_{\uuPicbr(\B)} \in H^3_{sym}(  \Inv(\Z_{sym}(\B)),\, k^\times)$
of the symmetric categorical group $ \uInv(\Z_{sym}(\B))$ corresponding to the homomorphism
\[
q_{\uuPicbr(\B)}  : \Inv(\Z_{sym}(\B)) \to \{ \pm 1\} \subset k^\times : Z \mapsto c_{Z,Z}.
\]

\begin{proposition}
\label{PB is well defined} 
Let  $(\M,\, \sigma^\M)$ be an indecomposable  braided $\B$-module category
and let $Z$ be an invertible object in $\Z_{sym}(\B)$.
The Whitehead bracket   $[\, ,\, ]: \pi_0\times \pi_1 \to \pi_2$ \eqref{new pairing}
of $\uuPicbr(\B)$  satisfies 
\begin{equation}
\label{new pairing Picbr}
\sigma^\M_{Z,\, M}=[\M,\, Z]\, \id_{Z*M}
\end{equation}
for all objects $M\in\M$.
\end{proposition}
\begin{proof}
For any simple $M$ in $\M$ we identify  $\sigma^\M_{Z,M}\in \Aut(Z\ot M)$ with a non-zero scalar.
It suffices to check that this scalar does not in fact depend on $M$. Note that
\[
\sigma^\M_{Z,\, X\ot M}=  c_{Z,X} c_{X,Z} \sigma^\M_{Z,M} = \sigma^\M_{Z,M}
\]
for any simple object $X\in \B$. Since every simple object $N$ of $\M$ is contained in some $X\ot M$ we conclude that
$\sigma^\M_{Z,\, M} = \sigma^\M_{Z,\, N}$.
\end{proof}

Recall from \cite{DN, ENO} a monoidal functor  
\begin{equation}
\label{partial recalled}
\partial:\uPic(\B)\to \uAut^{br}(\B):  \M \mapsto (\alpha_{\M}^+)^{-1} \circ \alpha_{\M}^-,
\end{equation}
where $\alpha_\M^\pm:\B^\opp \to \End_\B(\M)$ are equivalences \eqref{alpha inductions}.

\begin{proposition}
\label{exs}
There is an exact sequence
\begin{equation}
\label{exact sequence}
 1 \to \Inv(\Z_{sym}(\B))   \xrightarrow{}  \Inv(\B) \xrightarrow{\alpha} \Aut_\ot(\id_\B)  \xrightarrow{\eps} 
 \Pic_{br}(\B)  \xrightarrow{\phi} \Pic(\B) \xrightarrow{\partial} \Aut_{br}(\B),
\end{equation}
where  $\alpha$ is defined in \eqref{DGNO a},
 $\eps(\nu) = \B^{\nu}$ (see Example~\ref{Bnu}),  and 
 $\phi(\M,\, \sigma)=\M$.
\end{proposition}
\begin{proof}
This is an immediate consequence of the definitions.
\end{proof}

\begin{remark}
By the {\em fiber}  of the monoidal functor $F:\G\to\H$ between groupoids we mean the category of pairs $(X,x)$, where $X\in \G$ and $x:F(G)\to I$ for the unit object $I\in \H$. 
It follows from Proposition~\ref{cbs} that the fiber of the monoidal functor $\uPic(\B)\to \uAut(\B)$ coincides with $\uPic_{br}(\B)$. 
The exact sequence \eqref{exact sequence} can be seen as the Serre exact sequence of homotopy groups of the fibration of categorical groups
\[
\uPic_{br}(\B) \to  \uPic(\B)\to \uAut(\B).
\]
\end{remark}

\begin{example}
\label{uuPicbrI}
Let $\uuPicbrI(\B)$ be the braided $2$-categorical subgroup of $\uuPicbr(\B)$ consisting of braided $\B$-module categories
whose underlying $\B$-module category is the regular $\B$-module category $\B$.  That is, $\uuPicbrI(\B) =\textbf{Inv}(\ModbrI(\B))$, 
see Example~\ref{ModbrI}.

The objects of $\uuPicbrI(\B)$ are braided $\B$-module categories $\B^\nu,\, \nu\in \Aut_{\ot}(\id_\B)$. The module braiding of $\B^\nu$ is
\begin{equation}
\label{sigma-nu}
\sigma^\nu_{X,M}=(\nu_X \ot \id_M)\circ c_{M,X}c_{X,M},\quad X,M\in \B.
\end{equation}
The homotopy groups of $\uuPicbrI(\B)$ are 
\begin{eqnarray}
&& \pi_0(\uuPicbrI(\B)) = \text{Coker}(\Inv(\B) \xrightarrow{\partial}  \Aut_{\ot}(\id_\B)), \\
&& \pi_1(\uuPicbrI(\B)) =  \Inv(\Z_{sym}(\B)),\\
&& \pi_2(\uuPicbrI(\B)) =  k^\times.
\end{eqnarray}
The following  is a convenient ``non-skeletal'' description of $\uuPicbrI(\B)$.  The objects $\B^\nu$ correspond to elements of $\Aut_{\ot}(\id_\B)$,
$1$-morphisms are given by
\begin{equation}
\label{c2mu =nu}
\uuPicbrI(\B)(\B^{\mu},\, \B^{\nu}) = \{ Z\in \uInv(\B)) \mid c_{Z,X} \circ  c_{X,Z} \circ  (\mu_X \ot \id_Z) = \nu_X \ot \id_Z,\, X\in \B \},
\end{equation}
and $2$-cells are isomorphisms between invertible objects of  $\B$. The tensor product is  $\B^\mu \bt \B^\nu := \B^{\mu\mu}$ for all
$\mu,\nu\in \Aut_{\ot}(\id_\B)$. The associativity and braiding $2$-cells are identities and the pseudo-naturality $2$-cell for the tensor product is
\[
\bt_{Z,W} =c_{Z,W},\quad Z,W\in \uInv(\B),
\]
where $c$ denotes the braiding of $\B$. We  also have  
\[
B_{Z, \B^\nu} = \nu_Z,\quad Z\in \uInv(\B),\quad \nu\in \Aut_\ot(\B).
\] 
All other structural $2$-cells are identities.
\end{example}

\subsection{The symmetric $2$-categorical Picard group of a symmetric tensor category}

Let $\E$ be a symmetric tensor category.
The  {\em symmetric $2$-categorical Picard group} of  $\E$ is 
\begin{equation}
\uuPicsym(\E)=\uuInv(\Modsym(\E)) = \uuInv(\Mod(\E)) = \uuPic(\E).
\end{equation}
Its objects are invertible symmetric $\E$-module categories, $1$-cells are $\E$-module equivalences, and $2$-cells
are natural isomorphisms of $\E$-module equivalences. 
Let $\uPicsym(\E)$ denote the categorical group obtained by truncating 
$\uuPicsym(\E)$ and let $\Picsym(\E)$ denote the group of isomorphism classes of objects.

The homotopy groups of $\uuPicsym(\E)$  are
\begin{eqnarray}
&& \pi_0(\uuPic(\E)) = \Picsym(\E) =\Pic(\E), \\
&& \pi_1(\uuPic(\E)) =  \Inv(\E),\\
&& \pi_2(\uuPic(\E)) =  k^\times.
\end{eqnarray}

The $1$-categorical truncations of $\uuPicsym(\E)$ are 
\begin{eqnarray}
&& \varPi_{\leq 1}(\uuPic(\E)) = \uPicsym(\E) =\uPic(\E),   \\
&& \varPi_{1\leq}(\uuPic(\E)) = \uInv(\E).
\end{eqnarray}

The first canonical class is the symmetric associator $\alpha_{\uuPicsym(\B)} \in H^3_{sym}(\Pic(\E),\,  \Inv(\E))$ 
of the symmetric categorical group $\uuPic(\B)$ corresponding to  the homomorphism 
\[
Q_{\uuPicsym(\E)}:  \Pic(\E) \to  \Inv(\E)_2 : \M \mapsto B_{\M,\M},
\]
where $B$ denotes the braiding of $\uuPic(\E)$.  The second canonical class is the
symmetric associator $q_{\uuPic(\E)} \in H^3_{sym}( \Inv(\E),\, k^\times)$
of the braided categorical group $ \uInv(\E)$ corresponding to the homomorphism 
\[
q_{\uuPicsym(\E)}  : \Inv(\E) \to \{ \pm 1\} \subset k^\times : Z \mapsto c_{Z,Z}.
\]

\subsection{Azumaya algebras in braided tensor categories}
\label{appendix AzAlg}

Let $R$ be an algebra in a braided tensor  category $\B$, i.e. an object together with morphisms $\mu:R\ot R\to R$ (the {\em product}) 
and $\iota:I\to R$ (the {\em unit} map) satisfying the associativity and  unit conditions.

Denote by $\Aut_{alg}(R)$ the group of algebra automorphisms of $R$.
\begin{remark}
The assignment 
\beq
\lb{gta}
\Aut_\ot(\id_\B)\to \Aut_{alg}(R) : a\mapsto a_R
\eeq 
is a group homomorphism.
\ere

Let $M$ be a right $R$-module in $\B$ with the structural map $\rho : M \ot R \to M$. 
For any  $X\in\B$ there is an $R$-module structure on  $X\otimes M$  defined by 
\[
\id_X\otimes\rho:X\otimes M\otimes R \to X\otimes M.
\]  
Thus, the category $\B_R$ of right $R$-modules in $\B$ 
is a left $\B$-module category via
$$
\B \times \B_R \to \B_R, \quad (X,\,M)\mapsto X\otimes M.
$$
The $\alpha$-induction functors \eqref{alpha inductions} for  $\B_R$  are 
\begin{equation}\lb{aind}
\alpha_{\B_R}^\pm : \B \to {_R}\B_R = \End_\B(\B_R)^{\opp} : X \mapsto X\ot R,
\end{equation}
with the obvious right $R$-module structures and the left $R$-module structures  given by
\begin{eqnarray*}
R \ot X \ot R  & & \xrightarrow{c_{R,X} \ot \id_R}  X \ot R \ot R \xrightarrow{\id_X \ot \mu} X \ot R,\\
R \ot X \ot R  & & \xrightarrow{c_{X,R}^{-1} \ot \id_A}  X \ot R \ot R \xrightarrow{\id_X \ot \mu} X \ot R, \qquad X\in \B
\end{eqnarray*}
for $\alpha_{\B_R}^+$ and $\alpha_{\B_R}^-$, respectively.

The tensor product $R\ot S$ of two algebras $R,S\in\B$ has an algebra structure, 
with the multiplication map $\mu_{R\ot S}$ defined as 
$$
R\ot S\ot R\ot S \xrightarrow{\id_R\ot  c_{S,R}\ot \id_S }  R\ot R\ot S\ot S \xrightarrow{\mu_R\ot \mu_S}  R\ot S,
$$ 
where $\mu_R$ and $\mu_S$ are multiplications of algebras $R$ and $S$, respectively (here we
suppress the associativity constraints in $\B$). We have
\[
 \B_R \bt_\B \B_S \cong \B_{R\ot S}.
\]

Let $R^{op}=R$ denote the algebra with the multiplication  opposite to that of $R$:
\[
R\ot R \xrightarrow{c_{R,R}} R\ot R \xrightarrow{\mu} R.
\]

Following \cite{OZ}, we say that an algebra $R$ in a braided monoidal category $\B$ is {\em Azumaya} if the morphism
\[
R\ot R^{\opp}\ot R \xrightarrow{\id_R \ot c_{R,R}}   R\ot R\ot R \xrightarrow{\mu\ot \id_R} R\ot R \xrightarrow{\mu} R
\]
induces an isomorphism  $R\ot R^{\opp}\to R\ot R^*$.   The $B$-module category $\B_R$ is invertible in $\uuPic(\B)$
if and only if $R$ is an Azumaya algebra (in which case $\alpha_{\B_R}^\pm$ are equivalences).

Thus, the $2$-categorical Picard group of $\uuPic(\B)$ is monoidally $2$-equivalent
to the group of Morita equivalence classes of exact Azumaya algebras in $\B$ (the latter group was called in
\cite{OZ} the Brauer group of $\B$). 

It was shown in \cite[Theorem 3.1]{OZ} that for an Azumaya algebra $R$ the functors \eqref{ainb}
are monoidal equivalences.

For an Azmaya algebra $R\in \B$ and an automorphism $\phi \in \Aut_{alg}(R)$ let ${_\phi R}$ be the invertible $R$-bimodule
obtained from $R$ by twisting the right $R$-action by $\phi$.  Under the equivalence $\alpha_{\B_R}^\pm$
it corresponds to an invertible object $P_\phi\in \B$ and we have a group homomorphism
\begin{equation}
\label{ati}
\Aut_{alg}(R)\to \Inv(\B): \phi \mapsto P_\phi.
\end{equation}

\begin{remark}
\label{efp}
An isomorphism of $R$-bimodules $f:P\ot R\to {_\phi R}$ is completely determined by the morphism $g:P\to R$ defined by $g = f(1\ot \iota)$.
Indeed, $f = \mu(g\ot 1)$. While the right $R$-module property of such $f$ is automatic, the left $R$-module property amounts to the condition
\beq\lb{efpf}\mu(\phi\ot g) = \mu(g\ot 1)c_{R,P}\ . \eeq
\end{remark}

\begin{remark}
\label{rhh}
An Azumaya algebra $R\in\B$ gives rise to a homomorphism $\varsigma_R: \Aut_\ot(Id_\B)\to \Inv(\B)$, which is the composition of the homomorphisms 
\eqref{gta} and \eqref{ati}.

Note that $\varsigma_{R\ot S}(a) = \varsigma_R(a)\ot \varsigma_S(a)$, so that we have a homomorphism
\begin{equation}
\label{varsigma}
\varsigma: \Pic(\B)\ \to\ \Hom_{gr}(\Aut_\ot(\id_\B),\,\Inv(\B)),
\end{equation}
or, equivalently, a (bimultiplicative) pairing
\begin{equation}
\label{pairing <>} 
\langle - \, , \,- \rangle : \Pic(\B)\times \Aut_\ot(\id_\B) \to \Inv(\B).
\end{equation}
This can be interpreted in terms of module categories as follows. 
For any $\M\in  \uuPic(\B)$ there is an isomorphism $ \Aut_\B(\M) \cong \Inv(\B)$
given by $\alpha_\M^+$.
For  $\nu \in  \Aut_\ot(\id_\B)$ the value of $\langle \M ,\, \nu \rangle$ is  the image of $\nu$ under the composition 
\begin{equation}
\label{<> explained}
\Aut_\ot(\id_\B) \xrightarrow{\iota} \Inv(\Z(\B)) \cong \Inv(\Z(\End_\B(\M)))  \to \Inv(\End_\B(\M)) \cong  \Inv(\B).
\end{equation}
Note that the object $\langle \M ,\, \nu \rangle$ coincides with the central structure of the braided module $\B$-category $\B^\nu$
(see Example~\ref{uuPicbrI}), i.e. with the value of the half-braiding $\M \bt_\B \B^\nu \to  \B^\nu  \bt_\B \M$ viewed as an object of $\Aut_\B(\M) \cong \Inv(\B)$.
\end{remark}


\section{The braided $2$-categorical Picard group of a symmetric fusion category}
\label{symmetric section}

Let $G$ be a finite group and let $\Rep(G)$ denote the 
category of representations of $G$.
It was proved by Deligne  \cite{De} that a symmetric fusion category  is equivalent to the following ``super" generalization of $\Rep(G)$.
Namely, let $G$ be a finite group and let $t\in G$ be a central element such that $t^2=1$. Then $\Rep(G)$
has a braiding defined by
\begin{equation}
\label{braiding on RepGt}
c_{V,W}: V\ot W \xrightarrow{\sim} W \ot V : v\ot w \mapsto 
\begin{cases}
-w\ot v & \text{if  $tv=-v$, $tw=-w$}, \\
\quad w\ot v & \text{otherwise.}
\end{cases}
\end{equation}
The fusion category $\Rep(G)$ equipped with the above braiding will be denoted $\Rep(G,\,t)$. Any symmetric
fusion category is equivalent to $\Rep(G,\,t)$ for a unique up to an isomorphism pair $(G,\,t)$. 
Under this notation $\Rep(G,\,1)$ is nothing but $\Rep(G)$ with its usual transposition braiding. 
We call $\Rep(G,\,t)$  {\em Tannakian} if $t=1$ and {\em super-Tannakian} if $t\neq 1$. 

The Picard group of $\Pic(\Rep(G,\,t))$ was computed by Carnovale in \cite{Car}.  We recall this description in  
Sections~\ref{pstc}  and  describe
the symmetric categorical group $\uPic(\Rep(G,\,t))$, i.e. the homomorphism
\[
Q_{\Pic(\Rep(G,\,t))}: \Pic(\Rep(G,\,t)) \to \widehat{G}.
\]
In Sections~\ref{symcase} and \ref{suptancase} we describe the braided categorical Picard group $\uPicbr(\Rep(G,\,t))$.

\subsection{The $2$-categorical Picard group of a Tannakian category}
\label{Pic Tannakian}

Let $\B=\Rep(G)$ be the category of finite dimensional representations of a finite group $G$ with its standard symmetric braiding. 
For a 2-cocycle $\gamma\in Z^2(G,k^\times)$ denote by $\Rep_\gamma(G)$ the category of $\gamma$-projective representations of $G$.

The first statement of the following Proposition is well known (see e.g. \cite{Car}). 
\begin{proposition}
\lb{tanc}
The assignment 
\[
H^2(G,k^\times)\to Pic(\Rep(G)) \ :\quad \gamma\mapsto \Rep_\gamma(G)
\] 
is an isomorphism. The homomorphism $Q_{Pic(\Rep(G))} : H^2(G,k^\times)\to (\widehat{G})_2$  is trivial.
\end{proposition}
\begin{proof}
Since an Azumaya algebra in $\Rep(G)$ is also an Azumaya algebra in $\Vect$ 
it should have the form $End_k(V)$ for some vector space $V$. The $G$-action on $End_k(V)$ corresponds to the structure of a projective $G$-representation on $V$. Its Schur multiplier (as a class in $H^2(G,\,k^\times)$) is the only Morita invariant of the $G$-algebra $End_k(V)$ .

We describe $Q_{Pic(\Rep(G))}$ as follows.
The transposition automorphism $c_{End_k(V),End_k(V)}$ of the tensor square $End_k(V)^{\ot 2}$ of the 
Azumaya algebra $End_k(V)$ is inner, i.e. is given by conjugation with an invertible element $\zeta\in End_k(V)^{\ot 2}$. 
The value $Q(End_k(V))$ is the character $\chi\in \widehat G$ defined by $g(\zeta) = \chi(g)\zeta$ for all $g\in G$.
Under the isomorphism $End_k(V)^{\ot 2}\cong End_k(V^{\ot 2})$ the element $\zeta$ corresponds to $c_{V,V}\in End_k(V^{\ot 2})$. 
Clearly, $c_{V,V}$ is $G$-invariant, which makes the character $\chi$ trivial. 
\end{proof}

\bpr\lb{prg}
The pairing \eqref{pairing <>} for $\Rep(G)$ is given by
\begin{equation}
\label{gammat}
H^2(G,k^\times)\times Z(G)\to \widehat G\ : \quad (\gamma,z)\mapsto \gamma_z,\qquad \gamma_z(g) = \frac{\gamma(g,z)}{\gamma(z,g)}.
\end{equation}
\epr
\bpf
We need to compute the invertible object $\varsigma_R(z)\in \Rep(G)$ for $R=End_k(V)$, where $V$ is a projective $G$-representation $\rho:G\to GL(V)$ with multiplier $\gamma$. 
According to the remark \ref{efp} the invertible object corresponding to an automorphism $\phi$ of $End_k(V)$ is given by the (unique) character $\chi$ such that there is an invertible $\zeta\in End_k(V)$ (the image of $g:P\to R$ on a basic vector of $P$) with the properties
$\phi(x) = \zeta x\zeta^{-1}$ and $\chi(g)\zeta\rho(g) = \rho(g)\zeta$ for any $x\in End_k(V)$ and $g\in G$.
Taking $\zeta = \rho(z)$ we get that $\chi$ is of the form $\gamma_z$. 
\epf

\subsection{The $2$-categorical Picard group of a super-Tannakian category}
\label{pstc}

We start with the basic example. 
\bex
\lb{svs}
Let $G=\bZ/2\bZ$ and let $t$ be the nontrivial element of $G$. Then $\Rep(G,\,t)=\sVect$, the category of super vector spaces. It goes back to \cite{W} that $\Pic(\sVect) = \bZ/2\bZ$. Let $\Pi$ denote the non-identity simple object of $\sVect$.
\nl
Let $R$ be an Azumaya algebra in $\sVect$ and let $\phi:R\to R$ be an automorphism. The equation \eqref{efpf} can be rewritten as $\phi(r)\zeta = \zeta r (-1)^{|r||\zeta|}$, where $\zeta$ is the value of $\phi$ on a basic element of the invertible object $P\in \sVect$.
\nl
The homomorphism $Q_{\Pic(\sVect)} : \Pic(\sVect)=\bZ/2\bZ\ \to \Inv(\sVect)= \bZ/2\bZ$ is the identity map.
Indeed, the Azymaya algebra $A=k\langle x| x^2=1\rangle = I\op\Pi$ (with $x$ odd) represents the non-trivial class in $\Pic({\sVect})$.
Its tensor square in $\sVect$ is $A^{\ot 2}=k\langle x, y| x^2=y^2=1\ xy+yx=0\rangle$. 
The braiding $c_{A,A}$ is the algebra automorphism $\tau$ interchanging $x$ and $y$.
Note that  $\tau(r)(x-y) = (x-y)r(-1)^{|r|}$ for $r\in A^{\ot 2}$.
Since the element $\zeta=x-y\in A^{\ot 2}$ is odd, the invertible object in $\sVect$ corresponding to $\tau$ is the non-trivial element $\Pi$ of $\Inv({\sVect})$.
\nl
The pairing \eqref{pairing <>} for $\sVect$ is given by
$$\Pic(\sVect)\times \Aut_\ot(\id_{\sVect})\to \Inv({\sVect})\qquad \langle A,\pi\rangle = \Pi\ ,$$
where $\pi$ is the natural automorphism of the identity functor of $\sVect$ such that $\pi_\be=\id_\be$ and $\pi_\Pi=-\id_\Pi$.
Indeed the automorphism $\pi_A$ of the Azymaya algebra $A$ satisfies $\pi(a)x = x a(-1)^{|a|}$ for $a\in A$.
\eex
We say that $\Rep(G,\,t)$ is {\em split super-Tannakian} if $\langle t \rangle$ is a direct summand of $G$ and   
{\em non-split super-Tannakian} otherwise.

The following definition was given and Theorem~\ref{pisu} below was proved by Carnovale \cite{Car}. We include the argument
for the sake of completeness and to set up notation for subsequent computations.  

Define the group  $H^2(G,t,k^\times)$ to be  the second cohomology $H^2(G,\,k^\times)$ as a set, 
with the group operation (on the level of cocycles) given by
\beq
\lb{tmc}
(\gamma*\nu)(f,g) = (-1)^{\xi_\gamma(f)\xi_\nu(g)}\gamma(f,g)\nu(f,g)\qquad f,g\in G,\, \gamma,\nu\in H^2(G,\,k^\times).
\eeq
where $\xi_\gamma:G\to \bZ/2\bZ$ is the homomorphism defined by 
\begin{equation}
\label{-1xi}
(-1)^{\xi_\gamma(g)} = \gamma_t(g) = \frac{\gamma(t,g)}{\gamma(g,t)}.
\end{equation} 

\begin{remark}
It was explained in \cite{Car}  that $H^2(G,t,k^\times)$ is (non-canonically) isomorphic to $H^2(G,k^\times)$.
\end{remark}

\begin{theorem}
\label{pisu}
The Picard group of a split super-Tannakian symmetric fusion category is 
\begin{equation}
Pic(\Rep(G,t))\ \cong\  H^2(G,t,k^\times) \times \bZ/2\bZ.
\end{equation}
The Picard group of a non-split super-Tannakian symmetric fusion category is 
\begin{equation}
Pic(\Rep(G,t))\ \cong\ H^2(G,t,k^\times).
\end{equation}
The homomorphism $Q_{Pic(\Rep(G,t))}:Pic(\Rep(G,t))\to Inv(\Rep(G,t))_2=(\widehat G)_2$ restricted to $H^2(G,t,k^\times)$ is given by
\begin{equation}
\label{chi-gamma}
\gamma\ \mapsto\ \gamma_t,\qquad \gamma_t(g) = \frac{\gamma(t,g)}{\gamma(g,t)}.
\end{equation}
In the split case, the homomorphism $Q_{Pic(\Rep(G,t))}$ restricted to $\bZ/2\bZ$ is the isomorphism $\bZ/2\bZ\to\widehat{\langle t\rangle}$.  
\end{theorem}
\bpf
Consider the homomorphism $\Pic(\Rep(G,\,t))\to Pic(\sVect)$ induced by the restriction functor $\Rep(G,t)\to \sVect$. 
We start by showing that this homomorphism is surjective if and only if  $\Rep(G,\,t)$ is split.
Indeed, a splitting of the restriction functor $\Rep(G,t)\to \sVect$ induces a splitting of the homomorphism $\Pic(\Rep(G,\,t))\to Pic(\sVect)$.
Conversely, an Azumaya algebra $R\in \Rep(G,t)$, which class is mapped to the class of $I\op\Pi\in\sVect$, 
has the form $(I\op\Pi)\ot End_k(U)$ for some vector space $U$. In particular its classical center (computed in $\Vect$) coincides with $I\op\Pi$. 
The $G$-action descends from $R$ to its center and gives a splitting $G\to \Aut_{alg}(I\op\Pi) = \bZ/2\bZ$.

The restriction of the homomorphism $Q_{Pic(\Rep(G,t))}$ to $Pic(\sVect)$ is described in Example~\ref{svs}.

In the following we argue that the kernel of the homomorphism $Pic(\Rep(G,t))\to Pic(\sVect)$ is isomorphic to $H^2(G,t,k^\times)$.
This kernel consists of classes of Azumaya algebras of the form $End_k(V)$ for a projective $G$-representation $V$. It is straightforward to see 
(e.g., by computing the left center \cite{DKR} in $\Rep(G,t)$) that $End_k(V)$ is an Azumaya algebra in $\Rep(G,t)$ for any projective $G$-representation $V$. Thus we have a set-theoretic bijection $H^2(G,t,k^\times)\to \Ker(Pic(\Rep(G,t))\to Pic(\sVect))$ sending $\gamma\in Z^2(G,k^\times)$ to the (class of) $End_k(V)$, where $V$ is a projective $G$-representation with the Schur multiplier $\gamma$.
To show that this is a group isomorphism we need a few facts.

Let $U$ and $V$ be super vector spaces.
Define a map 
\begin{equation}
\label{map phi}
\phi: End_k(U)\ot End_k(V)\to End_k(U\ot V) : a\ot b \mapsto \phi_{a,b},
\end{equation}
by $\phi_{a,b}(u\ot v)=(-1)^{|b||u|} a(u)\ot b(v)$
for all homogeneous maps $a\in End_k(U),\, b\in End_k(V)$ and homogeneous vectors $u\in U,\, v\in V$,
where $|a|$ denotes the degree of $a$. 

The map \eqref{map phi} is an isomorphism of algebras in $\sVect$, i.e. $\phi_{a,b}\circ\phi_{c,d} = (-1)^{|b||c|}\phi_{a\circ c,b\circ d}$. 
Indeed, 
\begin{align*}
\phi_{a,b}(\phi_{c,d}(u\ot v)) &= (-1)^{|d||u|}\phi_{a,b}(c(u)\ot d(v)) \\
& = (-1)^{|d||u|+|b||c(u)|}a(c(u))\ot b(d(v)) \\
& = (-1)^{|d||u|+|b||c|+||b|u|}(a\circ c)(u)\ot (b\circ d)(v) \\
& = (-1)^{|b||c|}\phi_{a\circ c,b\circ d}(u\ot v).
\end{align*}

Let now $U$ be a projective $G$-representation with the Schur multiplier $\gamma\in Z^2(G,k^\times)$.
Let $t\in G$ be a central involution. Denote by $U = U_0\op U_1$ the $\bZ/2\bZ$-grading corresponding to $t$, i.e. $u\in U$ is homogeneous 
of degree $|u|$ iff $t(u) = (-1)^{|u|}u$.
Then the $G$-action is related to the grading in the following way $|g.u| = \xi_\gamma(g)+|u|$, where $\xi_\gamma$ is defined in \eqref{-1xi}.
Indeed,
$$
t.(g.u) = \gamma(t,g)(tg).u = \gamma(t,g)(gt).u = \frac{\gamma(t,g)}{\gamma(g,t)}g.(t.u) = (-1)^{|u|}\frac{\gamma(t,g)}{\gamma(g,t)}g.u.
$$ 
Now let $U$ and $V$ be projective $G$-representations with the Schur multipliers $\gamma_U, \gamma_V\in Z^2(G,k^\times)$ correspondingly. 
Note that $End_k(U)$ is a $G$-algebra with $g(a)$ defined by $g(a)(u) = g.(a(g^{-1}.u))$ (and similarly for $End_k(V)$).
Then the homomorphism $\phi:End_k(U)\ot End_k(V)\to End_k(U\ot V)$ has the following $G$-equivariance property: 
$$\phi_{g(a),g(b)} = \rho(g)\circ \phi_{a,b}\circ\rho(g)^{-1}\ ,$$
where $\rho(g):U\ot V\to U\ot V$ is given by $\rho(g)(u\ot v) = (-1)^{\xi_\gamma(g)|v|}g.u\ot g.v$. 
Indeed, the relation $\phi_{g(a),g(b)}\circ\rho(g) = \rho(g)\circ \phi_{a,b}$ can be checked directly:
\begin{align*}
\phi_{g(a),g(b)}(\rho(g)(u\ot v)) &= (-1)^{\xi_\gamma(g)|v|}\phi_{g(a),g(b)}(g.u\ot g.v) \\
& = (-1)^{\xi_\gamma(g)|v|+|g(b)||g.u|}g(a)(g.u)\ot g(b)(g.v) \\
& = (-1)^{|b||u|+\xi_\gamma(g)|b(v)|}g.(a(u))\ot g.(b(v)) \\
& = (-1)^{|b||u|+\xi_\gamma(g)|b(v)|}g.(a(u))\ot g.(b(v)) \\
& = (-1)^{|b||u|}\rho(g)(a(u)\ot b(v)) \\
&= \rho(g)(\phi_{a,b}(u\ot v)). 
\end{align*}
The map $\rho:G\to GL(U\ot V)$ is a projective representation with the  Schur multiplier 
\begin{equation}
\label{schur}
\gamma(f,g) = (-1)^{\xi_\gamma(f)\xi_V(g)}\gamma_U(f,g)\gamma_V(f,g),\qquad f,g\in G.
\end{equation}
To see this,  we compute
\begin{align*}
\rho(fg)(u\ot v)  &= (-1)^{\xi_\gamma(fg)|v|}(fg).u\ot (fg).v \\
&= (-1)^{\xi_\gamma(fg)|v|}\gamma_U(f,g)\gamma_V(f,g)f.(g.u)\ot f.(g.v) \\
&= (-1)^{\xi_\gamma(g)|v|+\xi_\gamma(f)|g.v|}\gamma(f,g)f.(g.u)\ot f.(g.v) \\
&= (-1)^{\xi_\gamma(g)|v|}\gamma(f,g)\rho(f)(g.u\ot g.v) \\
&= \gamma(f,g)\rho(f)(\rho(g)(u\ot v)).
\end{align*}
In the rest of the proof we describe the restriction of the homomorphism $Q_{\Rep(G,t)}$ to $H^2(G,t,k^\times)$.
Let again $V$ be a projective $G$-representation with the Schur multiplier $\gamma\in Z^2(G,k^\times)$.
The automorphism $c_{End_k(V),End_k(V)}$ of the algebra $End_k(V)^{\ot 2}$ in $\Rep(G,t)$
$$
c_{End_k(V),End_k(V)}(a\ot b) = (-1)^{|a||b|}b\ot a,
$$ 
transported (along $\phi$) to an automorphism of the algebra $End_k(V^{\ot 2})$, is inner.
More precisely, we have
$$
c_{V,V}\circ \phi_{a,b}\circ c_{V,V}^{-1} = (-1)^{|a||b|}\phi_{b,a},
$$
since
\begin{align*}
c_{V,V}(\phi_{a,b}(u\ot v)) &= (-1)^{|b||u|}c_{V,V}(a(u)\ot b(v)) \\
&= (-1)^{|b||u|+|a(u)||b(v)|}b(u)\ot a(v) \\
& =  (-1)^{|a||b|+|u||v|+|a||v|}b(u)\ot a(v)\\
& = (-1)^{|a||b|+|u||v|}\phi_{b,a}(v\ot u) \\
& = (-1)^{|a||b|}\phi_{b,a}(c_{V,V}(u\ot v)). 
\end{align*}

The element $c_{V,V}\in End_k(V)^{\ot 2}$ has the following $G$-equivariance property:
$g(c_{V,V}) = \chi_\gamma(g)c_{V,V}$, where $\chi_\gamma$ is defined in \eqref{chi-gamma}. 
That is, $\rho(g)\circ c_{V,V}\circ \rho(g)^{-1} = \chi(g)c_{V,V}$ since we have
\begin{align*}
\rho(g)(c_{V,V}(u\ot v)) &= (-1)^{|u||v|}\rho(g)(v\ot u) \\
&= (-1)^{|u||v|+\xi(g)|u|}g.v\ot g.u \\
&= (-1)^{\xi(g)^2+\xi(g)|v|+|g.u||g.v|}g.v\ot g.u\\
&= (-1)^{\xi(g)+\xi(g)|v|}c_{V,V}(g.u\ot g.v) \\
&= \chi_\gamma(g)c_{V,V}(\rho(g)(u\ot v)).
\end{align*}
The formula for $Q_{Pic(\Rep(G,t))}$ on $H^2(G,t,k^\times)$ now follows from Remark \ref{efp}. 
\epf

Recall that  the character $\xi_\gamma :G \to \mathbb{Z}/2\mathbb{Z} =\{0,1\}$ was defined in \eqref{gammat}.
It depends on  $\gamma \in H^2(G,\,k^\times)$ as well as on $t\in Z(G)$.

\begin{proposition}
\label{prsg}
The pairing \eqref{pairing <>} for $\Rep(G,t)$ is given by
\begin{align}
\langle\gamma,\,z \rangle &= \gamma_{zt^{\xi_\gamma(z)}} \qquad \text{in the non-split case}, \lb{ff} \\
\label{split formula}
\langle(\gamma,\,\ve),\,z\rangle &= \gamma_{zt^{\xi_\gamma(z)}} \nu^{\xi_\gamma(z)\ve}  \qquad \text{in the split case},
\end{align}
for $\gamma \in H^2(G,\,k^\times)$ and  $z\in Z(G)$,  where $\nu$ is the composition $G\to \langle t\rangle\to k^\times$ of the non-trivial character on 
$\langle t\rangle$ with a  (chosen) splitting $G\to \langle t\rangle$ (i.e., $\nu\in \widehat G$ corresponds to the image of $\Pi$ under a (chosen) splitting $\sVect\to \Rep(G,t)$). 
\end{proposition}
\bpf
First assume that the class of $R$ is in the kernel of the homomorphism $Pic(\Rep(G,t))\to Pic(\sVect)$, i.e. $R=End_k(V)$, where $V$ is a projective $G$-representation $\rho:G\to GL(V)$ with multiplier $\gamma$. 
According to Remark \ref{efp}, the invertible object corresponding to the automorphism $\phi=\rho(z)(-)\rho(z)^{-1}$ of $End_k(V)$ is given by the (unique) character $\chi$ such that there is an invertible $\zeta\in End_k(V)$ with the properties
$\phi(r) = \zeta r\zeta^{-1}(-1)^{|r||\zeta|}$ and $\zeta\rho(g) = \chi(g)\rho(g)\zeta$ for any $g\in G$.
A (unique up to a scalar) solution is $\zeta = \rho(zt^{\xi_\gamma(z)})$, where $\xi_\gamma(z) = |\rho(z)|$ (or, equivalently, $(-1)^{\xi_\gamma(z)} = \gamma_t(z)$). 
Thus, $\chi=\gamma_{zt^{\xi_\gamma(z)}}$. This computes the pairing between $H^2(G,\,k^\times)$ and $Z(G)$ and, in particular, proves the formula \eqref{ff}. 

In the split case the restriction of the pairing to $\sVect\subset \Rep(G,t)$ is $\langle(1,\,\ve),\,z\rangle = \nu^{\xi_\gamma(z)\ve}$ (according to Example \ref{svs}). 
\epf

\begin{remark}
It should be noted that {\em both} sides of  \eqref{split formula} depend on the choice of splitting $G\to \langle t\rangle$. 
Indeed, the identification $\Pic(\Rep(G,t))\cong H^2(G,\,k^\times) \times \mathbb{Z}/2\mathbb{Z}$ (implicit on the left hand side
of \eqref{split formula}) is not canonical and depends on this choice.
\end{remark}

\subsection{The braided categorical Picard group  of a Tannakian category}
\label{symcase}

\bth
\label{pbsy}
Let $\E$ be a symmetric tensor category. There is a group isomorphism  
\begin{equation}
\label{Picbr E as a group}
\Picbr(\E) \cong \Pic(\E)\times \Aut_\ot(\id_\E)\ .
\end{equation}
The first canonical class of $\Picbr(\E)$ is 
\begin{equation}
\label{QE}
Q_{\Picbr(\E)}(\M,\, \nu) = Q_{\Pic(\E)}(\M)\ \langle \M,\, \nu \rangle.
\end{equation}
where the pairing $\langle \M,\, \nu \rangle$ is defined in \eqref{pairing <>}.
\eth
\begin{proof}
Since every $\E$-module category admits an identity $\E$-module braiding
it follows that $\uPicbr(\E)$ is generated by  symmetric categorical subgroups $\uPic(\E)$ and $\uAut_\ot(\id_\E)$
(the latter consists of invertible categories in $\ModbrI(\E)$).
These subgroups intersect trivially, so $\Picbr(\E)$ is their direct product. 
\nl
For $\M\in \Pic(\E)$ and $\nu\in \Aut_\ot(\id_\E)$ the self-braiding $C_{(\M,\, \nu),(\M,\, \nu)}$ on $(\M,\, \nu) = \M\bt_\B\B^\nu$ is the composition of (conjugates of) $C_{\M,\M}, C_{\M,\B^\nu}, C_{\B^\nu,\M}$ and $\C_{\B^\nu,\B^\nu}$. The self-braiding $C_{\M,\M}$ is the symmetric one of $\uPicbr(\E)$. The relation between the self-braiding $C_{\M,\B^\nu}$ and $\langle \M,\, \nu \rangle$ is explained in  Example~\ref{Bnu}. The braidings $C_{\B^\nu,\M}$ and $\C_{\B^\nu,\B^\nu}$ are in effect trivial (the first since $C_{\B^\nu,\M}$ is the symmetric braiding $C_{\B,\M}$ $\uPicbr(\E)$, while the second is the trivial case $\M=\B$ of Example~\ref{Bnu}).
This proves the formula for $Q_{\Picbr(\E)}$.  
\end{proof}

The Whitehead bracket $\pi_0\times \pi_1\to \pi_2$ \eqref{new pairing} of $\uuPicbr(\E)$ is given by
\begin{equation}
\label{Whitehead E}
( \Pic(\E)\times \Aut_\ot(\id_\E)) \times \Inv(\E) \to k^\times : [(\M,\,\nu),\, Z] =  \nu_Z.
\end{equation}

We can now describe the braided categorical Picard group of a Tannakian category.

\begin{corollary}
\label{Picbr Tannakian and non-split}
We have
\begin{align}
\pi_0(\uuPic_{br}(\Rep(G))) &\cong H^2(G,\, k^\times) \times  Z(G), \\
\pi_1(\uuPic_{br}(\Rep(G))) & \cong\widehat{G},  
\end{align}
where $\widehat{G}$ denotes  the group of characters of $G$.
\nl
The first canonical class is the quadratic form
\begin{equation}
Q_{\uuPic_{br}(\Rep(G))} (\gamma,\, z) = \gamma_z(-), \qquad  z\in Z(G),\, \gamma \in H^2(G,\, k^\times).
\end{equation}
The second canonical class  is trivial.
\nl
The Whitehead bracket \eqref{Whitehead E}  is given by
\begin{equation}
\label{Whitehead RepG}
[(\gamma,z),\, \chi]= \chi(z),\qquad \chi\in \widehat{G}, z\in Z(G).
\end{equation}
\end{corollary}
\bpf
Follows from Propositions~\ref{tanc}, \ref{prg} and  Theorem  \ref{pbsy}.
\epf

\subsection{The braided categorical Picard group  of a super-Tannakian category}
\label{suptancase}

Here we deal with the super-Tannakian case. We start with the basic example
of the symmetric fusion category $\sVect$ of super vector spaces. As before, $\Pi$ denotes the non-identity simple
object of $\sVect$.

\begin{example}
The group $\pi_0(\uuPicbr({\sVect}))\cong Pic_{br}({\sVect})\times \Aut_\ot (\id_{\sVect}) \cong \bZ/2\bZ \times \bZ/2\bZ $ consists of pairs
$(\I,id),\ (\I,\pi),\ (\R,id),\ (\R,\pi)$, where $\I$ is the regular $\sVect$-module category, $\R=\Vect$ viewed as an $\sVect$-module
category (i.e. $\R=\sVect_A$, where $A$ is the algebra from Example \ref{svs}),  
$\pi$ is the natural automorphism of the identity functor of $\sVect$ such that $\pi_\be=\id_\be$ and $\pi_\Pi=-\id_\Pi$ (as in the algebra from Example \ref{svs}).
It follows from Example \ref{svs} and  Theorem \ref{pbsy} that the quadratic function $Q_{Pic_{br}({\sVect})}:\Pic_{br}(\sVect)\to Inv(\sVect)$ is given by
\[
Q_{Pic_{br}({\sVect})}(\I,id) = Q_{Pic_{br}({\sVect})}(\I,\pi) = Q_{Pic_{br}({\sVect})}(\R,\pi) = \I,\qquad Q_{Pic_{br}({\sVect})}(\R,id) = \Pi.
\]
\end{example}

\begin{corollary}
\begin{align}
\pi_0(\uuPic_{br}(\Rep(G,t))) & \cong
\begin{cases}
H^2(G,t,k^\times)\times Z(G) & \text{in the non-split case}, \\
H^2(G,t,k^\times)\times \bZ/2\bZ\times Z(G) & \text{in the split case}.
\end{cases} \\
\pi_1(\uuPic_{br}(\Rep(G,t))) & \cong  \widehat{G}.
\end{align}
The first canonical class of $\uuPic_{br}(\Rep(G,t))$ is given by the quadratic form (with values in $\widehat G$)
\begin{align}
Q_{\Pic_{br}(\Rep(G,t))}(\gamma,\,z) &= \gamma_{zt^{\xi_\gamma(z)+1}} \qquad \text{in the non-split case}, \\
Q_{\Pic_{br}(\Rep(G,t))}(\gamma,\,\ve,\,z) &= \gamma_{zt^{\xi_\gamma(z)+1}} \nu^{\xi_\gamma(z)\ve}  \qquad \text{in the split case},
\end{align}
for $\gamma \in H^2(G,\,k^\times)$, $z\in Z(G)$, and $x\in G$, where $\nu\in \widehat G$ is as in Proposition \ref{prsg}. 
\nl
The second canonical class is the homomorphism
\begin{equation}
\widehat{G} \to \{ \pm 1\} : \chi \mapsto \chi(t).
\end{equation}
The Whitehead bracket is given by \eqref{Whitehead RepG} (and does not depend on $t$).
\end{corollary}
\bpf
Follows from Theorems \ref{pisu} and \ref{pbsy} and Proposition \ref{prsg}.
\epf

\section{The braided Picard group of a pointed braided fusion category}
\label{pointed section}
 
Recall \cite{JS, DGNO} that a pointed braided fusion category $\B$ is determined by a quadratic form $q:A \to k^\times$, where
$A$ is the finite Abelian group of isomorphisms classes of simple objects of $\B$ and $q(x) = c_{x,x}, \, x\in A$, 
where $c$ denotes the braiding of $\B$.

\begin{proposition}
\label{Picbr ptd} 
Let $\B$ be a pointed braided fusion category. There is an equivalence
\begin{equation}
\label{syc}
\uPic_{br}(\B) \cong \uPic_{br}(\Z_{sym}(\B))
\end{equation}
 of braided categorical groups.
\end{proposition}
\begin{proof}
The group $\Pic_{br}(\B)$ can be computed using the exact sequence \eqref{exact sequence}.  Namely, we have
a short exact sequence
\begin{equation}
\label{ses Coker}
0\to \mbox{Coker} (\Inv(\B) \xrightarrow{\alpha}  \Aut_\ot(\id_\B)) \to \Pic_{br}(\B) \to \mbox{Ker}(\Pic(\B)\xrightarrow{\partial}   \Aut^{br}(\B)) \to 0.
\end{equation}
By \cite[Proposition 5.17]{DN}, the induction 
\[
 \mbox{Ind}: \Pic(\Z_{sym}(\B)) \to \Pic(\B) : \M \mapsto \B \bt_{\Z_{sym}(\B)} \M
\]
establishes a group isomorphism $\Pic(\Z_{sym}(\B)) \cong \mbox{Ker}({\partial})$. Note that \eqref{inducing to center} provides a splitting
for \eqref{ses Coker}.  The homomorphism $\alpha$ is the map $A \mapsto \widehat{A}$ coming from the bilinear form on $A$ associated to $q$.
Its cokernel is  $\widehat{A^\perp}\cong \Aut_\ot(\id_{\Z_{sym}(\B)})$. 

Thus, the split exact sequence \eqref{ses Coker} yields a group isomorphism \eqref{syc} by Proposition~\ref{pbsy}:
\[
\Pic_{br}(\Z_{sym}(\B))  \cong \Pic(\Z_{sym}(\B))  \times \Aut_\ot(\id_{\Z_{sym}(\B)}) \to \Pic_{br}(\B) : (\M,\,\nu) \mapsto \mbox{Ind}(\M)^\nu,
\]
cf.\ Example~\ref{uuPicbrI}. The value of the quadratic form $q_{\uuPicbr(\B)}$ on $\mbox{Ind}(\M)^x$ is given by the half braiding
\[
 \mbox{Ind}(\M) \bt_\B \B^\nu \to \B^\nu \bt_\B  \mbox{Ind}(\M).
\]
The latter coincides with the value of the pairing $\langle \M,\, x|_{\Z_{sym}(\B)} \rangle$  \eqref{pairing <>}, see Remark~\ref{efp}.
So the result follows from Proposition~\ref{pbsy}.
\end{proof}

Note that $\tau:=q|_{A^\perp}$ is an element of $\widehat{A^\perp}$ of order at most $2$.

\begin{corollary}
\begin{equation}
\Pic_{br}(\C(A,\,q)) \cong
\begin{cases}
\Hom(\Extpower^2 \widehat{A},\,k^\times) \times \widehat{A^\perp} & \text{if  $\tau=1$}, \\
\Hom(\Extpower^2 \widehat{A},\,k^\times) \times \widehat{A^\perp} & \text{if $\tau\neq 1$ and $ A= \mbox{Ker}(\tau)\times \bZ/2\bZ$}, \\
\Hom(\Extpower^2 \widehat{A},\,k^\times)  \times \bZ/2\bZ\times \widehat{A^\perp} & \text{if $\tau\neq 1$ and $ A\neq  \mbox{Ker}(\tau)\times \bZ/2\bZ$}.
\end{cases}
\end{equation}
\end{corollary}
\begin{proof}
This follows from Proposition~\ref{Picbr ptd} and the description of the Picard group of a symmetric fusion category, see Sections~\ref{Pic Tannakian}
and~\ref{pstc}.
\end{proof}

\section{Classification of graded extensions}
\label{ext chapter}

\subsection{Graded tensor extensions \cite{ENO}}
\label{sect4.1}

Let $\D$ be a tensor category and let  $G$ be a finite group.   

\begin{definition}
A  {\em tensor $G$-graded extension} (or, simply, a {\em $G$-extension}) of a tensor category $\D$ is a tensor category
\begin{equation}
\label{G-extension}
\C =\bigoplus_{x\in G}\, \C_x,\qquad \C_e =\D,
\end{equation}
such that $\C_x\neq 0$ and the tensor product of $\C$ maps $\C_x\times \C_y$ to $\C_{xy}$  for all $x,y\in G$. 
\end{definition}

\begin{definition}
An {\em equivalence} between two $G$-graded extensions, $\C =\bigoplus_{x\in G}\, \C_x$ and $\tilde{\C} =\bigoplus_{x\in G}\, \tilde{\C}_x$
of $\D$ is a tensor equivalence $F: \C \xrightarrow{\sim} \tilde{\C}$ such that $F|_\D =\id_\D$
and $F(\C_x)= \tilde{\C}_x$ for all $x\in G$.  An {\em isomorphism} between equivalences of $G$-extensions
$F, F': \C \xrightarrow{\sim} \tilde{\C}$ is a tensor isomorphism $\eta: F \xrightarrow{\sim} F'$ whose restriction
on $F|_\D =\id_\D$ is the identity isomorphism. 
\end{definition}

Thus, $G$-extensions of $\D$ form a $2$-groupoid $\uuEx(G,\, \D)$  whose objects are extensions,
$1$-cells are equivalences of extensions, and $2$-cells are isomorphisms of equivalences. 

\begin{example}
\label{ext Vec}
$G$-extensions of $\Vect$ are precisely pointed fusion categories $\Vect_G^\omega$, where $\omega\in Z^3(G,\,k^\times)$
is a $3$-cocyle. 
Equivalences between extensions $\Vect_G^{\omega}$ and $\Vect_G^{\tilde{\omega}}$ correspond to  $2$-cochains 
$\mu \in C^2(G,\, k^\times)$ such that $d(\mu) = \tilde{\omega}/\omega$. Thus, $\pi_0(\Ex(G,\, \Vect)) = H^3(G,\,k^\times)$. 
\end{example}

\begin{remark}
Example~\ref{ext Vec} shows that  there exist equivalent tensor categories that are not equivalent as extensions.
Indeed, if the cohomology classes of $\omega$ and $\tilde{\omega}$ are in the same $\Aut(G)$-orbit
then $\Vect_G^\omega \cong \Vect_G^{\tilde{\omega}}$ as tensor categories. 
\end{remark}

The following theorem is essentially proved in \cite{ENO}. We include its proof for the reader's convenience.
Our arguments for central, braided, and symmetric extensions in subsequent sections will follow the same pattern.

\begin{theorem}
\label{main ENO}
There is an equivalence of $2$-groupoids $\uuEx(G,\, \D) \cong \FUN(G,\, \uuBrPic(\D))$.
\end{theorem}
\begin{proof}
We construct a  $2$-functor 
\begin{equation}
\label{M:}
M: \uuEx(G,\, \D) \to\FUN(G,\, \uuBrPic(\D))
\end{equation}
as follows. 
Given a $G$-extension $\C =\bigoplus_{x\in G}\, \C_x$ of $\D$, each homogeneous
component $\C_g$ is an invertible $\D$-bimodule category. The restrictions  $\otimes_{x,y}: \C_x \times \C_y \to \C_{xy},\, x,y\in G,$
of the tensor product of $\C$ are $\D$-balanced functors  and so give rise to  $\D$-bimodule equivalences 
\begin{equation}
\label{Mgh}
M_{x,y}: \C_x \bt_\D \C_y \xrightarrow{\sim} \C_{xy}.
\end{equation}
The associativity constraints of $\C$ restricted to $\C_x \times  \C_y \times \C_z$  can be viewed as 
natural isomorphisms of $\D$-balanced functors and so  give rise to natural isomorphisms
of $\D$-bimodule functors
\begin{equation}
\label{alphas}
\xymatrix{ 
\C_x \bt_\D \C_y \bt_\D \C_z \ar[rr]^{M_{y,z}} \ar[d]_{M_{x,y}}="a" && \C_{x}  \bt_\D \C_{yz} \ar[d]^{M_{x,yz}}="b" \\
\C_{xy} \bt_\D \C_z \ar[rr]_{M_{xy,z}} && \C_{xyz},
\ar@{}"a";"b"^(.25){}="a1"^(.75){}="b1" \ar@{=>}^{\alpha_{x,y,z}}"a1";"b1" 
}
\end{equation} 
for all $x,y,z\in G$, cf.\ \eqref{alphaLMN}. The pentagon identity for the associativity constraints of $\C$  implies that 
\eqref{mon2fun}  is satisfied (equivalently, the cubes \eqref{cubes fghk} commute
for all $x,y,z,w\in G$).
This means that the above data consisting of $\D$-bimodule categories $\C_x$, equivalences $M_{x,y}$,
and natural isomorphisms $\alpha_{x,y,z},\, x,y,z\in G,$ determine a monoidal $2$-functor $M(\C): G\to \uuBrPic(\D)$.  

Suppose that there is  another $G$-extension $\tilde{\C} =\bigoplus_{g\in G}\, \tilde{\C}_g$ of $\D$ and an equivalence 
of extensions $F: \C \to \tilde{\C}$.  It restricts to $\D$-bimodule equivalences 
\begin{equation}
\label{Fs}
F_x: \C_x \xrightarrow{\sim}  \tilde{\C}_x.
\end{equation}
The tensor structure of $F$ restricted to $\C_x \times \C_y$ gives rise to an invertible $2$-cell
\begin{equation}
\label{mus}
\xymatrix{ 
\C_x \bt_\D \C_y \ar[rr]^{F_x \bt_\D F_y} \ar[d]_{M_{x,y}}="a" && \tilde\C_{x} \bt_\D \tilde\C_y  \ar[d]^{\tilde{M}_{x,y}}="b" \\
 \C_{xy} \ar[rr]_{F_{xy}} && \tilde\C_{xy},
\ar@{}"a";"b"^(.25){}="a1"^(.75){}="b1" \ar@{=>}^{\mu_{x,y}}"a1";"b1" 
}
\end{equation}
$x,y\in G$,  and the coherence axiom for the tensor structure of $F$ implies that  
\eqref{MT cube} is satisfied (equivalently, that the cubes \eqref{cubes for PW1} commute for all $x,y,z\in G$), 
so that we have a pseudo-natural isomorphism $M(F): M(\C) \to M(\tilde\C)$.  

Given an isomorphism $\eta$ between a pair of equivalences $F,F'$ of extensions $\C$ and $\C'$
its components are natural isomorphisms of $\D$-bimodule functors:
\begin{equation}
\label{etas}
\xymatrix{
\C_x  \ar@/^1pc/[rrrr]^{F_x}_{}="a"     \ar@/^-1pc/[rrrr]_{F'_x}_{}="d"  &&&& \tilde{\C}_x.
\ar@2{->}^{\eta_x}"a";"d"
}
\end{equation}
The tensor property of $\eta$ implies that \eqref{MM cylinder} is satisfied, i.e. the cylinder
\begin{equation}
\label{cylinder eta}
\xymatrix{
\C_x \bt_\D \C_y \ar@/^1pc/[rrrr]^{F_xF_y}_{}="a"^>>>>>>>>>>>>>>>{}="x"     \ar@/^-1pc/[rrrr]_{F'_xF'_y}_{}="d"^<<<<<<<<<<<<<<{}="z"  
\ar[dd]_{M_{x,y}}
&&&& \tilde{\C}_x \bt_\D  \tilde{\C}_y \ar[dd]^{\tilde{M}_{x,y}} \\
\\
\C_{xy} \ar@{-->}@/^1pc/[rrrr]^{F_{gh}}_{}="b"^>>>>>>>>>>>>>>>>>{}="y"     \ar@/^-1pc/[rrrr]_{F'_{xy}}_{}="e"^<<<<<<<<<<<<<<<<{}="w"  
&&&& \tilde{\C}_{gh}.
\ar@{}"a";"d"^(.1){}="X"^(.9){}="Y"   \ar@2{->}^{\eta_g \eta_h}"X";"Y" 
\ar@{}"b";"e"^(.1){}="X"^(.9){}="Y"  \ar@2{->}^{\eta_{gh}}"X";"Y"
\ar@{}"x";"y"^(.15){}="X"^(.85){}="Y"  \ar@2{-->}_{\mu_{g,h}}"X";"Y"
\ar@{}"z";"w"^(.15){}="X"^(.85){}="Y"   \ar@2{->}^{\mu'_{g,h}}"X";"Y"
}
\end{equation}
commutes for all $x,y\in G$.
So we get  an invertible modification $M(\eta)$ between pseudo-natural isomorphisms $M(F)$ and $M(F')$. 
This completes the construction of a monoidal $2$-functor \eqref{M:}.

A $2$-functor
\begin{equation}
\label{L:}
L: \FUN(G,\, \uuBrPic(\D)) \to \uuEx(G,\, \D) 
\end{equation}
quasi-inverse to \eqref{M:}
can be constructed by reversing the above constructions.  Namely, let $C: G \to \uuBrPic(\D): x \mapsto \C_x$ be a monoidal $2$-functor.
Form a $\D$-bimodule category $L(C):= \bigoplus_{x\in G}\, \C_x$ with the tensor product given by composing
$\C_x\times \C_y \to \C_x\bt_\D \C_y$ with $1$-cells \eqref{Mgh} and the associativity constraints coming from
$2$-cells \eqref{alphas}.  
The commuting polytopes \eqref{cubes fghk} give the pentagon identity for the associativity constraint. 

To  check that $L(C)$ is rigid, note that by Corollary~\ref{inv Cbimod is exact} it is exact as a $\D$-module category. Hence,
the dual category $\End_\D(L(C))$ is a tensor category (i.e. is rigid). Given a homogeneous object $X$ in $\C_x \subset L(C),\, x\in G,$
define  a $\D$-module endofuctor $L(X) \in  \End_\D(L(C))$ by setting $L(X)=X \ot -$ on $\D$ and $L(X)=0$ on $\C_g,\,g\neq 0$.
Its adjoints are given by functors $X^* \ot -,\, {}^*X \ot - : \C_x\to \D $  for some objects $X^*, {}^*X\in \C_{x^{-1}}$. These objects are the duals
of $X$. Thus, $L(C)$ is a tensor category and so it is a $G$-extension of $\D$.

From the universal property of $\bt_\D$, a pseudo-natural isomorphism of functors $C,\, C': G \to \uuBrPic(\D)$ 
gives an equivalence of extensions with the tensor structure coming from \eqref{mus} and that a modification
of pseudo-natural isomorphisms gives  a natural isomorphism of equivalences of extensions.
\end{proof}

\begin{remark}
The proof of Theorem~\ref{main ENO} is based on the correspondences (coming from the universal property of $\bt_\D$) 
between the structure functors and morphisms 
of graded tensor categories and the axioms they satisfy  and the structure $1$- and $2$-cells of monoidal $2$-functors
and the commutative polytopes satisfied by them. We summarize these correspondences in Table~\ref{table-3} (cf.\ the table 
from \cite[Section 2.3]{BJ}).

\begin{table}[h!]
\centering
\begin{tabular}{ |p{7cm}|p{7cm}|  }
 \hline
 \hline
\centerline{ Tensor $G$-extensions $\C$ of $\D$} & \centerline{Monoidal $2$-functors $M: G \to \uuBrPic(\D)$}  \\
\hline
\hline
homogeneous components $\C_g$ &  $0$-cells $M(g):=\C_g$ \\
tensor products $\C_g \times \C_h\to \C_{gh}$ &  monoidal $1$-cells $M_{g,h}: \C_g\bt_\D \C_h \to \C_{gh}$ \\
associativity constraints  $a_{X,Y,Z}$  & associativity $2$-cells $\alpha_{f,g,h}$ \eqref{alphas} \\
commuting pentagon diagram for $a$ & commuting cubes \eqref{cubes fghk} for $\alpha$ \\
\hline
equivalence $F:\C \to \tilde\C$ of extensions &  $1$-cells $F_g: \C_g \to  \tilde\C_g$ \\
tensor structure of $F$  & monoidal $2$-cells $\mu_{g,h}$ \eqref{mus} \\
commuting tensor property  diagram for $F$  & commuting cubes \eqref{cubes for PW1} for $\mu$ \\
\hline
isomorphism $\eta: F \to F'$ of equivalences  & modification $2$-cells $\eta_g: F_g \to F'_g$ \eqref{etas} \\
commuting tensor property diagram for  $\eta$ & commuting cylinders \eqref{cylinder eta}  for $\eta$\\
\hline
\end{tabular}
\caption{\label{table-3} A  correspondence between tensor $G$-extensions  and monoidal $2$-functors.}
\end{table}
\end{remark}

We can describe $G$-graded extensions of $\D$ in terms of group cohomology. 
It follows from constructions of Section~\ref{Sect Functors and cohomology} that given a monoidal functor  $M:G \to \uBrPic(\D)$
there exist a canonical cohomology class $\mathcal{p}_M^0\in H^4(G,\,k^\times)$ and a canonical group homomorphism
\begin{equation*}
\mathcal{p}_M^1: H^1(G,\, \Inv(\Z(\D))) \to  H^3(G,\,k^\times)
\end{equation*} 
defined in \eqref{cubes fghk} and \eqref{cubes for PW1}, respectively.

\begin{corollary}
\label{summary monoidal ext}
A monoidal functor  $M:G \to \uBrPic(\D)$ gives rise to a $G$-graded extension of $\D$
if and only  if $ \mathcal{p}^0_M=0$ in $H^4(G,\,k^\times)$. 
Equivalence classes  of such extensions form a torsor over the cokernel of~$\mathcal{p}^1_M$.
\end{corollary}
\begin{proof}
This follows from Theorem~\ref{main ENO} and  Corollary~\ref{summary monoidal fun}.
\end{proof}

\begin{remark}
In \cite{ENO} the notion of an equivalence of graded extensions was not explicitly defined and extensions were
parameterized by a torsor over $H^3(G,\,k^\times)$. We would like to point that the map $\mathcal{p}_M^1$ is non-trivial
in general.  Here is a simple example. Let $\D =\Vect_{\mathbb{Z}/2\mathbb{Z}}^{\omega}$, where $\omega$ is the non-trivial
element of $H^3(\mathbb{Z}/2\mathbb{Z},\, k^\times)$. Then $\Inv(\Z(\D)) \cong \mathbb{Z}/2\mathbb{Z} \times \mathbb{Z}/2\mathbb{Z}$
with the asscoaitor $\omega \times \omega^{-1}$. Take the trivial monoidal functor $M:\mathbb{Z}/2\mathbb{Z} \to \uBrPic(\D)$.
The homomorphism $\mathcal{p}_M^1$ is given by \eqref{p1 formula}, i.e.
\[
\mathcal{p}_M^1: \Hom(\mathbb{Z}/2\mathbb{Z} ,\,  \mathbb{Z}/2\mathbb{Z} \times \mathbb{Z}/2\mathbb{Z} ) \to H^3(\mathbb{Z}/2\mathbb{Z},\, k^\times) :
P \mapsto (\omega \times \omega^{-1})\circ (P\times P \times P),
\]
which is clearly non-zero. This explains the difference between our parameterization of extensions and that of \cite[Theorem 1.3]{ENO}.
\end{remark}

\subsection{Central graded extensions}

Let $\B$ be a braided tensor category.

\begin{definition}
A {\em central $G$-extension} of $\B$ is a pair $(\C,\, \iota)$, where $\C$ is a $G$-extension
and  $\iota: \B \hookrightarrow \Z(\C)$
is  a braided tensor functor  whose composition with the forgetful functor $\Z(\C)\to \C$
coincides with the inclusion  $\B \hookrightarrow \C$. 
\end{definition}

\begin{definition}
Let $(\C,\iota:\B \to \Z(\C))$ and $(\tilde{\C},\tilde{\iota} :\B \to \Z(\tilde{\C}))$ be two central $G$-extensions of $\B$.
An {\em equivalence} between these extensions is an equivalence $F: \C\xrightarrow{\sim} \tilde{\C}$ of $G$-extensions
such that $\tilde\iota = \text{ind}(F)\circ \iota$, where $ \text{ind}(F):  \Z(\C)  \xrightarrow{\sim} \Z(\tilde{\C})$ is the
braided equivalence induced by $F$. 
\end{definition}

Central $G$-extensions of $\B$ form a $2$-groupoid $\uuExctr(G,\, \B)$.

Recall that a {\em $G$-crossed braided tensor category} is a $G$-graded tensor category $\C=\bigoplus_{x\in G}\, \C_x$
equipped with the action  of $G$ on $\C$, i.e. a monoidal functor $G\to \Aut_{\ot}(\C)$, such that $x(\C_y) = \C_{xyx^{-1}}$
and with a $G$-crossed braiding 
\begin{equation}
\label{crossed braiding}
c_{X,Y}: X\ot Y \to g(Y) \ot X,\qquad X\in \C_x,\, Y\in \C,
\end{equation}
satisfying certain natural axioms.
Note that the trivial component of the grading $\C_e$ is a braided tensor category.
Let $\uuExcrbr(G,\,\B)$ denote the $2$-groupoid of $G$-crossed braided fusion categories whose trivial component
is $\B$.  

The next Proposition  was essentially proved in \cite{JMPP}. 
It shows that the notions of a central extension and a $G$-crossed extension coincide. 

\begin{proposition}
\label{central = crossed braided}
There is a $2$-equivalence $\uuExcrbr(G,\,\B) \cong \uuExctr(G,\, \B)$.
\end{proposition}
\begin{proof}
We need to explain how a $G$-crossed braided  structure 
translates into a  central functor and vice versa.

Let $\C$ be a $G$-crossed braided tensor category with $\C_e=\B$. 
The restriction of the crossed braiding  \eqref{crossed braiding} provides  $X\in \C_e$ with the structure
of a central object of $\C$ and 
\[
\iota: \B = \C_e \to \Z(\C) : X \mapsto (X, \,c_{X,-}^{-1}) 
\]
is a braided tensor functor whose composition with teh forgetful functor $\Z(\C)\to \C$ is identity.   

In the opposite direction, a central $G$-extension  $(\C,\iota:\B \to \Z(\C))$ yields a natural isomorphism
\begin{equation}
\label{partial braiding}
c_{X,Z} : X \ot Z \xrightarrow{\sim}  Z\ot X,\qquad X\in \B,\, Z\in \C,
\end{equation}
satisfying the hexagon axioms.  This turns each $\C_x$ into an invertible $\B$-module category. Furthermore,
there are $\B$-module equivalences 
\begin{eqnarray*}
\C_y \to \Fun_\B(\C_x,\, \C_{xy}) &:& Y \mapsto -\, \ot Y, \\
\C_{xyx^{-1}} \to \Fun_\B(\C_x,\, \C_{xy}) &:& Y \mapsto Y\ot \, -,
\end{eqnarray*}
for all $x,y\in G$. Here the functor categories consist of right exact $\B$-module functors.
Combining  these equivalences for a fixed $x\in G$ we obtain a tensor autoequivalence $x\in \Aut_\ot(\C)$
such that $x(\C_y) = \C_{xyx^{-1}}$  and there is a natural isomorphism
\[
x(Y) \ot X \cong X \ot Y\qquad \text{ for all }  X\in \C_x,\, Y\in \C. 
\]
The latter is a crossed braiding on $\C$.  

These constructions are inverses of each other and are 
compatible with equivalences of $G$-crossed braided and central extensions, i.e. define a $2$-equivalence
between the corresponding $2$-groupoids.
\end{proof}

\begin{remark}
Let $\C$ be a central $G$-extension of $\B$. The braided tensor category $\C^G$ obtained from $\C$ as the equivariantization \cite[Section 4]{DGNO}
with respect to the canonical action of $G$ constructed in the proof of Proposition~\ref{central = crossed braided} coincides with 
the centralizer of the image of $\B$ in $\Z(\C)$.
\end{remark}

Recall that the Picard group of $\B$ was introduced in Section~\ref{Section Pic(B)}.
The following result is essentially a consequence of Proposition~\ref{central = crossed braided} and \cite[Theorem 7.12]{ENO}.
We include the proof for the sake of completeness.

\begin{theorem}
\label{secondary ENO}
There is an equivalence of $2$-groupoids $\uuExctr(G,\, \B) \cong \FUN(G,\, \uuPic(\B))$.
\end{theorem}
\begin{proof}
We adjust the proof of Theorem~\ref{main ENO} to the present setting (with $\D$-bimodule categories,
functors, and isomorphisms replaced by $\B$-module ones).

A central structure on a $G$-extension $\C=\bigoplus_{x\in G}\,\C_x$ of $\B$ consists of isomorphisms
\eqref{partial braiding} that turn every component $\C_x$ into an invertible $\B$-module category, i.e.
$\C_x$ belongs to  $\uuPic(\B)$. 
Equivalences \eqref{Mgh} coming from tensor products $\C_x \times \C_y \to \C_{xy}$ are $\B$-module equivalences 
in this case and natural isomorphisms \eqref{alphas} are isomorphisms of $\B$-module functors.

An equivalence of  central $G$-extensions of $\B$ yields $\B$-module equivalences \eqref{Fs} between homogeneous components
and isomorphisms \eqref{mus} of $\B$-module functors. An isomorphism between equivalences of central $G$-extensions
yields an isomorphism \eqref{etas} of $\B$-module functors.

Diagrams \eqref{cubes fghk}, \eqref{cubes for PW1}, and \eqref{cylinder eta} commute for the same reason as in the proof
of Theorem~\ref{main ENO}. 

Thus, \eqref{M:}  becomes a $2$-functor 
\begin{equation}
\label{M:ctr}
\uuExctr(G,\, \B) \to \FUN(G,\, \uuPic(\B)).
\end{equation}

Conversely, given a monoidal $2$-functor $G\to \uuPic(\B)$, consider its composition with the  inclusion
$ \uuPic(\B) \to \uuBrPic(\B)$. By Theorem~\ref{main ENO}, this yields a $G$-extension $\C$ of $\B$.  The $\B$-bimodule
structure of $\C$ comes from its left $\B$-module structure, so there is  a natural  isomorphism between
the functors of left and right tensor multiplication by $X\in \B$:
\begin{equation}
\label{partial braiding-}
c_{X,Z} : X \ot Z \xrightarrow{\sim}  Z\ot X,\qquad  Z\in \C.
\end{equation}
The hexagon for \eqref{partial braiding-} follows from the above definition of a $\B$-bimodule category structure of $\C$
and from the monoidal property of the $2$-functor  $\textbf{Mod}(\B) \to \textbf{Bimod}(\B)$. Thus, \eqref{partial braiding-}
is a central structure on the $G$-extension $\C$ of $\B$ and there is a $2$-functor
\begin{equation}
\label{L:ctr}
\FUN(G,\, \uuPic(\B)) \to \uuExctr(G,\, \B) 
\end{equation}
quasi-inverse to \eqref{M:ctr}.
\end{proof}

\begin{remark}
It follows from Theorem~\ref{secondary ENO} that central $G$-extensions of $\B$ can be described in terms of 
monoidal functors $G\to \uPic(\B)$ and group cohomology analogously to Corollary~\ref{summary monoidal ext}.
\end{remark}

\subsection{Braided graded extensions}

Let $\B$ be a braided tensor category and let $A$ be an Abelian group. 

\begin{definition}
A {\em braided $A$-extension} of $\B$ is a braided tensor category $\C$ that is an $A$-extension of $\B$.
\end{definition}

\begin{definition}
An equivalence between braided $A$-extensions $\C,\, \tilde{\C}$ of $\B$ is an equivalence of $A$-extensions
that is a braided functor. 
\end{definition}

\begin{example}
Braided $A$-extensions of $\Vect$ are precisely pointed braided fusion categories $\Vect_A^\omega$, where $\omega\in Z^3_{br} (G,\,k^\times)$
is an abelian $3$-cocycle. 
Equivalences between extensions $\Vect_A^{\omega}$ and $\Vect_A^{\tilde{\omega}}$ correspond to  abelian 
$2$-cochains $\mu \in C_{br}^2(A,\, k^\times)$ such that $d(\mu) = \omega/\tilde{\omega}$. 
Thus, the set of isomorphism classes $Ex_{br}(A,\,\Vect)$ of braided $A$-extensions of $\Vect$ is in bijection with $H^3_{br}(A,\,k^\times)
\cong \Quad(A,\, k^\times)$. 
\end{example}

Let $\FUNbr(A,\, \uuPicbr(\B))$ denote the $2$-groupoid of braided monoidal $2$-functors from $A$ to $\uuPicbr(\B)$.

\begin{theorem}
\label{main DN}
There is a $2$-equivalence $\uuExbr(A,\,\B) \xrightarrow{\sim} \FUNbr(A,\, \uuPicbr(\B))$.
\end{theorem}
\begin{proof}
Let $\C =\bigoplus_{x\in A}\, \C_x$ be a braided $A$-extension of $\B$.  Each homogeneous component $\C_y,\,y\in A,$
is an invertible $\B$-module category. The squared braiding
\[
\sigma_{X,Y}=c_{YX} c_{XY},\qquad X\in \B,\, Y\in \C_y,
\]
equips it with the structure of a braided $\B$-module category, i.e. $\C_y\in \uuPicbr(\B)$. The equivalences
$M_{x,y}: \C_x \bt_\B \C_y \xrightarrow{\sim} \C_{xy}$ from \eqref{Mgh} are braided module equivalences. Indeed,  
commutativity of the diagram~\eqref{bmf diagram} is a consequence of the identity
\[
c_{Y_1\ot Y_2, X} c_{X, Y_1\ot Y_2} = (c_{Y_1,X} \ot \id_{Y_2}) (\id_{Y_1} \ot c_{Y_2, X} c_{X, Y_2} ) (c_{X,Y_1} \ot \id_{Y_2}),
\qquad X,Y_1,Y_2\in \B,
\]
where we omit the associativity constraints in $\C$. 

As in the proof of Theorem~\ref{main ENO}, equivalences $M_{x,y}$ along with
the associativity $2$-cells $\alpha_{x,y,z}$  from \eqref{alphas} define
a monoidal structure  on the $2$-functor
\[
M(\C): A \to  \uuPicbr(\B) : x \mapsto \C_x. 
\]
Furthermore, the commutativity constraint of $\C$ gives rise to invertible $2$-cells 
\begin{equation}
\label{deltas xy}
\xymatrix{
\C_x\bt_\B \C_y  \ar[rrrr]^{B_{x,y}}_{}="a"  \ar[drr]_{M_{x,y}} &&&& \C_y\bt_\B \C_x \ar[dll]^{M_{y,x}}_{}="b" \\
&& \C_{xy} &&
\ar@{}"2,3";"a"^(.25){}="x"^(.95){}="y"  \ar@2{->}^{\delta_{x,y}}"x";"y"
}
\end{equation}
for all $x,y\in A$. 
The conditions \eqref{braided2fun1} and  \eqref{braided2fun2}  
in the definition of a braided monoidal $2$-functor  (i.e. commutativity of the octahedra \eqref{oct theta} and \eqref{oct psi})
follow from the hexagon axioms satisfied by the braiding of $\C$.

Thus, $M(\C):A \to \uuPicbr(\B)$ is a braided monoidal $2$-functor.

Suppose there is  another braided $A$-extension $\tilde{\C} =\bigoplus_{x\in A}\, \tilde{\C}_x$ of $\B$ and an equivalence 
of braided $A$-extensions $F: \C \to \tilde{\C}$.  The $\B$-module 
equivalences $F_x:\C_x\xrightarrow{\sim} \tilde{\C}_x$  between the homogeneous components are braided $\B$-module equivalences.
Indeed, commutativity of diagram~\eqref{bmf diagram} is a consequence of the braided property of $F$. 
We have  invertible $2$-cells \eqref{mus} satisfying \eqref{cubes fghk} as  in the proof of Theorem~\ref{main ENO}. 
The condition \eqref{BT cube} (i.e. commutativity of the prism \eqref{prism x|y})
follows from the braided property of $F$. 
Thus, we have a pseudo-natural isomorphism $M(F):M(\C) \to M(\tilde{\C})$ of braided monoidal $2$-functors. 

Given an isomorphism $\eta$ between a pair of equivalences  $F,\, F'$  of braided extensions $\C$ and $\C'$ 
one constructs an invertible modification $M(\eta)$ between $M(F)$ and $M(F')$ as in \eqref{etas}. 

Thus, we have a $2$-functor $M: \uuExbr(A,\,\B) \xrightarrow{\sim} \FUNbr(A,\, \uuPicbr(\B))$.
In the opposite direction, the $2$-fumctor \eqref{L:}  constructed in the proof of Theorem~\ref{main ENO}
carries a braided structure on a $2$-functor $\C: A \to \uuPicbr(\B))$ to a braiding on $L(\C)=\bigoplus_{x\in A}\, \C_x$.
Namely, $2$-cells \eqref{deltas xy} give rise to the braiding constraints  for $L(\C)$ while commuting  
octahedra  \eqref{oct theta},  \eqref{oct psi}  ensure that they satisfy give the hexagon identities.
\end{proof}

\begin{remark}
\label{Table 4 remark} 
The proof of Theorem~\ref{main DN} extends that of Theorem~\ref{main ENO}. So it extends  
the correspondences in Table~\ref{table-3} as follows:
\begin{table}[h!]
\label{braided pm table}
\centering
\begin{tabular}{ |p{7cm}|p{7cm}|  }
 \hline
 \hline
\centerline{ Braided tensor $A$-extensions $\C$ of $\B$} & \centerline{Monoidal $2$-functors $M: A \to \uuPicbr(\B)$}  \\
\hline
\hline
braiding constraints  $c_{X,Y}$  & braiding  $2$-cells $\delta_{x,y}$ \eqref{deltas xy} \\ 
commuting hexagon diagrams for $c$ & commuting octahedra \eqref{oct theta},  \eqref{oct psi}  \\
\hline
braided property  diagram for $F$  & commuting prism \eqref{prism x|y} \\
\hline
\end{tabular}
\caption{\label{table-4} A  correspondence between braided extensions  and  braided monoidal $2$-functors.}
\end{table}
\end{remark}

We can describe $A$-graded extensions of $\B$ in terms of braided group cohomology. 
It follows from constructions of Section~\ref{Sect Braided functors and cohomology}  that given a braided monoidal functor  $M:A \to \uBrPic(\D)$
there exist a canonical braided cohomology class $\mathcal{p}_M^0\in H^4_{br}(A,\,k^\times)$ and a canonical group homomorphism
$\mathcal{p}_M^1: H^1(A,\, \Inv(\B)) \to  H^3_{br}(A,\,k^\times)$. 

\begin{corollary}
A braided monoidal functor  $M:A \to \uPicbr(\B)$ gives rise to an $A$-graded extension of $\B$
if and only  if $ \mathcal{p}^0_M=0$ in $H^4_{br}(A,\,k^\times)$. 
Equivalence classes  of such extensions of  form a torsor over the cokernel of $\mathcal{p}^1_M$.
\end{corollary}
\begin{proof}
This follows from Theorem~\ref{main DN} and  Corollary~\ref{cohomology description braided}.
\end{proof}

\begin{example}
Let $\B$ be a non-degenerate braided fusion category. By Proposition~\ref{ModbrVect}, $\uuPicbr(\B) = \uuPicbr(\Vect)$ and so 
$\uuExbr(A,\,\B) \xrightarrow{\sim}  \uuExbr(A,\,\Vect)$. Thus, any braided $A$-extension of $\B$ is equivalent 
to one of the form
$\B\boxtimes \C(A,\,q)$  for some $q\in \Quad(A,\, k^\times) = H^3_{br}(A,\, k^\times)$.

Thus, the only braided fusion categories that admit interesting extensions are degenerate ones. 
\end{example}

\subsection{Symmetric graded extensions}

Let $\E$ be a symmetric tensor category and let $A$ be an Abelian group. 

\begin{definition}
A {\em symmetric $A$-extension} of $\E$ is a symmetric tensor category $\C$ that is an $A$-extension of $\E$.
\end{definition}

Equivalences of symmetric $A$-extensions are the same as for braided $A$-extensions. 
The $2$-groupoid  $\uuExsym(A,\,\E)$ of symmetric $A$-extensions of $\E$ is a $2$-subgroupoid of  $\uuExbr(A,\,\E)$. 

\begin{example}
Symmetric $A$-extensions of $\Vect$ are precisely pointed braided fusion categories $\Vect_A^\omega$, where $\omega\in Z^3_{sym} (G,\,k^\times)$
is a symmetric $3$-cocycle. 
Equivalences between extensions $\Vect_A^{\omega}$ and $\Vect_A^{\tilde{\omega}}$ correspond to  symmetric
$2$-cochains $\mu \in C_{sym}^2(A,\, k^\times)$ such that $d(\mu) = \omega/\tilde{\omega}$. 
Thus, the set of isomorphism classes $Ex_{sym}(A,\,\Vect)$ of braided $A$-extensions of $\Vect$ is in bijection with $H^3_{sym}(A,\,k^\times)
\cong \Hom(A,\, k^\times)_2 = \Hom(A,\, \mathbb{Z}_2)$.
\end{example}

Let $\C =\bigoplus_{x\in A}\, \C_x$ and $\C' =\bigoplus_{x\in A}\, \C'_x$ be symmetric  $A$-extensions of $\E$. 
Then $\C \bt_\E \C'$ is an $(A\times A)$-extension of $\E$. Define the tensor product of these
extensions to be the diagonal subcategory of $\C \bt_\E \C'$:
\begin{equation}
\label{odot product of exts}
\C \odot_\E \C' =  \bigoplus_{x\in A}\, \C_x \bt_\E \C'_x.
\end{equation}
This equips the $2$-groupoid  $\uuExsym(A,\,\E) $  of symmetric $A$-extensions of $\E$ with a  structure
of a symmetric $2$-categorical group.

Recall that the symmetric $2$-categorical group $\uuPicsym(\E)$ of symmetric $\E$-module categories is equivalent to 
$\uuPic(\E)$, the Picard group of $\E$.
Let $\FUNsym(A,\, \uuPic(\E))$ denote the $2$-groupoid of symmetric monoidal $2$-functors from $A$ to $\uuPic(\E)$.

\begin{theorem}
\label{main DN sym}
There is a symmetric monoidal $2$-equivalence $\uuExsym(A,\,\E) \xrightarrow{\sim} \FUNsym(A,\, \uuPic(\E))$.
\end{theorem}
\begin{proof}
We extend Theorem~\ref{main DN} to the symmetric setting. Observe that the homogeneous components of a symmetric
extension  $\C =\oplus_{x\in A}\, \C_x$ of $\E$ are necessarily symmetric $\E$-module categories. Commutativity of the cones \eqref{cone}
is equivalent to the squared braiding of $\C$ being identity, i.e. to the braided monoidal
$2$-functor $x\mapsto \C_x$ being symmetric.

The monoidal structure of this $2$-equivalence is established by comparing the tensor products, associativities, and braidings 
of $\uuExsym(A,\,\E)$ and  $\FUNsym(A,\, \uuPic(\E))$.
\end{proof}

\begin{corollary}
\label{exact sequence for Exsym groups}
There is an  exact sequence of group homomorphisms:
\begin{multline}
\label{exact sequence Exsym}
H^1(A,\,\Inv(\E)) \to H^3_{sym}(A,\, k^\times) \to \pi_0(\uuExsym(A,\,\E))  \to \pi_0(\Fun_{sym}(A,\, \uuPic(\E))) \to H^4_{sym}(A,\,k^\times).
\end{multline}
\end{corollary}
\begin{proof}
This follows from Theorem~\ref{exact sequence for Funsym groups}.
\end{proof}

\subsection{The group of symmetric extensions of a symmetric fusion category}

Let $G$ be a finite abelian group and let $t\in G$ be a central element such that $t^2=1$.
Let $A$ be a finite Abelian group.  In this Section we compute  the group 
\[
Ex_{sym}(A,\, \E) := \pi_0(\uuExsym(A,\,\E))
\]
of symmetric $A$-extensions of $\E=\Rep(G,\, t)$.

\begin{theorem} 
\label{Ex groups computed}
There are group isomorphisms
\begin{eqnarray}
\label{ExsymRepG}
Ex_{sym}(A,\,\Rep(G)) &\cong& H^2(G,\, \widehat{A}) \oplus H^1(A,\, \mathbb{Z}_2),\\ 
\label{ExsymRepGt}
Ex_{sym}(A,\, \Rep(G,\, t))  &\cong& H^2(G/\langle t \rangle ,\, \widehat{A})\quad \text{if $t\neq 1$}.
\end{eqnarray}
\end{theorem}
\begin{proof}
Let us first consider symmetric $A$-extensions  of $\Rep(G)$. They are of the form
$\Rep(\tilde{G},\, \tilde{t})$, where $\tilde{G}$ is a central extension
\begin{equation}
\label{extension m}
1\to  \widehat{A} \to \tilde{G}\xrightarrow{\pi} G \to 1 
\end{equation}
and $\tilde{t}$ is a central element of $\tilde{G}$ such that $\tilde{t}^2=1$ and $\pi(\tilde{t})=1$.
Thus,  every symmetric $A$-extension of $\Rep(G)$ is completely determined by the pair consisting of
a cohomology class in $H^2(G,\, \widehat{A})$ corresponding to the isomorphism class of the
group extension \eqref{extension m} and $\tilde{t} \in (\widehat{A})_2 = H^1(A,\, \mathbb{Z}_2)$.
The corresponding map $Ex_{sym}(A,\,\Rep(G)) \to H^2(G,\, \widehat{A}) \oplus H^1(A,\, \mathbb{Z}_2)$
is a group isomorphism. It is clearly injective. To see that it is surjective note that  the elements 
of $H^2(G,\, \widehat{A})$ form a subgroup $Ex_{Tan}(A,\, \Rep(G))$ of Tannakian
$A$-extensions of $\Rep(G)$ while the elements of $H^1(A,\, \mathbb{Z}_2)$ form the subgroup of split extensions. 

Now consider a symmetric $A$-extension $\C$ of $\Rep(G,\, t)$ with $t \neq 1$. It contains
a unique maximal Tannakian subcategory $\C_0$ of index $2$ which is a Tannakian $A$-extension of $\Rep(G/\langle t \rangle)$. 
We have a group homomorphism
\begin{equation}
\label{map f}
f: Ex_{sym}(A,\,\Rep(G,\, t))  \to    Ex_{Tan}(A,\,\Rep(G/\langle t \rangle))  = H^2(G/\langle t \rangle ,\, \widehat{A}) :\C \mapsto \C_0.
\end{equation}
We claim that $f$ has an inverse  given by the induction 
\begin{equation}
\label{map g}
g: Ex_{Tan}(A,\,\Rep(G/\langle t \rangle))  \to  Ex_{sym}(A,\,\Rep(G,\, t)) : \mathcal{T} \mapsto \Rep(G) \bt_{\Rep(G/\langle t \rangle)} \mathcal{T},
\end{equation}
where the tensor product of fusion categories over a symmetric fusion category is defined in \cite[Section 2.5]{DNO}.

Indeed, we have $f \circ g =\id$ since the maximal Tannakian subcategory of $\Rep(G) \bt_{\Rep(G/\langle t \rangle)} \mathcal{T}$
is $\Rep(G/\langle t \rangle) \bt_{\Rep(G/\langle t \rangle)} \mathcal{T} \cong \mathcal{T}$.  To check that $g \circ f =\id$ 
we observe that there is a surjective symmetric tensor functor $F: \C_0 \bt \Rep(G,\,t) \to \C$ given by embedding of factors. Since
the intersection of $\C_0$ and $\Rep(G,\,t)$ in $\C$ is $\Rep(G/\langle t \rangle)$ we see that $F$ factors through
$\C_0 \bt_{\Rep(G/\langle t \rangle)} \Rep(G,\,t)$. The latter fusion category has the same Frobenius-Perron dimension as
$\C$ so that $\C \cong \C_0 \bt_{\Rep(G/\langle t \rangle)} \Rep(G,\,t)$.
\end{proof}

\begin{corollary}
$Ex_{sym}(A,\, \sVect) = 0$. 
\end{corollary}

Below we  describe the exact sequence \eqref{exact sequence Exsym} computing the group of symmetric extensions.
This is meant to illustrate our obstruction theory and give an alternative proof of Theorem~\ref{Ex groups computed}.

\begin{proposition}
\label{Fun sym groups computed}
There are group isomorphisms
\begin{eqnarray}
\label{FunsymRepG}
\pi_0(\Fun_{sym}(A,\,\uPicsym(\Rep(G)))) &\cong& H^2(G,\, \widehat{A}), \\
\label{FunsymRepGt}
\pi_0(\Fun_{sym}(A,\, \uPicsym(\Rep(G,\, t))))  &\cong& \mbox{Ker}\left( H^2(G,\, \widehat{A}) \xrightarrow{\Xi_t} \Hom(G/\langle t \rangle,\, (\widehat{A})_2) \right),
\end{eqnarray}
where
\begin{equation}
\label{Xit}
\Xi_t: H^2(G,\, \widehat{A}) \to \Hom(G/\langle t \rangle,\, (\widehat{A})_2) : m \mapsto  m(t,\, -) m(-,\, t)^{-1}.
\end{equation}
\end{proposition}
\begin{proof}
Let us consider the Tannakian case first.  Let $G'=[G,\, G]$ and $\widehat{G}=\Hom(G,\, k^\times)$. 
We  have a homomorphism of short exact sequences
\begin{equation}
\label{sesiso}
\xymatrix{
0  \ar[r] & \Ext(G/G',\, \widehat{A}) \ar[r] \ar[d]^\alpha &
 H^2(G,\, \widehat{A}) \ar[r] \ar[d]^\beta &  
 \Hom(A,\, H^2(G,\,k^\times)) \ar@{=}[d]  \ar[r]&  0 \\
 0  \ar[r] & H^2_{sym}(A,\, \widehat{G}) \ar[r]  & \pi_0(\Fun_{sym}(A, \uPicsym(\Rep(G))))  \ar[r] &
\Hom(A,\, H^2(G,\,k^\times)) \ar[r] & 0.
}
\end{equation}
Here $\alpha$ is the duality isomorphism. The homomorphism $\beta$ is defined as follows.
An element $m \in H^2(G,\, \widehat{A})$ gives rise to a central group extension 
\begin{equation}
1\to  \widehat{A} \to \tilde{G}\xrightarrow{\pi} G \to 1. 
\end{equation}
The category $\Rep(\tilde{G})$ is a symmetric $A$-extension of $\Rep(G)$ and, therefore, yields a symmetric
monoidal functor $\alpha(m): A \to \uPic_{sym}(\Rep(G))$.
The first row of \eqref{sesiso} is split exact \cite[Theorem 2.1.19]{K} and the second
row comes from assigning to  a symmetric functor a group homomorphism. Hence, $\beta$ is an isomorphism.
This proves \eqref{FunsymRepG}. 

In the super-Tannakian case we have an exact sequence
\begin{equation}
\label{4term}
0  \to H^2_{sym}(A,\, \widehat{G}) \to Fun_{sym}(A, \uPicsym(\Rep(G,\,t)))  \to
\Hom(A,\, H^2(G,\,k^\times))  \xrightarrow{q^*} \Hom(A,\, (\widehat{G})_2),
\end{equation}
where $q^*$ is induced by the second canonical class of $\uuPicsym(\Rep(G,\,t))$,
\[
q: \Picsym(\Rep(G,\,t)) =H^2(G,\,k^\times) \to \Inv(\Rep(G,\,t))_2= (\widehat{G})_2 : \mu \mapsto \frac{\mu(t,\,-)}{\mu(-,\, t)}
\] 
see Theorem~\ref{pisu}. Combining \eqref{4term} with the commuting square
\begin{equation}
\xymatrix{
H^2(G,\, \widehat{A}) \ar[rr]^{\Xi_t} \ar[d] && \Hom(A,\, (\widehat{G})_2) \ar[d] \\
\Hom(A,\, H^2(G,\,k^\times))  \ar[rr]^{q^*} && \Hom(G,\, (\widehat{A})_2) 
}
\end{equation}
we obtain \eqref{FunsymRepGt}.
\end{proof}

Recall that isomorphism \eqref{H4sym}  identifies $H^4_{sym}(A,\,k^\times)$ with $\Hom(A_2,\, k^\times) = \widehat{(A_2)}$.
Combining this with the isomorphisms $H^2(\mathbb{Z}_2,\, \widehat{A}) \cong \widehat{A}/(\widehat{A})^2 \cong \widehat{(A_2)}$
we obtain
\begin{equation}
\label{H4 sym as H2}
H^4_{sym}(A,\,k^\times) \cong H^2(\mathbb{Z}_2,\, \widehat{A}). 
\end{equation}

\begin{proposition}
\label{obstruction hom}
The obstruction homomorphism $\pi_0(\Fun_{sym}(A,\, \uPic_{sym}(\E))) \to H^4_{sym}(A,\,k^\times)$ in \eqref{exact sequence Exsym} is 
given by 
\begin{equation}
\label{obstruction hom description}
Fun_{sym}(A,\, \uPic_{sym}(\E))  \cong \mbox{Ker}({\Xi_t})
\hookrightarrow H^2(G,\, \widehat{A})  \xrightarrow{\text{res}}  H^2(\langle t \rangle,\, \widehat{A}) \cong H^4_{sym}(A,\,k^\times),
\end{equation}
where the first isomorphism is  \eqref{FunsymRepGt}, the last one is \eqref{H4 sym as H2}, $\Xi_t$ is defined in \eqref{Xit},
and ${\text res}$ is the restriction map in cohomology.
\end{proposition}
\begin{proof}
For a symmetric monoidal functor $F: A \to \uPic_{sym}(\E))$ let  $a=a(F) \in  H^4_{sym}(A,\,k^\times)$ be the obstruction
to lifting it to a symmetric monoidal $2$-functor. The isomorphism \eqref{H4sym} expresses $a$ as an element of 
$\widehat{(A_2)}$ in terms of its components $a(x,x,x,x)$, $a(x,x|x)$, and  $a(x|x,x)$, $x\in A_2$. 
 We have $a(x,x|x)=a(x|x,x)=1$ while
the value of $a(x,x,x,x)$ is found as follows.
Let us view the $\E$-module equivalence $M_{x,x}:  F(x) \bt_\E F(x) \xrightarrow{\sim} \E$ 
coming from the monoidal functor structure of $F$ as an element of $\Inv(\E)= \widehat{G}$. 
Then $a(x,x,x,x)$  is equal to the value of the self-braiding of $M_{x,x}$, i.e. to the evaluation $M_{x,x}(t)$.
Note that the map
\[
A_2 \to \widehat{G} : x \mapsto M_{x,x}|_{\langle t \rangle}
\]
is a homomorphism, since
\[
M_{x,x} M_{y,y} = M_{xy, xy} M_{x,y} M_{y,x},\qquad x,y\in A_2,
\]
and $M_{x,y}M_{y,x}|_{\langle t \rangle}  =M_{xy}^2|_{\langle t \rangle} =1$.

By \eqref{H4sym}, $a$ is identified with the homomorphism
\begin{equation}
A_2 \to k^\times: x \mapsto M_{x,x}(t). 
\end{equation}
That this map coincides with the restriction map $ \mbox{Ker}({\Xi_t})
\hookrightarrow H^2(G,\, \widehat{A})  \xrightarrow{\text{Res}}  H^2(\langle t \rangle,\, \widehat{A})$
follows from commutativity of the following diagram
\begin{equation}
\xymatrix{
H^2(G,\, \widehat{A}) \ar[rr]^s \ar[d]^{\text{Res}}  && H^2_{sym}(G/[G,G],\, \widehat{A})  \ar[rr]^{\sim} && H^2_{sym}(A,\, \widehat{G}) \ar [d] \\
H^2(\langle t \rangle,\, \widehat{A}) \ar[rrrr]^{\sim} &&&& H^2_{sym}(A,\, \widehat{\langle t \rangle}),
}
\end{equation}
where $s$ denotes a splitting of the first row of \eqref{sesiso}.
\end{proof}

Thus, for $t=1$ the exact sequence \eqref{exact sequence Exsym}  gives rise to a split
short exact sequence
\begin{equation}
0\to \Hom(A,\, \mathbb{Z}_2)  \to 
Ex_{sym}(A,\,\Rep(G))  \to  H^2(G,\, \widehat{A}) \to 0,
\end{equation}
while for $t\neq 1$ it becomes
\begin{equation}
\label{Exsym for t not1}
\begin{split}
H^1(G,\, \widehat{A}) \xrightarrow{\text{Res}}  H^1(\langle t \rangle,\, \widehat{A}) \to 
& Ex_{sym}(A,\,\Rep(G,\,t))  \to  \\
&  \text{Ker}\left(  H^2(G,\, \widehat{A}) \xrightarrow{\Xi_t} H^1(G/\langle t \rangle,\,  H^1(\langle t \rangle,\, \widehat{A})) \right)
  \xrightarrow{\text{Res}}  H^2(\langle t \rangle,\, \widehat{A}).
 \end{split}
\end{equation}
The isomorphism $Ex_{sym}(A,\,\Rep(G,\,t))  \cong H^2(G/\langle t \rangle,\,\widehat{A})$
from \eqref {ExsymRepGt} can be recovered  by comparing the sequence
\eqref{Exsym for t not1} with the exact sequence coming from the Lyndon-Hochschild-Serre spectral sequence
\cite{DHW, Sah}:
\begin{equation}
\begin{split}
H^1(G,\, \widehat{A}) \xrightarrow{\text{Res}}  H^1(\langle t \rangle,\, \widehat{A}) \to  
& H^2(G/\langle t \rangle,\,\widehat{A}) \xrightarrow{\text{Inf}}  \\
& \text{Ker}\left(   H^2(G,\,\widehat{A})  \xrightarrow{\text{Res}}  H^2(\langle t \rangle,\, \widehat{A})  \right)
 \xrightarrow{\Xi_t} H^1(G/\langle t \rangle,\,  H^1(\langle t \rangle,\, \widehat{A})).
\end{split}
\end{equation}

\subsection{The  Pontryagin-Whitehead quadratic function and zesting}
\label{PW section}

Let $\B$ be a braided tensor category and let $A$ be an Abelian group. 
Fix a homomorphism 
 \begin{equation}
 \label{hom f}
f:A \to \Picbr(\B) : x\mapsto \C_x
\end{equation}
that extends to a braided monoidal $2$-functor  $A \to \uuPicbr(\B)$. That is, there is 
a braided extension
\[
\C =\bigoplus_{x\in A}\,\C_x.
\]
Let $c$ denote the braiding of $\C$.

Let $\Ex_{br}^f(A,\, \B)\subset \Ex_{br}(A,\, \B)$ be the $2$-subgroupoid of extensions corresponding to $f$.
Our goal here is to describe $\pi_0(\Ex_{br}^f(A,\, \B))$.

An extension of \eqref{hom f} to a braided monoidal functor $A \to \uPicbr(\B)$ amounts to choosing
$\B$-equivalences $\C_x \bt_\B \C_y \xrightarrow{\sim} \C_{xy},\, x,y\in A,$ satisfying coherence conditions. Any two such
equivalences differ by a  tensor multiplication by an invertible object $L_{x,y} \in \Z_{sym}(\B)$.
Hence, any extension $\tilde{\C} \in \Ex_{br}^f(A,\, \B)$ is equal to $\C$ as an abelian category  and has the tensor product 
\begin{equation}
\label{zesting ot}
X \tilde{\ot} Y = L_{x,y} \ot X \ot Y,\qquad X\in \C_x\,Y\in \C_y ,\, x,y, \in A.
\end{equation}
To get associativity and braiding constraints  of  $\tilde{\C}$ it is necessary to have isomorphisms
\begin{equation}
\label{xi kappa}
\xi_{x,y,z}: L_{xy,z}\ot L_{x,y} \xrightarrow{\sim} L_{x,yz}\ot L_{y,z} ,\quad 
\kappa_{x,y}: L_{x,y} \xrightarrow{\sim} L_{y,x} ,\qquad  x,y,z\in A,
\end{equation}
i.e. $L=\{L_{x,y}\}_{x,y\in A}$ must be a $2$-cocycle in $Z^2_{br}(A,\,\Inv(\Z_{sym}(\B)))$.
These constraints are given  by
\begin{equation}
\label{zesting a}
\begin{split}
(X \tilde{\ot}  Y) \tilde{\ot}  Z
= L_{xy,z}\ot L_{x,y} \ot X\ot Y\ot Z  \xrightarrow{\xi_{x,y,z}}  L_{x,yz}\ot L_{y,z} \ot X\ot Y\ot Z \\
\xrightarrow{c_{L_{y,z}, X}}  \quad  L_{x,yz}\ot X \ot L_{y,z} \ot Y\ot Z  = X \tilde{\ot}  (Y \tilde{\ot}  Z)
\end{split}
\end{equation}
and
\begin{equation}
\label{zesting c}
X \tilde{\ot}  Y = L_{x,y} \ot X\ot Y  \xrightarrow{\kappa_{x,y}}    L_{y,x} \ot X\ot Y \xrightarrow{c_{X,Y}} L_{y,x} \ot Y\ot X =  Y \tilde{\ot}  X
\end{equation}
for all objects $X\in \C_x,\,Y\in \C_y ,\,Z\in \C_z$, $x,\,y,\,z\in A$, where we omit the associativity constraints of $\C$.

This is a  braided version of the construction introduced in \cite[Section 8.7]{ENO}.
Such extensions were considered in \cite{BGHNPRW}  where they were called {\em zestings} of $\C$. Recently,
a more general construction was studied  in great detail in \cite{DGPRZ}.
In  Propositions~\ref{PW2 components defined}  and \ref{pi0 is fiber bundle} below we compute obstructions and give
a parameterization of such extensions. Our treatment of equivalence classes of zesting extensions and obstructions seems
to be different from that of \cite[Section 4]{DGPRZ}. 

By Proposition~\ref{PB is well defined} the Whitehead bracket 
\[
[-,\, -]: \Pic_{br}(\B) \times \Inv(\Z_{sym}(\B)) \to k^\times 
\]
satisfies $c_{M,Z} c_{Z,M} = [\M,\, Z] \id_{Z\ot M}$ for all $Z\in  \Inv(\Z_{sym}(\B))$ and $M \in \M$, where $\M \in \Pic_{br}(\B)$.

Define a group homomorphism
\begin{equation}
\label{PW1 map}
PW^1_\C : H^1_{br}(A,\, \Inv(\Z_{sym}(\B))) \to H^3_{br}(A,\, k^\times) : Z \mapsto Q_Z,
\end{equation}
where $Q_Z$ is identified with the quadratic from 
\begin{equation}
\label{QZx}
Q_Z(x)=  [\C_x,\, Z(x)]  \,, c_{Z(x),Z(x)}, \qquad x\in A.
\end{equation}

Define a quadratic function
\begin{equation}
\label{PW2 map}
PW^2_\C : H^2_{br}(A,\, \Inv(\Z_{sym}(\B))) \to H^4_{br}(A,\, k^\times),
\end{equation}
by setting  the components of $PW^2_\C(L)$ for a braided $2$-cocycle $L$ to be
\begin{align}
\label{PW' xyzw}
PW^2_\C(L)(x,y,z,w) &= c_{L_{x,y}, L_{z,w}}, \\
\label{PW' xy|z}
PW^2_\C(L)(x,y|z) &= 1, \\
\label{PW' x|yz}
PW^2_\C(L)(x|y,z) &= [\C_x,\,L_{y,z}], \qquad x,y,z,w \in A.
\end{align}

\begin{definition}
We will call \eqref{PW1 map}  and \eqref{PW2 map} the {\em first} and {\em second} {\em Pontryagin-Whitehead maps}, cf.\ \cite[Section 8.7]{ENO}.
\end{definition}

\begin{remark}
The maps $PW^1_\C$ and $PW^2_\C$ depend on the homomorphism $f: A \to \Picbr(\B): x\mapsto \C_x$.
\end{remark}

\begin{proposition}
\label{PW2 components defined}
Let $L$ be a $2$-cocycle in $Z^2_{br}(A,\,\Inv(\Z_{sym}(\B)))$. One can choose isomorphisms  \eqref{xi kappa} so that
the associativity and braiding isomorphisms \eqref{zesting a}, \eqref{zesting c} satisfy the pentagon and hexagon axioms
(i.e. give rise to a tensor category) if and only if $PW^2_\C(L)$ is  trivial in $H^4_{br}(A,\, k^\times)$.
\end{proposition}
\begin{proof} 
Define a cochain  $a\in C^3_{br}(A,\,k^\times)$ by
$a(x,y,z)= \xi_{x,y,z},\, a(x|y)= \kappa_{x,y}$ for all $x,y,z\in A$,
where $\xi$ and $\kappa$ are isomorphisms \eqref{xi kappa}.  

In the diagrams below we will omit the the tensor product sign and the associativity constraints of $\C$.
The pentagon  for the associativity constraint \eqref{zesting a} becomes  the diagram
\begin{equation}
\label{zest pent}
\scalebox{0.85}{
\xymatrix{
 & L_{xyz,w}L_{xy,z}L_{x,y}XYZW  \ar[dl]_{\xi_{x,y,z}} \ar[r]^{\xi_{xy,z,w}} &
L_{xy,zw}L_{z,w}L_{x,y}XYZW \ar[dr]^{c_{L_{z,w},L_{x,y}}}&   \\
  L_{xyz,w}L_{x,yz}L_{y,z}XYZW  \ar[d]_{c_{L_{y,z},X}} \ar[dr]^{\xi_{x,yz,w}}  &&& 
L_{xy,zw}L_{x,y}L_{z,w}XYZW  \ar[dl]_{\xi_{x,y,zw}} \ar[d]^{c_{L_{z,w},XY}}    \\
 L_{xyz,w}L_{x,yz}XL_{y,z}YZW \ar[d]_{\xi_{x,yz,w}} & L_{x,yzw}L_{yz,w}L_{y,z}XYZW \ar[dl]^{c_{L_{y,z},X}} \ar[r]^{\xi_{y,z,w}}&
 L_{x,yzw}L_{y,zw}L_{z,w}XYZW \ar[d]_{c_{L_{z,w},X}} \ar[dr]_{c_{L_{z,w},XY}} & L_{xy,zw}L_{x,y}XYL_{z,w}ZW  \ar[d]^{\xi_{x,y,zw}}\\ 
 L_{x,yzw}L_{yz,w}XL_{y,z}YZW \ar[d]_{c_{L_{yz,w},X}} &  &   L_{x,yzw}L_{y,zw}XL_{z,w}YZW  \ar[d]_{c_{L_{y,zw},X}}&  
 L_{x,yzw}L_{y,zw}XYL_{z,w}ZW  \ar[d]^{c_{L_{y,zw},X}}  \\
  L_{x,yzw}XL_{yz,w}L_{y,z}YZW  \ar[rr]^{\xi_{y,z,w} }&& L_{x,yzw}XL_{y,zw}L_{z,w}YZW  \ar[r]^{c_{L_{z,w},Y}} & L_{x,yzw}XL_{y,zw}YL_{z,w}ZW, 
}
}
\end{equation}
while the hexagons  are  the diagrams
\begin{equation}
\label{zest hex1}
\xymatrix{
L_{xy,z}L_{x,y}XYZ \ar[rr]^{c_{L_{x,y}XY,Z }} &&  L_{xy,z}ZL_{x,y}XY \ar[rr]^{\kappa_{xy,z}}  \ar[d]^{c^{-1}_{L_{x,y},Z }} && L_{z,xy}ZL_{x,y}XY  \ar[d]^{c^{-1}_{L_{x,y},Z }}  \\
L_{x,yz}L_{y,z}XYZ \ar[u]^{\xi_{x,y,z}^{-1}} \ar[drr]^{c_{XY,Z}} &&  L_{xy,z}L_{x,y}ZXY \ar[rr]^{\kappa_{xy,z}}   && L_{z,xy}L_{x,y}ZXY   \ar[d]^{\xi_{z,x,y}^{-1}}  \\
L_{x,yz}XL_{y,z}YZ \ar[u]^{c^{-1}_{L_{y,z},X }} \ar[d]_{c_{Y,Z}} &&  L_{x,yz}L_{y,z}ZXY  \ar[u]_{\xi_{x,y,z}^{-1}} \ar[d]^{\kappa_{y,z}}  && L_{zx,y}L_{z,x}ZXY  \\
L_{x,yz}XL_{y,z}ZY \ar[d]_{\kappa_{y,z}}   &&  L_{x,yz}L_{y,z}ZXY \ar[rr]^{\xi_{x,z,y}^{-1}}  && L_{xz,y}L_{x,z}ZXY  \ar[u]_{\kappa_{x,z}} \\
L_{x,zy}XL_{z,y}ZY \ar[rr]^{c^{-1}_{L_{z,y},X }}  &&  L_{x,yz}L_{y,z}XZY \ar[rr]^{\xi_{x,z,y}^{-1}}  \ar[u]_{c_{X,Z}}  && L_{xz,y}L_{x,z}XZY, \ar[u]_{c_{X,Z}} 
}
\end{equation}
and 
\begin{equation}
\label{zest hex2}
\xymatrix{
L_{x,yz}XL_{y,z} YZ  \ar[rr]^{\kappa_{x, yz}} && L_{yz,x}XL_{y,z} YZ  \ar[rr]^{c_{X,L_{y,z}YZ}}  \ar@/^1pc/[dd]^{c_{X,L_{y,z}}}&&  L_{yz,x}L_{y,z} YZX \ar[dd]^{\xi_{y,z,x}} \\
\\
L_{x,yz}L_{y,z} XYZ  \ar[uu]^{c_{L_{y,z},X}}  \ar[rr]^{\kappa_{x, yz}}  && 
L_{yz,x}L_{y,z} XYZ 
\ar[d]^{\xi_{y,z,x}}  \ar@/^1pc/[uu]^{c_{L_{y,z},X}}&&  
L_{y,zx}L_{z,x} YZX \ar[d]^{c_{L_{z,x},Y}} \\
L_{xy,z}L_{x,y} XYZ   \ar[u]^{\xi_{x,y,z}}  \ar[d]_{\kappa_{x, y}} && L_{y,zx}L_{z,x} XYZ \ar[urr]_{c_{X,YZ}}  &&  L_{y,zx}YL_{z,x} ZX \\
L_{xy,z}L_{y,x} XYZ  \ar[rr]^{\xi_{y,x,z}}   \ar[d]_{c_{X,Y}}  && L_{y,zx}L_{x,z} XYZ \ar[u]_{\kappa_{x,z}}  \ar[d]^{c_{X,Y}} &&  L_{y,zx}YL_{z,x} XZ\ar[u]_{c_{X,Z}}  \\
L_{xy,z}L_{y,x} YXZ  \ar[rr]^{\xi_{y,x,z}}  && L_{y,xz}L_{x,z} YXZ \ar[rr]^{c_{L_{x,z},Y}}  &&  L_{y,xz}YL_{x,z} XZ, \ar[u]_{\kappa_{x,z}}
}
\end{equation}
for all $x,y,z,w\in A$ and  $X\in \C_x,\, Y\in \C_y,\, Z\in \C_z,\, W\in \C_w$.

After cancelling internal polygones commuting by the functoriality of the tensor product of $\C$, naturality of $c$, and the Yang-Baxter equation,
we see that the clockwise compositions given by the perimeters of \eqref{zest pent}, \eqref{zest hex1}, and \eqref{zest hex2} are
\[
c_{L_{x,y},L_{z,w}}\, d(a)(x,y,z,w),\quad d(a)(x,y|z),\text{\quad  and  $c_{X,L_{y,z}} c_{L_{y,z},X} \,d(a)(x|y,z)$, \quad respectively.}  
\]
Comparing this with the definition of $PW^2_\C(L)$ we get the result.
\end{proof}

\begin{proposition}
\label{pi0 is fiber bundle}
There is a fibration $F \to \pi_0(\Ex_{br}^f(A,\, \B)) \to B$,
where the base $B$ is the set of zeroes of $PW^2_\C$ and the fiber $F$ is the cokernel of $PW^1_\C$.
\end{proposition}
\begin{proof}
The assertion about the base follows from Proposition~\ref{PW2 components defined}. 

Let $\C,\, \tilde\C$  be  $A$-extensions of $\B$ corresponding to the same braided monoidal functor $A\to \uPicbr(\B)$.
Then $\tilde\C = \C^{(\omega,\varsigma)}$ for some $(\omega,\varsigma)\in H^3(A,\,k^\times)$. 
An equivalence of extensions $\C^{(\omega,\varsigma)} \xrightarrow{\sim} \C$ is given on homogeneous component $\C_x,\,x\in A$, by   
$X \mapsto  Z(x)\ot X,\, X\in \C_x$ for $Z(x)\in \Inv(\Z_{sym}(\B))$. The tensor property of this equivalence means that $Z:A \to  \Inv(\Z_{sym}(\B))$
is a homomorphism, while its  braided property   translates to commutativity of the diagram
\begin{equation}
\label{torsor computation}
\xymatrix{
Z(x)\ot X \ot Z(y) \ot Y \ar[rrrrr]^{c_{Z(x)\ot X, Z(y) \ot Y }} \ar[d]_{c_{X,Z(y)}} &&&&& Z(y)\ot Y \ot Z(x) \ot X \ar[d]^{c_{Y,Z(x)}}  \\
Z(xy) \ot X\ot Y \ar[rrrrr]^{\varsigma(x,y)\, \,c_{X,Y}} &&&&& Z(xy) \ot Y\ot X, 
}
\end{equation}
for all $x,y\in A,\, X\in \C_x\,Y\in \C_y$. Here $c$ denotes the braiding of $\C$.

Comparing the compositions in \eqref{torsor computation} we see that
\[
 \varsigma(x,y) = c_{Z(x), Z(y)} \,  c_{Y, Z(x)}  c_{Z(x),Y},  \qquad \text{for all $Y\in \C_y$},
\]
and so the corresponding quadratic form is  $\varsigma(x,x) = [\C_x,\, Z(x)] \, c_{Z(x),Z(x)} =  Q_Z(x),\, x\in A$.
\end{proof}

\subsection{Quasi-trivial braided extensions}
\label{qt extensions}

Let $\B$ be a  braided tensor category. We saw in Example~\ref{uuPicbrI} that
the braided $2$-categorical group $\uuPicbr(\B)$ contains a full $2$-categorical subgroup $\uuPicbrI(\B)$ consisting of 
braided $\B$-module categories $\M$ such that $\M \cong \B$ as a $\B$-module category.  

\begin{definition}
\label{def qt ext}
Let $A$ be a finite group.
We say that a braided $A$-graded extension of $\B$ is {\em quasi-trivial} if it contains an invertible object in every homogeneous
component.
\end{definition}

Equivalently, an $A$-extension of $\B$ is quasi-trivial if  the corresponding homomorphism $A\to \Picbr(\B)$
factors through $\PicbrI(\B)$.  

\begin{remark}
A quasi-trivial extension is a special type of a braided {\em zesting} considered in \cite{DGPRZ}. Namely, it is a zesting
of $\C(A,\,1) \bt \B$.  
\end{remark}

Let  $\uuExbrqt(A,\,\B)$ denote the $2$-groupoid of quasi-trivial  braided $A$-extensions of $\B$. We have
an equivalence of $2$-groupoids
\[
\uuExbrqt(A,\,\B) \cong \FUNbr(A,\, \uuPicbrI(\B)).
\]

Since objects of  $\uuPicbrI(\B)$ are of the form $\B^\nu,\, \nu \in \Aut_\ot(\id_\B)$ (see Example~\ref{uuPicbrI}), 
any braided monoidal $2$-functor $A\to \uuPicbrI(\B)$ (and any extension in $\uuExbrqt(A,\,\B)$ comes from a group homomorphism
\[
f: A \to \Aut_\ot(\id_\B).
\]

\begin{example}
\label{any f}
Given $f$ as above,
there is a canonical quasi-trivial $A$-graded braided extension $\B(f)$ of~$\B$  such that
$\B(f)= \B \bt \Vect_A$ as a tensor category and its braiding is given by 
\[
c_{X\bt x,Y\bt y} = f(x)_Y\, c_{X,Y},\qquad X,Y\in \B,\, x,y\in A,
\]
where $x\in A$ denote the simple objects of  $\Vect_A$. 
\end{example}

Hence, 
\[
\uuExbrqt(A,\,\B) = \bigvee_{f\in \Hom(A,\,\Aut_\ot(\id_\B))} \, \mathbf{Ex}^f_\mathbf{br-qt}(A,\, \B),
\]
where $\mathbf{Ex}^f_\mathbf{br-qt}(A,\, \B)$ is the $2$-subgroupoid of quasi-trivial extensions corresponding to $f$. Furthermore,
$\mathbf{Ex}^{f_1}_\mathbf{br-qt}(A,\, \B) = \mathbf{Ex}^{f_2}_\mathbf{br-qt}(A,\, \B)$ if and only if $f_2 = f_1 \partial(Z)$ for some $Z\in \Inv(\B)$.

The Pontryagin-Whitehead maps \eqref{PW1 map} and \eqref{PW2 map} in this situation are given by
\begin{equation}
\label{PW1-qt}
PW^1_{\B(f)}(Z)(x) =   f(x)_{Z(x)}\, c_{Z(x), Z(x)} ,\qquad Z\in \Hom(A,\, \Inv(\Z_{sym}(\B))),
\end{equation}
and
\begin{align}
\label{obstruction for qt-qt}
PW^2_{\B(f)}(L) (x,y,z,w) &= c_{L_{x,y}, L_{z,w}}, \\
\label{PW' xy|z-qt}
PW^2_{\B(f)}(L) (x,y|z) &= 1, \\
\label{PW' x|yz-qt}
PW^2_{\B(f)}(L)  (x|y,z) &= f(x)_{L_{y,z}}, \qquad L\in H^2_{br}(A,\, \Inv(\Z_{sym}(\B))),
\end{align}
for all $x,y,z,w \in A$.

\begin{corollary}
\label{fiber-qt}
There is a fibration $F \to \pi_0(\mathbf{Ex}^f_\mathbf{br-qt}(A,\, \B)) \to B$,
where the base $B$ is the set of zeroes of  $PW^2_{\B(f)}$ and the fiber $F$ is the cokernel of $PW^1_{\B(f)}$.
\end{corollary}

Thus, quasi-trivial $A$-extensions of $\B$ are obtained by choosing a homomorphism $f: A \to \Aut_\ot(\id_\B)$,
deforming (``zesting'') the tensor product and constraints of  $\B(f)$ by means of  $L\in Z^2_{br}(A,\, \Inv(\Z_{sym}(\B)))$
such that $PW^2_{\B(f)}(L)  =0$
via  \eqref{zesting ot} - \eqref{zesting c}, and then twisting the result by means of a braided $3$-cocycle 
$(\omega,\varsigma)\in Z^3_{br}(A,\, k^\times)$.  Corollary~\ref{fiber-qt} gives a description of equivalence classes of such extensions.

\bibliographystyle{ams-alpha}

\end{document}